\documentclass[11pt, leqno]{amsart}
\usepackage{amsfonts,amssymb,amscd,amsmath,amsthm}
\usepackage{tikz}
\usetikzlibrary{decorations.fractals}
\usepackage[nice]{nicefrac}
\usepackage{lmodern}
\usepackage{mathrsfs}
\usepackage{qsymbols}
\usepackage{mathtools}
\mathtoolsset{showonlyrefs}
\usepackage{dsfont}
\usepackage{graphicx}
\usepackage{latexsym}
\usepackage{chngcntr}
\usepackage{paralist}
\usepackage[parfill]{parskip}
\usepackage{bm}
\usepackage{tikz-cd}
\usepackage{csquotes}
\usepackage[plainpages=false,pdfpagelabels,backref=page,citecolor=red]{hyperref}
\usepackage{xcolor}
\hypersetup{
	colorlinks,
	linkcolor={cyan!90!black},
	citecolor={magenta},
	urlcolor={green!40!black}}
\renewcommand\qedsymbol{$\blacksquare$}
\let\oldproofname=\proofname
\renewcommand{\proofname}{\rm\bf{\oldproofname}}
\usepackage[shortlabels]{enumitem}
\renewcommand{\labelenumi}{\normalfont{(\roman{enumi})}}
\newtheorem{theorem}{Theorem}[section]
\newtheorem{lemma}[theorem]{Lemma}
\newtheorem{proposition}[theorem]{Proposition}
\newtheorem{corollary}[theorem]{Corollary}
\newtheorem{assumption}[theorem]{Assumption}
\theoremstyle{definition}
\newtheorem{definition}[theorem]{Definition}
\newtheorem{remark}[theorem]{Remark}

\numberwithin{equation}{section}
\DeclareMathOperator{\loc}{loc}
\DeclareMathOperator{\e}{e}
\DeclareMathOperator{\Div}{div}
\DeclareMathOperator{\sgn}{sgn}
\DeclareMathOperator{\supp}{supp}
\DeclareMathOperator{\diag}{diag}
\DeclareMathOperator{\cc}{c}
\DeclareMathOperator{\Capp}{cap}
\newcommand{\C}{\mathbb{C}}
\newcommand{\N}{\mathbb{N}}
\newcommand{\F}{\mathbb{F}}
\newcommand{\K}{\mathbb{K}}
\newcommand{\Q}{\mathbb{Q}}
\newcommand{\R}{\mathbb{R}}
\newcommand{\T}{\mathbb{T}}
\newcommand{\Z}{\mathbb{Z}}
\newcommand{\smooth}[1][]{\rC_{\cc}^{\infty}(#1)}
\newcommand{\smoothD}[1][]{\rC_{D}^{\infty}(#1)}
\DeclareRobustCommand{\Cdot}{\dot{\rC}\protect{\vphantom{C}}}
\renewcommand{\H}{\mathrm{H}}
\DeclareRobustCommand{\Hdot}{\dot{\H}\protect{\vphantom{H}}}
\newcommand{\W}{\mathrm{W}}
\DeclareRobustCommand{\Wdot}{\dot{\W}\protect{\vphantom{W}}}
\renewcommand{\L}{\mathrm{L}}
\newcommand{\Simp}{\mathrm{Simp}}
\newcommand{\cD}{\mathcal{D}}
\newcommand{\cS}{\mathcal{S}}
\newcommand{\D}{\mathsf{D}}
\newcommand{\sR}{\mathsf{R}}
\newcommand{\sN}{\mathsf{N}}
\renewcommand{\d}{\mathrm{d}}
\newcommand{\diam}{\mathrm{diam}}
\newcommand{\rad}{\mathrm{r}}
\newcommand{\osc}{\mathrm{osc}}
\DeclareMathOperator{\esssup}{esssup}
\DeclareMathOperator{\essinf}{essinf}
\newcommand{\les}{\lesssim}
\newcommand{\ges}{\gtrsim}
\renewcommand{\S}{\mathrm{S}}
\let\ii\i
\renewcommand{\i}{\mathrm{i}}
\DeclareMathOperator{\im}{Im}
\DeclareMathOperator{\re}{Re}
\renewcommand{\t}{\boldsymbol{t}}
\newcommand{\z}{\mathbf{z}}
\newcommand{\1}{\boldsymbol{1}}
\newcommand{\rA}{\mathrm{A}}
\newcommand{\rB}{\mathrm{B}}
\newcommand{\rC}{\mathrm{C}}
\newcommand{\rD}{\mathrm{D}}
\newcommand{\rF}{\mathrm{F}}
\newcommand{\rG}{\mathrm{G}}
\newcommand{\rI}{\mathrm{I}}
\newcommand{\rJ}{\mathrm{J}}
\newcommand{\rN}{\mathrm{N}}
\newcommand{\rP}{\mathrm{P}}
\newcommand{\rR}{\mathrm{R}}
\newcommand{\rT}{\mathrm{T}}
\newcommand{\rU}{\mathrm{U}}
\newcommand{\rZ}{\mathrm{Z}}
\newcommand{\cA}{\mathcal{A}}
\newcommand{\cC}{\mathcal{C}}
\newcommand{\cE}{\mathcal{E}}
\newcommand{\cF}{\mathcal{F}}
\newcommand{\cH}{\mathcal{H}}
\newcommand{\cI}{\mathcal{I}}
\newcommand{\cJ}{\mathcal{J}}
\newcommand{\cL}{\mathcal{L}}
\newcommand{\cM}{\mathcal{M}}
\newcommand{\cN}{\mathcal{N}}
\newcommand{\cP}{\mathcal{P}}
\newcommand{\cR}{\mathcal{R}}
\newcommand{\cU}{\mathcal{U}}
\newcommand{\cW}{\mathcal{W}}
\newcommand{\cX}{\mathcal{X}}
\newcommand{\cY}{\mathcal{Y}}
\newcommand{\cZ}{\mathcal{Z}}
\newcommand{\up}{\ldots}
\newcommand{\sub}{\subseteq}
\newcommand{\LU}{\textcolor{black}{\mathrm{(LU)}}}
\newcommand{\Fat}{\textcolor{black}{\mathrm{(Fat)}}}
\newcommand{\Fatp}{\textcolor{black}{\mathrm{(Fat)}_p}}
\newcommand{\PN}{\textcolor{black}{\mathrm{(}\mathrm{P}_{N}\mathrm{)}}}
\newcommand{\PD}{\textcolor{black}{\mathrm{(}\mathrm{P}_{D}\mathrm{)}}}
\newcommand{\AssP}{\textcolor{black}{\mathrm{(}\mathrm{P} \mathrm{)}}}
\newcommand{\AssPp}{\textcolor{black}{\mathrm{(}\mathrm{P} \mathrm{)}_p}}
\newcommand{\Assrnot}{\textcolor{black}{\mathrm{(} r_0 \mathrm{)}}}
\newcommand{\Ndelta}{\textcolor{black}{\mathrm{(}\mathrm{N}_{\delta}\mathrm{)}}}
\newcommand{\ITC}{\textcolor{black}{\mathrm{(ITC}_{N_{\delta}} \mathrm{)}}}
\newcommand{\ICC}{\textcolor{black}{\mathrm{(ICC}_{N_{\delta}} \mathrm{)}}}
\newcommand{\AssE}{\textcolor{black}{\mathrm{(E)}}}
\newcommand{\AssExt}{\textcolor{black}{\mathrm{(} \mathcal{E} \mathrm{)}}}

\def\Xint#1{\mathchoice
	{\XXint\displaystyle\textstyle{#1}}%
	{\XXint\textstyle\scriptstyle{#1}}%
	{\XXint\scriptstyle\scriptscriptstyle{#1}}%
	{\XXint\scriptscriptstyle%
		\scriptscriptstyle{#1}}%
	\!\int}
\def\XXint#1#2#3{{\setbox0=\hbox{$#1{#2#3}{%
				\int}$ }
		\vcenter{\hbox{$#2#3$ }}\kern-.6\wd0}}
\def\barint{\,\Xint-} 
\title{Gaussian estimates vs.\@ elliptic regularity on open sets}
\author{Tim B\"ohnlein}
\author{Simone Ciani}
\author{Moritz Egert}
\address{Fachbereich Mathematik, Technische Universit\"at Darmstadt, Schlossgartenstr. 7, 64289 Darmstadt, Germany; Department of Mathematics, University of Bologna, Piazza di Porta San Donato 5, Bologna, Italy}
\email{boehnlein@mathematik.tu-darmstadt.de}
\email{simone.ciani3@unibo.it}
\email{egert@mathematik.tu-darmstadt.de}
\date{\today}
\dedicatory{}
\begin{document}
	\begin{abstract}
		Given an elliptic operator $L= - \Div (A \nabla \cdot)$ subject to mixed boundary conditions on an open subset of $\R^d$, we study the relation between Gaussian pointwise estimates for the kernel of the associated heat semigroup, H\"older continuity of $L$-harmonic functions and the growth of the Dirichlet energy. To this end, we generalize an equivalence theorem of Auscher and Tchamitchian to the case of mixed boundary conditions and to open sets far beyond Lipschitz domains. Yet, we prove the consistency of our abstract result by encompassing operators with real-valued coefficients and their small complex perturbations into one of the aforementioned equivalent properties. The resulting kernel bounds open the door for developing a harmonic analysis for the associated semigroups on rough open sets.
	\end{abstract}
	\maketitle

	\section{Introduction}
	Let $d \geq 2$, $O \sub \R^d$ be open, $A \in \L^{\infty}(O; \C^{d \times d})$ be elliptic and $L = - \Div (A \nabla \cdot)$ be realized as an m-accretive operator in $\L^2(O)$ subject to boundary conditions. To give a first idea of our results, let us consider the simplest case of pure Dirichlet boundary conditions, $u= 0$ on $\partial O$. We show that the following three properties are equivalent up to minimal changes in the parameter $\mu \in (0,1]$:
	\begin{enumerate}
		\item[$\rD(\mu)$] De Giorgi estimates for $L$ and $L^*$-harmonic functions: There exists $c > 0$ such that for all balls $B(x,R) \sub \R^d$ with $x \in \overline{O}$ and $R \in (0,1]$, all $u \in \H_0^1(O)$ that are $L$ or $L^*$-harmonic in $B(x,R)$ and each $r \in (0, R]$ the estimate
		\begin{equation*}
			\int_{O \cap B(x,r)} |\nabla u|^2 \, \d y \leq c \left( \frac{r}{R} \right)^{d-2+2 \mu} \int_{O \cap B(x,R)} |\nabla u|^2 \, \d y
		\end{equation*}
		is valid.

		\item[$\rG(\mu)$] H\"older continuity and pointwise Gaussian estimates for the kernel of the semigroup $(\e^{-t L})_{t \geq 0}$.

		\item[$\H(\mu)$] Local H\"older regularity of $L$ and $L^*$-harmonic functions with $\L^2$-norm control: There exists $c > 0$ such that for all balls as above and all $L$ or $L^*$-harmonic functions $u \in \H_0^1(O)$ in that ball the estimate
		\begin{equation*}
			\| u \|_{\L^{\infty}(O \cap B(x,\frac{r}{2}))} + r^{\mu} [u]^{(\mu)}_{O \cap B(x,\frac{r}{2})} \leq c r^{-\frac{d}{2}} \| u \|_{\L^2(O \cap B(x,r))}
		\end{equation*}
		holds true, where $[u]^{(\mu)}_{E} \coloneqq \sup_{y,z \in E, \, y \neq z} \frac{|u(y) - u(z)|}{|y-z|^{\mu}}$.
	\end{enumerate}
	Our main interest lies in property $\rG(\mu)$ and we consider the other two properties as a means of getting there. Property $\rG(\mu)$ opens the door for developing for the first time harmonic analysis, and in particular a theory of geometric Hardy spaces for $L$ as in \cite{Auscher_Russ, Duong_Yan-JAMS-H1-BMO-Duality, Bui_Duong_Ly_JFA, Song_Yan-JEE-Self-adjoint} on rough open sets far beyond Lipschitz domains. This link is explored in the work \cite{BB-Hardy_on_open_sets} of the first author with S.~Bechtel. Let us stress that positivity methods via the Beurling--Deny criterion as in \cite{Arendt_ter-Elst} are not suitable for getting $\rG(\mu)$  --- this approach can give the pointwise Gaussian bound, but it misses the H\"older continuity of the kernel that is key to treating the semigroup via methods from singular integral operators.

	\subsection*{The geometric framework}

	In order to prove this equivalence, we only assume that $O^c$ is locally $2$-fat. It is shown in \cite[Thm.~3.3]{P_D_implies_2-fat}, see also Proposition~\ref{Optimality of the geometric setup: Proposition: 2-fatness optimal} below, that this is equivalent to the following weak Poincar\'{e} inequality at the boundary: There are $c_0, r_0 > 0$, $c_1 \geq 1$ such that
	\begin{equation}
		\| u \|_{\L^2(O \cap B(x,r))} \leq c_0 r \| \nabla u \|_{\L^2(O \cap B(x, c_1 r))}  \tag*{\text{$\PD$}}
	\end{equation}
	for all $u \in \H_0^1(O)$ and balls $B(x,r) \sub \R^d$ with $x \in \partial O$ and $r \in (0, r_0]$. This condition seems indispensable, for example to control the growth of the Dirichlet energy in the derivation of $\rD(\mu)$ from $\H(\mu)$, making it reasonable to conjecture that our geometric assumptions are in the realm of the best possible.

	Now, we pass to the case of pure Neumann boundary conditions for $L$. The Poincar\'{e} inequality that we need (for the same reason as above) is that for the same balls as before and $u \in \H^1(O)$ we have
	\begin{equation*}
		\| u - (u)_{O \cap B(x,r)} \|_{\L^2(O \cap B(x,r))} \leq c_0 r \| \nabla u \|_{\L^2(O \cap B(x, c_1 r))}, \tag*{\text{$\PN$}}
	\end{equation*}
	where $(u)_E \coloneqq \frac{1}{|E|} \int_E u \, \d x$. However, in the Neumann case we need different geometric properties of $O$, because, roughly speaking, the extension of functions in the form domain $\H^1(O)$ by $0$ is not meaningful anymore. The most general geometric framework that we are aware of and that satisfies all our needs is the class of locally uniform domains with a positive radius condition. This should be thought of as a quantitative connectedness condition of $O$, see \cite{Jones_Uniform_Domains, Extension_Operator}.

	Our methods are flexible enough to treat mixed Dirichlet/Neumann boundary conditions with hardly any additional effort. Let $D \sub \partial O$ be closed and $N \coloneqq \partial O \setminus D$. We call $D$ the Dirichlet part and $N$ the Neumann part of the boundary. The boundary conditions are encoded in the form domain $V \coloneqq \H_D^1(O)$, which is defined as the closure in $\H^1(O)$ of smooth functions that vanish near $D$ (Definition~\ref{Operator theoretic setting and relevant function spaces: Definition: Mixed boundary condition Sobolev space}). Then the properties $\rD(\mu),\rG(\mu)$ and $\mathrm{H}(\mu)$ are understood in this context (see Section~\ref{Section: D(mu), G(mu), H(mu) for MBC}).

	Remarkably, and in contrast to several earlier results \cite{ter-Elst_Rehberg, Stampacchia_MBC_Hard_to_Check}, we can work without an interface condition between $D$ and $N$, which is where typically the main difficulties lie. Our geometric assumptions ``interpolate" between the two extremal cases. Namely, we need:
	\begin{enumerate}
		\item[$\Fat$] $O^c$ is locally $2$-fat away from $N$,

		\item[$\LU$] $O$ is locally uniform near $N$,
	\end{enumerate}
	see Section~\ref{Subsection: Assumptions Fat and LU} for precise definitions. Within this setting, we can express both $\PD$ and $\PN$ by the single property
	\begin{equation*}
		\tag*{\text{$\AssP$}} \qquad \quad\| u - \1_{[x \in N]} \cdot (u)_{O \cap B(x,r)} \|_{\L^2(O \cap B(x,r))} \leq c_0 r \| \nabla u \|_{\L^2(O \cap B(x, c_1 r))}, \qquad \qquad \qquad
	\end{equation*}
	for all balls as before and $u \in \H^1_D(O)$. Here, $\1_{[x \in N]} = \1_N(x)$ denotes the indicator function of the set $N$.

	We recall the basic $\L^2$ operator theory for $L$ in Section~\ref{Section: Operator theoretic setting and relevant function spaces} and introduce all relevant geometric concepts in Section~\ref{Section: The geometric setup}. The rest of the paper is divided into three parts.

	\textbf{The equivalence theorem (Sections~\ref{Section: D(mu), G(mu), H(mu) for MBC}~to~\ref{Section: (3) to (1)}).} We prove the equivalence of $\rD(\mu)$, $\rG(\mu)$ and $\H(\mu)$ for mixed boundary conditions. This is the main result and illustrated in Figure~\ref{fig: Summary Equivalence Theorem}.

	\begin{theorem} \label{Theorem: Main result of the paper}
		Let $d \geq 2$, $O \sub \R^d$ be locally uniform near $N$, $O^c$ locally $2$-fat away from $N$ and $\mu_0 \in (0,1]$. Then the following assertions are equivalent.
		\begin{enumerate}
			\item[$\mathrm{(i)}$] $L$ and $L^*$ have property $\rD(\mu)$ for all $\mu \in (0, \mu_0)$.
			\item[$\mathrm{(ii)}$] $L$ has property $\rG(\mu)$ for all $\mu \in (0, \mu_0)$.
			\item[$\mathrm{(iii)}$] $L$ and $L^*$ have property $\mathrm{H}(\mu)$ for all $\mu \in (0, \mu_0)$.
		\end{enumerate}
	\end{theorem}

	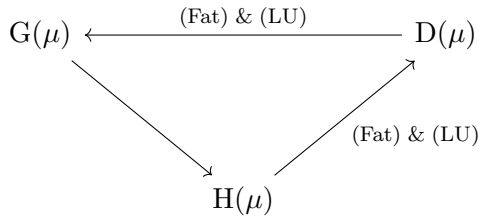
\begin{figure}[h!]
		\begin{center}
			\begin{tikzcd}[row sep=4em, column sep=4em]
				\rG(\mu)   \arrow[dr] &         & \arrow[ll,"\Fat \; \& \;  \LU",swap] \rD(\mu) \\
				& \H(\mu) \arrow[ur,"\Fat \; \& \; \LU",swap] &
			\end{tikzcd}
		\end{center}
		\caption{The geometric assumptions needed in Theorem \ref{Theorem: Main result of the paper}. Except for the implication $\mathrm{H}(\mu) \Longrightarrow \rD(\mu)$ the value of $\mu$ changes. Any property for $\mu = \mu_0$ implies any other property for all $\mu < \mu_0$.}
		\label{fig: Summary Equivalence Theorem}
	\end{figure}

	Theorem~\ref{Theorem: Main result of the paper} originates from results of Auscher and Tchamitchian on $O = \R^d$, see \cite[Chap.~4]{Auscher_Heat-Kernel} and \cite[Chap.~1]{Auscher_Tchamitchian_Kato}. They have been extended to special Lipschitz domains with pure Dirichlet or Neumann boundary conditions using localization techniques \cite{Auscher_Tchamitchian_Domains}. We highlight that we do not only recover the known statements of \cite{Auscher_Tchamitchian_Domains}, but we generalize the geometric setting to a far larger class of admissible geometries.

	The proof of Theorem~\ref{Theorem: Main result of the paper} is inspired by the monograph \cite{Auscher_Tchamitchian_Kato} and the work of ter Elst and Rehberg \cite{ter-Elst_Rehberg}, who studied property $\rG(\mu)$ for real-valued coefficients, when $O$ has a weakly Lipschitz boundary around the Neumann part, satisfies an exterior thickness condition around the Dirichlet part and an interface condition in between. Concerning the geometric setup, the present work also extends their result, see the discussion in Section~\ref{Subsection: Comparison of the geometric setup}.

	It remains the question whether some operator $L$ that satisfies one or equivalently all three of these properties exists. This leads us to the second part.

	\textbf{Real-valued coefficients (Sections~\ref{Section: Real-valued A} and~\ref{Section: Property D(mu)}).} Here, we study real-valued $A$ and show:

	\begin{theorem}  \label{Theorem: Main result 2}
		Let $d \geq 2$, $O \sub \R^d$ be locally uniform near $N$, $O^c$ locally $2$-fat away from $N$ and let $A$ be real-valued. Then $L$ has property $\H(\mu)$ for some $\mu \in (0,1]$.
	\end{theorem}

	To this end, we combine De Giorgi's classical approach \cite{DeGio57, Giaquinta_Martinazzi} with a method of DiBene-detto for nonlinear operators of $p$-growth \cite{DiBe89, DiBenedetto}. The simple underlying idea is that a Poincar\'e inequality, with a lower exponent than the $2$-growth of the operator, implies an estimate for the growth of the level sets as in the case of the isoperimetric inequality. Here, we use the deep fact that $p$-fatness is an open ended condition in $p$ \cite{Lewis_P-fat_Open_ended, Mikkonen}. Furthermore, we prove local boundedness and H\"older continuity up to the boundary for functions lying in a wider function class than the mere solutions to the equation (Definition~\ref{Real valued A: Definition: De Giorgi class}). In fact, we can also be slightly more general on the geometric side by replacing~\hyperref[The geometric setup: Definition (Fat)]{$\Fat$} and~\hyperref[The geometric setup: Definition (LU)]{$\LU$} by a $p$-adapted version of $\AssP$ for some $p \in (1,2)$ joint with the embedding $\H_D^1(O) \sub \L^{2^*}(O)$ for $d \geq 3$, where $2^* = \nicefrac{2 d}{(d-2)}$, or an interpolation type inequality when $d=2$.

	Theorem~\ref{Theorem: Main result 2} is already known in the pure Dirichlet case provided that $O$ satisfies an exterior thickness condition \cite[Chap.~II, App.~C \& D]{Kinderlehrer_Stampacchia}. In several papers \cite{Reference_Intro_Holder_Mixed_1, Reference_Intro_Holder_Mixed_2, Reference_Intro_Holder_Mixed_3, Robert_Joachim_Fractional_Powers, Stampacchia_MBC_Hard_to_Check} the H\"older regularity of solutions to non-homogeneous elliptic problems was studied for the case of mixed boundary conditions. However, all aforementioned papers use either Lipschitz coordinate charts around the Neumann part, stronger assumptions on $D$ and an interface condition between $D$ and $N$ or their geometric setup is almost impossible to check \cite{Stampacchia_MBC_Hard_to_Check}. We refer to \cite{ter-Elst_Rehberg} for further references and applications.

	\textbf{Further special cases (Sections~\ref{Section: Property D(mu)} and \ref{Section: Property G(mu) for d=2}).}  In the setup of Theorem~\ref{Theorem: Main result of the paper} we follow an argument of \cite{Auscher_Heat-Kernel} to prove that $\rD(\mu)$, and thus $\rG(\mu)$ and $\H(\mu)$, are stable under small complex perturbations of the coefficients, too. Property $\rG(\mu)$ for small perturbations of real-valued $A$ has been obtained in \cite{P-ellipticity-Moritz-Counterpart} via a different method. Moreover, when $d = 2$, only a slightly stronger geometric assumption on $D$ --- the so-called $(d-1)$-set property --- is needed to obtain for every elliptic operator $L$ some $\mu \in (0,1]$ such that $\rG(\mu)$ and hence also $\rD(\mu)$ and $\H(\mu)$ hold.

	\textbf{Acknowledgment} We would like to thank Sebastian Bechtel for his persistent interest in this project and for providing a proof of Theorem~\ref{Property G(mu) for d=2: Theorem} and Juha Lehrb\"ack for pointing out reference \cite{Mikkonen}. We are grateful to the anonymous referees for their careful reading and numerous suggestions for improving this paper. We acknowledge the Departments of Mathematics of TU Darmstadt and Universit\`a di Bologna Alma Mater for their support as well as PNR fundings 2021-2027 of MIUR.


	\section*{Notation} \label{Notation}

	Throughout, we use the following notation and abbreviations to simplify the exposition.

	{\small{
			\begin{itemize}
				\item For $\mathrm{X}, \mathrm{Y} \geq 0$, we write $\mathrm{X} \les \mathrm{Y}$, if there is some $c > 0$, which is independent of the parameters at stake, such that $\mathrm{X} \leq c \mathrm{Y}$. To emphasize that $c = c(a)$, we write $\mathrm{X} \les_a \mathrm{Y}$.

				\item Given $p \in [1, \infty]$, we write $p'$ for its H\"older conjugate satisfying $1=\nicefrac{1}{p}+\nicefrac{1}{p'}$. We denote by $p_* \coloneqq \nicefrac{dp}{(d+p)}$ the lower Sobolev conjugate of $p$ and, provided $p \in [1,d)$, we let $p^* \coloneqq \nicefrac{dp}{(d-p)}$ be the upper Sobolev conjugate of $p$.

				\item For $x \in \R^d$ and $r > 0$ we denote by $B(x,r)$ the open ball centered in $x$ with radius $r$. Given $E, F \sub \R^d$ we write $\d(E, F) \coloneqq \mathrm{dist}(E,F)$ for their Euclidean distance and abbreviate $\d_E(x) \coloneqq \d(x, E) \coloneqq \d( \{ x \}, E)$. The ball relative to $E$ is denoted by $E(x,r) \coloneqq E \cap B(x,r)$ and we write $\partial E (x,r) = \partial (E(x,r))$ for its boundary.

				\item Given $E \sub \R^d$ and $\delta > 0$, we set $E_{\delta} \coloneqq \{ x \in \R^d \colon \d_E(x) < \delta \}$.

				\item Given $E \sub \R^d$ and a function $u \colon E \to \C$, we denote by $u_0$ its $0$-extension to $\R^d$.

				\item We abbreviate norms in $\L^p(O)$ by $\| \cdot \|_p$ and norms in $\W^{1,p}(O)$ by $\| \cdot \|_{1,p}$. All integrals are taken with respect to the Lebesgue measure. Functions are $\C$-valued unless stated otherwise.

				\item We abbreviate $u_t \coloneqq \e^{-t L} u$ for all $t \geq 0$ and $u \in \L^2(O)$.
	\end{itemize}}}


	\section{Operator theoretic setting and relevant function spaces} \label{Section: Operator theoretic setting and relevant function spaces}

	Throughout this work, we let $O \sub \R^d$, $d \geq 2$, be open. We denote by $D \sub \partial O$ a closed set, which we call the \textbf{Dirichlet part of the boundary}, and we denote its complement by $N \coloneqq \partial O \setminus D$, the \textbf{Neumann part of the boundary}.

	In this section we recall the basic theory for $L \coloneqq - \Div (A \nabla \cdot)$ viewed as an m-accretive operator in $\L^2(O)$.

	\begin{definition}  \label{Operator theoretic setting and relevant function spaces: Definition: Mixed boundary condition Sobolev space}
		Let $\smoothD[\R^d] \coloneqq \smooth[\R^d \setminus D]$. We define the space of smooth functions in $O$ that vanish near $D$ as
		\begin{equation*}
			\rC_D^{\infty}(O) \coloneqq \{ \varphi|_O : \varphi \in \smoothD[\R^d] \}.
		\end{equation*}
		For $p \in [1, \infty)$ we let
		\begin{equation*}
			\W_D^{1,p}(O) \coloneqq \overline{\rC_D^{\infty}(O)}^{\| \cdot \|_{1,p}} \quad \& \quad \H_D^1(O) \coloneqq \W^{1,2}_D(O).
		\end{equation*}
	\end{definition}

    \begin{remark}
		In general, the space $\H_{\varnothing}^1(O)$ is a proper subspace of $\H^1(O)$ and models so-called \textbf{good Neumann boundary conditions}. However, if we have a bounded Sobolev extension operator $\cE \colon \H^1(O) \to \H^1(\R^d)$ at our disposal, then equality holds, because $\smooth[\R^d]$ is dense in $\H^1(\R^d)$.
	\end{remark}

	These Sobolev spaces with partial Dirichlet condition are closed under truncation in the following sense.

	\begin{lemma}[{\cite[Lem.~2.2~(a)]{ter-Elst_Rehberg}}] \label{Real valued A: Lemma: Auxiliary results}
		Let $p \in (1, \infty)$, $u \in \W^{1,p}_D(O; \R)$ and $k \geq 0$. Then $(u-k)^{+}$ and $u \wedge k$ are contained in $\W^{1,p}_D(O; \R)$.
	\end{lemma}

	We introduce the operator $L$ subject to mixed boundary conditions.

	\begin{assumption} \label{Assumption: Operator theoretic setting}
		We assume that $A \colon O \to \C^{d \times d}$ is elliptic in the sense that
		\begin{equation*}
			\exists \, \lambda > 0 \; \forall \, u \in \H_D^1(O) \colon \quad \re \int_O A \nabla u \cdot \overline{\nabla u} \geq \lambda \| \nabla u \|_2^2 \quad \& \quad   \Lambda \coloneqq \| A \|_{\infty} < \infty.
		\end{equation*}
	\end{assumption}

	The divergence form operator $L$ is realized in $\L^2(O)$ via the closed and densely defined sectorial form
	\begin{equation*}
		a \colon \H_D^1(O) \times \H_D^1(O) \to \C, \quad a(u,v) \coloneqq \int_O A \nabla u \cdot \overline{\nabla v}.
	\end{equation*}
	Its domain is given by
	\begin{equation*}
		\D(L) = \left\{ u \in \H_D^1(O) \colon \exists \, Lu \in \L^2(O) \; \forall \, v \in \H_D^1(O) \colon (Lu \, | \, v)_2 = a(u,v) \right\}.
	\end{equation*}

	Kato's form method \cite[Chap.~6]{Kato} yields that $L$ is m-accretive and thus generates an analytic $\rC_0$-contraction semigroup $(\e^{-t L})_{t \geq 0}$ in $\L^2(O)$. The semigroup and its gradient also satisfy so-called $\boldsymbol{\L^2}$ \textbf{off-diagonal estimates}, which the reader should think of as an $\L^2$-averaged form of kernel bounds.

	\begin{proposition}[{\cite[Prop.~3.2]{Bechtel_Lp}}, {\cite[Prop.~4.2]{Egert: Lp-Riesz-Trafo}}] \label{Operator theoretic setting and relevant function spaces: Proposition: L2 ODE}
		There are $C,c > 0$ depending only on $\lambda, \Lambda$ such that
		\begin{equation*}
			\| \1_F \e^{- t L} (\1_E u) \|_2 + \| \1_F \sqrt{t} \nabla \e^{- t L} (\1_E u) \|_2 \leq C\e^{-c \frac{\d(E,F)^2}{t}} \| \1_E u \|_2,
		\end{equation*}
		for all measurable sets $E, F \sub O$, $t > 0$ and $u \in \L^2(O)$.
	\end{proposition}


	\section{The geometric setup} \label{Section: The geometric setup}

	\subsection{Assumptions (Fat) and (LU)} \label{Subsection: Assumptions Fat and LU}

	We introduce the geometric setup and explain its consequences.

	\begin{definition}
		Let $p \in (1,d]$, $U \sub \R^d$ be open and $K \sub U$ be compact. The \textbf{$\boldsymbol{p}$-capacity of the condenser $\boldsymbol{(K, U)}$} is defined as
		\begin{equation*}
			\Capp_p(K; U) \coloneqq \inf \left\{ \| \nabla u \|_{\L^p(U)}^p \colon  u \in \smooth[U; \R] \; \text{with} \; u \geq 1 \; \text{pointwise on} \; K \right\}.
		\end{equation*}
	\end{definition}

	\begin{definition}
		Let $C \sub \R^d$ be closed, $\widehat{C} \sub C$ and $p \in (1,d]$. We call \textbf{$\boldsymbol{C}$ locally $\boldsymbol{p}$-fat in $\boldsymbol{\widehat{C}}$} if:
		\begin{equation}
			\exists \, c > 0 \; \forall \, x \in \widehat{C}, r \in (0, 1] \colon \quad \Capp_p(\overline{B(x,r)} \cap C; B(x,2r)) \geq c r^{d-p}.  \label{eq: Definition: p-fatness}
		\end{equation}
		If $\widehat{C} = C$, then we say that \textbf{$\boldsymbol{C}$ is locally $\boldsymbol{p}$-fat}.
	\end{definition}

	Here, we refer to Remark~\ref{Appendix: Remark: p-fatness} for elementary properties related to this definition and to \cite{Hedberg} for general background on capacities.

	\begin{definition}[Assumption $(\mathrm{Fat})_p$] \label{The geometric setup: Definition (Fat)}
		Let $p \in (1, d]$. We say that \textbf{$\boldsymbol{O^c}$ is locally $\boldsymbol{p}$-fat away from $\boldsymbol{N}$}, if there is some $\delta > 0$ such that:
		\begin{enumerate}
			\item $D$ is locally $p$-fat in $D \cap N_{\delta}$,

			\item $O^c$ is locally $p$-fat in $D$.
		\end{enumerate}
		For $p=2$ we write~\hyperref[The geometric setup: Definition (Fat)]{$\Fat$} instead of $\Fat_2$ to mean that $O^c$ is locally $2$-fat away from $N$.
	\end{definition}

	\begin{figure}
		\centering
		\begin{tikzpicture}[scale=1.5]
			\draw[black, domain=-2.5:0.7, samples=100] plot (\x, {\x*\x*\x/15 - \x*\x/15 + 0.5}); 
			\draw[black, line width = 2pt, domain=0.7:2, samples=100] plot (\x, {\x*\x*\x/15 - \x*\x/15 + 0.5}); 

			\draw[black, domain=0.7:2, samples=100] plot (\x, {\x*\x*\x/15 - \x*\x/15 + 1.7});
			\draw[black, domain=0.7:2, samples=100] plot (\x, {\x*\x*\x/15 - \x*\x/15 -0.7});
			\draw (0.7, 343/15000 - 0.49/15 + 1.7) arc (90:270:1.2cm);

			\draw (0.7, 343/15000 - 0.49/15 -0.7) -- (0.7, 343/15000 - 0.49/15 + 0.5);

			\draw[black, line width = 1.3pt, color=blue!30, domain=-0.5:0.5, samples=100] plot (\x, {\x*\x*\x/15 - \x*\x/15 + 0.5});

			\draw (0,0.5) circle (0.5cm);
			\draw (-2, -0.29) circle (0.5cm);

			\node[inner sep=1.5pt, fill=black, circle] at (0,0.5) (x) {};
			\node[inner sep=1.5pt, fill=black, circle] at (-2,-0.29) (y) {};
			\node[above] at (-0.7,-0.8) {$O$};
			\node[above] at (0,1.5) {$O^c$};
			\node[above] at (1,0.5) {$N$};
			\node[above] at (-1.4,0.2) {$D$};

			\node[above] at (x.north) {$x$};
			\node[below] at (y.south) {$y$};

			\draw (0.7, 343/15000 - 0.49/15 -0.7) -- (0.7, 343/15000 - 0.49/15 + 0.5)  node[midway, right] {$\delta$};


			\coordinate (center) at (-2, -0.29);
			\def\radius{0.5}

			\clip (-3.5, -0.5) -- plot[domain=-3.5:0.5, samples=100] (\x, {(\x)^3/15 - (\x)^2/15 + 0.5}) -- (-2, 0.5) -- cycle;
			\clip (center) circle (\radius);

			\fill[blue!30] (-3.5, 1) rectangle (0.5, -1.29);

			\node[inner sep=1.5pt, fill=black, circle] at (-2,-0.29) (y) {};
		\end{tikzpicture}
		\vspace{-90pt}
		\caption{The parts of the balls around $x$ and $y$ for which we require a lower bound on the $2$-capacity are in blue. The full picture of assumption~\hyperref[The geometric setup: Definition (Fat)]{$\Fat$} is obtained by letting $x$, $y$ and the size of the balls vary.}
		\label{fig. The geometric setup: Assumption (Fat)}
	\end{figure}
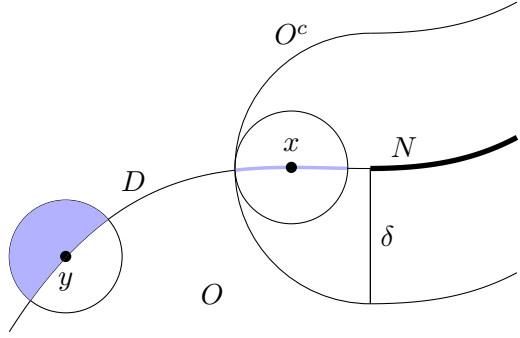

	This terminology carries the idea of a fatness assumption on $O^c ( \supseteq D )$ with the additional requirement that the lower bound on the capacity already has to come from the complementary boundary part $D (\subseteq O^c)$ as points get closer to $N$ (point $x$ instead of $y$ in Figure~\ref{fig. The geometric setup: Assumption (Fat)}). In Section~\ref{Section: Real-valued A} we will need a self-improvement property of~\hyperref[The geometric setup: Definition (Fat)]{$\Fatp$} with respect to $p$ in the spirit of Lewis' result \cite{Lewis_P-fat_Open_ended}. To this end, the equivalent formulation of~\hyperref[The geometric setup: Definition (Fat)]{$\Fatp$} below will be useful. For most of the paper, we shall work with~\hyperref[The geometric setup: Definition (Fat)]{$\Fat$} and set $p=2$.

	Given $\delta > 0$, let $\Sigma \sub \R^d$ be a grid of closed, axis-parallel cubes of diameter $\nicefrac{\delta}{8}$ and define
	\begin{equation*}
		N_{\delta}^{\Sigma} \coloneqq \mathrm{interior} \left( \bigcup \, \{ Q \in \Sigma \colon \, Q \cap \overline{N_{\delta}} \neq \varnothing \} \right).
	\end{equation*}
	The set $N_{\delta}^{\Sigma}$ is a regularized version of $N_{\delta}$ such that $N_{\delta} \sub N_{\delta}^{\Sigma} \sub N_{\nicefrac{9 \delta}{8}}$. In particular, as a union of cubes of the same size, $(N_{\delta}^{\Sigma})^c$ is locally $p$-fat for any $p \in (1,d]$ by Poincar\'{e}'s inequality, see also Section~\ref{Subsection: Comparison of the geometric setup}, which is not necessarily true for $(N_{\delta})^c$.

	\begin{lemma}  \label{The geometric setup: Lemma: (Fat) rephrased}
		Let $p \in (1, d]$. The following assertions are equivalent.
		\begin{enumerate}
			\item $D \cup (O^c \setminus N_{\delta}^{\Sigma})$ is locally $p$-fat for some $\delta > 0$.

			\item $O^c$ is locally $p$-fat away from $N$.
		\end{enumerate}
		In addition, if one of the conditions holds true with $\delta > 0$, then the other one holds true with $\nicefrac{\delta}{2}$.
	\end{lemma}

	\begin{proof}
		\textbf{(ii) $\boldsymbol{\Longrightarrow}$ (i):} Let $\delta > 0$ be as in Definition~\ref{The geometric setup: Definition (Fat)}. We show that $U \coloneqq D \cup (O^c \setminus N_{\nicefrac{\delta}{2}}^{\Sigma})$ is locally $p$-fat. Let $x \in U$ and $r \leq \nicefrac{\delta}{4}$. We make the following case distinction:

		\textbf{(1) $\boldsymbol{\overline{B(x, \nicefrac{r}{2})} \cap (D \cap N_{\delta}) \neq \varnothing}$.} Pick $z \in \overline{B(x, \nicefrac{r}{2})} \cap (D \cap N_{\delta})$. Then $\overline{B(z, \nicefrac{r}{2})} \cap D \sub \overline{B(x,r)} \cap U$ and the local $p$-fatness of $D$ in $D \cap N_{\delta}$ yields the claim.

		\textbf{(2) $\boldsymbol{\overline{B(x, \nicefrac{r}{2})}  \cap (D \cap N_{\delta}) = \varnothing}$.} Then $x \in O^c \setminus N_{\nicefrac{\delta}{2}}^{\Sigma}$ and we consider two subcases.

		\textbf{(2.1) $\boldsymbol{\overline{B(x, \nicefrac{r}{2})}  \cap D \neq \varnothing}$.} Let $z \in \overline{B(x, \nicefrac{r}{2})}  \cap (D \setminus N_{\delta})$. Since $r \leq \nicefrac{\delta}{4}$ we get $\overline{B(z, \nicefrac{r}{2})} \cap O^c \sub \overline{B(x,r)} \cap U$ and conclude from the local $p$-fatness of $O^c$ in $D$.

		\textbf{(2.2) $\boldsymbol{\overline{B(x, \nicefrac{r}{2})}  \cap D = \varnothing}$.} It follows that $\overline{B(x, \nicefrac{r}{2})}  \sub O^c$. Thus, we have $\overline{B(x, \nicefrac{r}{2})}  \cap (N_{\nicefrac{\delta}{2}}^{\Sigma})^c \sub \overline{B(x,r)} \cap U$ and deduce the claim from the local $p$-fatness of $(N_{\nicefrac{\delta}{2}}^{\Sigma})^c$.

		\textbf{(i) $\boldsymbol{\Longrightarrow}$ (ii):} Let $U \coloneqq D \cup (O^c \setminus N_{\delta}^{\Sigma})$. We show that $D$ is locally $p$-fat in $D \cap N_{\nicefrac{\delta}{2}}$ and $O^c$ is locally $p$-fat in $D$. The second assertion follows, since $U \sub O^c$ and $D \sub U$. For the first assertion, let $x \in D \cap N_{\nicefrac{\delta}{2}}$ and $r \leq \nicefrac{\delta}{4}$. Then $\overline{B(x,r)} \cap D = \overline{B(x,r)} \cap U$ and we conclude again from the local $p$-fatness of $U$.
	\end{proof}

	We will see in Proposition~\ref{The geometric setup: Proposition: (P)} that~\hyperref[The geometric setup: Definition (Fat)]{$\Fat$} is substantial for having a boundary Poincar\'{e}~inequality on $\H_D^{1}(O)$ without average. There are essentially two ways to get this inequality: either the extension of $u$ to the whole ball $B$ vanishes on a set that has measure comparable to $B$ or $u$ vanishes on a portion of $D$ that is ``nice enough" in this capacitary sense. The fatness assumption treats both cases simultaneously.

	While~\hyperref[The geometric setup: Definition (Fat)]{$\Fat$} describes $O$ away from $N$ in our main result, we use the following quantitative connectedness condition near $N$, see also Figure~\ref{fig. The geometric setup: Epsilon-Delta Condition}.

	\begin{definition}[Assumption (LU)] \label{The geometric setup: Definition (LU)}
		Let $\varepsilon \in (0, 1]$ and $\delta \in (0, \infty]$. We call \textbf{$\boldsymbol{O}$ locally an $\boldsymbol{(\varepsilon, \delta)}$-domain near $\boldsymbol{N}$}, if the following properties hold:
		\begin{enumerate}
			\item[(i)] All points $x,y \in O \cap N_{\delta}$ with $0 < |x-y| < \delta$ can be joined in $O$ by an \textbf{$\boldsymbol{\varepsilon}$-cigar with respect to $\boldsymbol{\partial O \cap N_{\delta}}$}, that is to say, a rectifiable curve $\gamma \sub O$ of length $\ell(\gamma) \leq \nicefrac{|x-y|}{\varepsilon}$ such that we have for all $z \in \mathrm{tr}(\gamma)$ that
			\begin{equation} \label{eq: The geometric setup: Epsilon-Delta Condition}
				\d(z, \partial O \cap N_{\delta}) \geq \frac{\varepsilon |z-x||z-y|}{|x-y|}.
			\end{equation}
			\item[(ii)] $O$ has \textbf{positive radius near $\boldsymbol{N}$}, that is, there is some $C > 0$ such that all connected components $O'$ of $O$ with $\partial O' \cap N \neq \varnothing$ satisfy $\diam (O') \geq C \delta$.
		\end{enumerate}
		If the values $\varepsilon, \delta, C$ need not be specified, then $O$ is called \textbf{locally uniform near $\boldsymbol{N}$}.
	\end{definition}

	\begin{figure}
		\centering
		\begin{tikzpicture}[scale=2.0]
			\draw[black, domain=-2.5:2, samples=100] plot (\x, {\x*\x/15 + 1});
			\draw[black, domain=-2.5:2, samples=100] plot (\x, {\x*\x*\x/15 - \x*\x/15 + 0.5});
			\draw[black, domain=-1:0, samples=100] plot (\x, {\x*\x/4 +5* \x/4});
			\draw[black, domain=-1:0, samples=100] plot (\x, {1.5*\x*\x + 2.5* \x});
			\node[inner sep=1.5pt, fill=black, circle] at (0,0) (x) {};
			\node[inner sep=1.5pt, fill=black, circle] at (-1,-1) (y) {};
			\node[above] at (-0.7,-0.4) {$O$};
			\node[above] at (0,1.2) {$O$};
			\node[right] at (x.east) {$x$};
			\node[left] at (y.west) {$y$};
		\end{tikzpicture}
		\caption{An illustration of an $\varepsilon$-cigar between $x$ and $y$. In view of \eqref{eq: The geometric setup: Epsilon-Delta Condition}, the cigar is contained in $O$. In order to understand the nature of the cigar shape, we use the length condition $\ell(\gamma) \leq \nicefrac{|x-y|}{\varepsilon}$ to obtain from \eqref{eq: The geometric setup: Epsilon-Delta Condition} that$$\frac{\varepsilon}{2}(|z-x| \land |z-y|) \leq \frac{\varepsilon |z-x| |z-y|}{|x-y|} \leq |z-x| \land |z-y|.\qquad \qquad \qquad$$Hence, \eqref{eq: The geometric setup: Epsilon-Delta Condition} means essentially that every point $z$ on $\mathrm{tr}(\gamma)$ keeps distance to $\partial O \cap N_{\delta}$ that is at least the minimum of its distance to $x$ and $y$. We refer to \cite{Uniform_domains_book, Vaisala-Uniform-domains} for more information.}
		\label{fig. The geometric setup: Epsilon-Delta Condition}
	\end{figure}
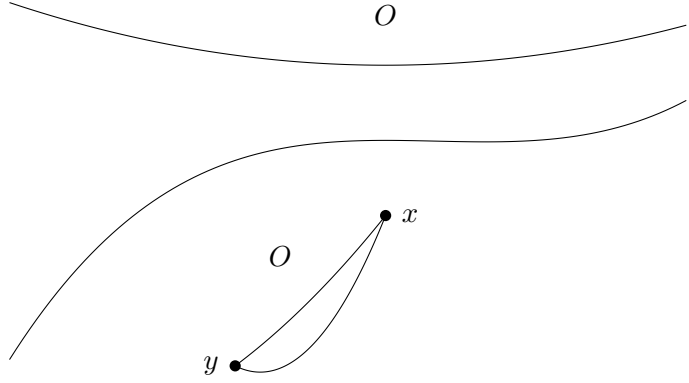

	\begin{definition}
		Let $c > 0$. We say that
		\begin{enumerate}
			\item $c$ depends on the \textbf{geometry} if $c$ depends only on dimension and the parameters in the definitions of~\hyperref[The geometric setup: Definition (Fat)]{$\Fat$} and~\hyperref[The geometric setup: Definition (LU)]{$\LU$}.

			\item $c$ depends on \textbf{ellipticity} if $c$ depends on $\lambda$ and $\Lambda$.
		\end{enumerate}
	\end{definition}

	Assumption~\hyperref[The geometric setup: Definition (LU)]{$\LU$} has been introduced in \cite{Kato_Mixed}, to which we refer for a detailed discussion. It is slightly stronger than the related condition in \cite{Extension_Operator}, see \cite[Prop.~2.5 (ii)]{Kato_Mixed}.

	Notice further that~\hyperref[The geometric setup: Definition (Fat)]{$\Fat$} and~\hyperref[The geometric setup: Definition (LU)]{$\LU$} become weaker as $\delta$ decreases. Hence, we can (and will) always assume that $\delta \leq 1$ with the same choice of $\delta$ in both conditions.

	\begin{theorem}[{\cite[Thm.~10.2]{Extension_Operator}}, $\AssExt$]  \label{The geometric setup: Theorem: Extension operator}
		Assume ~\hyperref[The geometric setup: Definition (LU)]{$\LU$}. There are $K \geq 1$, $A \leq \nicefrac{1}{2}$ and an extension operator $\cE$ from $\L^1_{\loc}(O)$ into the space of measurable functions defined on $\R^d$ such that for all $p \in [1, \infty)$ one has that $\cE$ restricts to a bounded operator from $\W^{1,p}_D(O)$ to $\W^{1,p}_D(\R^d)$, which is local and homogeneous, that is,
		\begin{equation} \label{eq: The geometric setup: E is local and hom}
			\| \nabla^{\ell} \cE u \|_{\L^p(B(x,r))} \les \| \nabla^{\ell} u \|_{\L^p(O(x, K r))}
		\end{equation}
		holds true for all $u \in \W_D^{1,p}(O)$, $\ell = 0, 1$, $x \in \partial O$ and $r \in (0, A \delta]$. The implicit constant depends only on the parameters in~\hyperref[The geometric setup: Definition (LU)]{$\LU$}.
	\end{theorem}

	Now, we draw important consequences from~\hyperref[The geometric setup: Definition (Fat)]{$\Fat$} and~\hyperref[The geometric setup: Definition (LU)]{$\LU$}. The first one implies that $O$ has no exterior cusps near $N$.

	\begin{proposition}[{\cite[Prop.~2.9]{Kato_Mixed}}] \label{The geometric setup: Proposition: (ICC)}
		Assume~\hyperref[The geometric setup: Definition (LU)]{$\LU$}. We have an interior corkscrew condition for $O$ near $N$:
		\begin{equation*}
			\tag*{\text{$(\mathrm{ICC}_{N_{\delta}})$}}  \exists \, \alpha > 0 \; \forall \, x \in \overline{O \cap N_{\nicefrac{\delta}{2}}}, r \in (0,1] \; \exists \, z \in O \colon \quad B(z, \alpha r) \sub O(x,r).
		\end{equation*}
	\end{proposition}

	The second one is a weak Poincar\'{e} inequality with correct scaling. In the formulation, $\cE, A, \delta$ and $K$ are as in Theorem~\ref{The geometric setup: Theorem: Extension operator}. In the proof, we frequently use the fact that$$\inf_{c \in \C} \| u - c \|_{\L^p(E)} \leq \| u - (u)_E \|_{\L^p(E)} \leq 2 \inf_{c \in \C} \| u - c \|_{\L^p(E)}$$whenever $p \in [1, \infty)$, $E \sub \R^d$ has positive and finite measure, and $u \in \L^p(E)$.

	\begin{proposition}[Weak Poincar\'e inequality] \label{The geometric setup: Proposition: (P)}
		Let $p \in (1,d]$ and assume~\hyperref[The geometric setup: Definition (Fat)]{$\Fatp$} and~\hyperref[The geometric setup: Definition (LU)]{$\LU$}. There is $c_0 > 0$ depending on the geometry and $p$ such that
		\begin{equation*}   \tag*{\text{$\AssPp$}}
			\| u - \1_{[\d_D(x) > r]} \cdot (u)_{O(x,r)} \|_{\L^p(O(x,r))} \leq c_0 r \| \nabla u \|_{\L^p(O(x, 3 K r))}
		\end{equation*}
		for all $u \in \W_D^{1,p}(O)$, each $x \in \overline{O}$ and all $r \in (0, \nicefrac{A \delta}{2}]$.
	\end{proposition}

	\begin{proof}
		By density, we can assume $u \in \smoothD[O]$. First, we note that if $\overline{B(x,r)} \sub O$, then $\d_D(x) > r$ and $O(x,r) = B(x,r)$, and \hyperref[The geometric setup: Proposition: (P)]{$\AssPp$} follows from the standard Poincaré inequality (with subtraction of the average) on the ball $B(x,r)$. Hence, we assume from now on that $\overline{B(x,r)} \cap \partial O \neq \varnothing$.

		We distinguish two cases.

		\textbf{(1) $\boldsymbol{\d_D(x) \leq r}$.} In this case there exists $x_D \in \overline{B(x,r)} \cap D$.

		\textbf{(1.1) $\boldsymbol{x_D \in D \cap N_{\delta}}$.} We estimate
		\begin{align*}
			\| u \|_{\L^p(O(x,r))}
			&\leq \| u \|_{\L^p(O(x_D, 2r))} \\
			&\leq \| \cE u \|_{\L^p(B(x_D, 2r))}.
			\intertext{Since $u \in \smoothD[O]$ and $\cE$ is local and homogeneous, we have for all $y \in D$ and sufficiently small $r > 0$ that $\| \cE u \|_{\L^p(B(y,r))} \les \| u \|_{\L^p(O(y, Kr))} =0$. From this we conclude that $\cE u$ vanishes almost everywhere on an open neighborhood of $D$. Hence, in view of~\hyperref[The geometric setup: Definition (Fat)]{$\Fatp$}, we can apply Mazya's Poincaré inequality \cite[Lem.~3.1]{Poincare-Capacity} and \eqref{eq: The geometric setup: E is local and hom} to continue by}
			&\les r \| \nabla \cE u \|_{\L^p(B(x_D, 2r))} \\
			&\les r \| \nabla u \|_{\L^p(O(x_D, 2Kr))} \\
			&\leq r \| \nabla u \|_{\L^p(O(x, 3 K r))}.
		\end{align*}
		\textbf{(1.2) $\boldsymbol{x_D \in D \setminus N_{\delta}}$.} Since $2 r \leq \nicefrac{\delta}{2}$, we know that $u_0$ belongs to $\W^{1,p}(B(x_D, 2r))$, and of course $u_0$ vanishes on $O^c$. Hence, the same argument as in case (1.1) applies with $u_0$ replacing $\cE u$.

		\textbf{(2) $\boldsymbol{r < \d_D(x)}$.} In this case there exists $x_N \in \overline{B(x,r)} \cap N$. Using the standard Poincaré inequality on $B(x_N, 2r)$ and \eqref{eq: The geometric setup: E is local and hom}, we get the desired estimate
		\begin{align*}
			\| u - (u)_{O(x,r)} \|_{\L^p(O(x,r))} &\leq 2 \| u - (\cE u)_{B(x_N, 2r)} \|_{\L^p(O(x, r))}
			\\&\leq 2 \| \cE u - (\cE u)_{B(x_N, 2r)} \|_{\L^p(B(x_N, 2r))}
			\\&\les r \| \nabla \cE u \|_{\L^p(B(x_N, 2r))}
			\\&\les r \| \nabla u \|_{\L^p(O(x_N, 2 K r))}
			\\&\leq r \| \nabla u \|_{\L^p(O(x, 3 K r))}. \qedhere
		\end{align*}
	\end{proof}

	As with~\hyperref[The geometric setup: Definition (Fat)]{$\Fatp$} we simply write~\hyperref[The geometric setup: Proposition: (P)]{$\AssP$} instead of $\AssP_2$. From now on we set $p=2$ and use the fixed constants
	\begin{equation*}
		\tag*{\text{$(c_1)$}} c_1 \coloneqq 3K \geq 3
	\end{equation*}
	and
	\begin{equation}
		\tag*{\text{$\Assrnot$}} r_0 \coloneqq (A \land C) \frac{\delta}{2}.
	\end{equation}
	Incorporating $C > 0$ from Definition~\ref{The geometric setup: Definition (LU)} in the definition of the radius $r_0$ will be useful at later occurrences.

	\begin{remark} \label{The geometric setup: Remark: (P) follows from boundary inequality}
		As we have seen in the last proof, working with weak Poincar\'{e} inequalities bears the advantage that the ball can be centered at the boundary. Because of this, \hyperref[The geometric setup: Proposition: (P)]{$\AssPp$} could equivalently be required with $x \in \partial O$ instead of $x \in \overline{O}$. Moreover, by using the triangle inequality, we get the Poincar\'{e} inequality with average,
		\begin{equation*}
			\| u - (u)_{O(x,r)} \|_{\L^p(O(x,r))} \les r \| \nabla u \|_{\L^p(O(x, c_1 r))},
		\end{equation*}
		for all $u \in \W^{1,p}_D(O)$, $x \in \overline{O}$ and $r \in (0, r_0]$.
	\end{remark}

	\subsection{Comparison of the geometric setup} \label{Subsection: Comparison of the geometric setup}

	Now, we provide a short comparison of our chosen geometry with the one in \cite{ter-Elst_Rehberg, Disser_ter-Elst_Rehberg}. We believe that it is instructive to see how their assumptions are built into our general framework.

	To show $\rG(\mu)$ for real-valued $A$, the following geometric setup is used, compare with \cite[Thm.~7.5]{ter-Elst_Rehberg}:

	\begin{enumerate}
		\item[(I)] \textbf{Uniform Lipschitz charts around $\boldsymbol{N}$:} There is $K \geq 1$ such that for all $x \in \overline{N}$ there is an open neighbourhood $U_x$ of $x$ and a bi-Lipschitz map $\Phi_x \colon U_x \to B(0,1)$ with bi-Lipschitz constant at most $K$ and the properties $\Phi_x(x) = 0$ and $\Phi_x(U_x \cap O) = (\R^d_+)(0,1)$.

		\item[(II)] \textbf{$\boldsymbol{O}$ is exterior thick in $\boldsymbol{D}$:}
		\begin{equation*}
			|(O^c)(x,r)| \ges r^d \qquad (x \in D, r \leq 1).
		\end{equation*}
		\item[(III)] \textbf{Interface condition between $\boldsymbol{D}$ and $\boldsymbol{N}$:} There is $c > 0$ such that
		\begin{equation*}
			\cH^{d-1} ( ((\R^{d-1} \times \{ 0 \}) \cap [ \d_{\Phi_x(N \cap U_x)}( \cdot) > c r ])(y,r)) \ges r^{d-1},
		\end{equation*}
		for all $x \in D \cap \overline{N}$, $y \in \Phi_x(D \cap \overline{N} \cap U_x)$ and $r \leq 1$. Here, $\cH^{d-1}$ is the $(d-1)$-dimensional Hausdorff measure in $\R^d$.
	\end{enumerate}

	\begin{lemma} \label{Comparison of the geometric setup: Lem: Joachim-Tom setting implies ours}
		If $O,D$ and $N$ satisfy $\mathrm{(I)}$, $\mathrm{(II)}$ and $\mathrm{(III)}$, then they also satisfy~\hyperref[The geometric setup: Definition (Fat)]{$\Fat$} and~\hyperref[The geometric setup: Definition (LU)]{$\LU$}.
	\end{lemma}

	\begin{proof}
		It is classical that (I) implies~\hyperref[The geometric setup: Definition (LU)]{$\LU$}, see \cite[p.~9]{Kato_Mixed} and references therein. The full details have been written out in \cite[Lem.~2.2.20]{Egert_PhD}. 
  
  To see that (II) implies that $O^c$ is locally $2$-fat in $D$, let $x \in D$, $r \leq 1$ and $u \in \smooth[ B(x,2 r)]$ with $u =1$ on $\overline{B(x,r)} \cap O^c$. Then (II) joint with Poincar\'{e}'s inequality yields
		\begin{equation*}
			r^{d-2} \les r^{-2} |(O^c)(x,r)| \leq r^{-2} \| u \|_{\L^2(B(x, 2r))}^2 \les \| \nabla u \|_{\L^2(B(x, 2r))}^2,
		\end{equation*}
		and hence
		\begin{equation*}
			r^{d-2} \les \Capp_2(\overline{B(x,r)} \cap O^c; B(x, 2r)).
		\end{equation*}
		Finally, let us explain why $D$ is locally $2$-fat in $D \cap N_{\delta}$ for some $\delta > 0$. In fact, (I), (III), and \cite[Lem.~5.4]{ter-Elst_Rehberg} show that there is some $\delta \in (0,1]$ such that
		\begin{equation}  \label{eq: Comparison of the geometric setup: D is (d-1)-set}
			\cH^{d-1} (D(x,r)) \simeq r^{d-1}  \qquad (r \leq 1, x \in D \cap N_{2 \delta}),
		\end{equation}
		compare also with \cite[p.~304]{ter-Elst_Rehberg}. It seems to be folklore that this implies the local $2$-fatness of $D$ in $D \cap N_{\delta}$. For convenience, we include the details in the appendix, see Lemma~\ref{App: Fat and Thick: Lem: Thickness and fatness}. Altogether, we have concluded~\hyperref[The geometric setup: Definition (Fat)]{$\Fat$} from (I), (II) and (III).
	\end{proof}

	To construct explicitly a set $O$ that fulfills~\hyperref[The geometric setup: Definition (LU)]{$\LU$} and~\hyperref[The geometric setup: Definition (Fat)]{$\Fat$}, but not the geometric setup from above, we consider $\R^2 \setminus \{ (x, y) \in \R^2 \colon x \geq 0 \; \& \; 0 \leq y \leq x^2 \}$ and add a part of the von Koch snowflake, see Figure~\ref{fig. The geometric setup: (Fat) and (LU), but not (I) to (III)}. We put Neumann boundary conditions on the ``fractal part'' coming from the snowflake and Dirichlet boundary conditions on its complement. Let us sketch that $O$ is an admissible example.
	\begin{enumerate}
		\item As $|(O^c)(0,r)| \leq \int_0^r x^2 \, \d x  = \nicefrac{r^3}{3}$ for small enough $r > 0$ it follows that $O$ is not exterior thick at the origin.
  
		\item The boundary of the von Koch snowflake is not even rectifiable, so there are no Lipschitz coordinate charts around $N$.
  
		\item By inspection, $\cH^1(D(x,r)) \simeq r$ for all $r \in (0,1]$ and $x \in D$. Hence, Lemma~\ref{App: Fat and Thick: Lem: Thickness and fatness} reveals $D$ is locally $2$-fat (in itself). As $D \sub O^c$, also~\hyperref[The geometric setup: Definition (Fat)]{$\Fat$} is satisfied.
  
		\item Since the von Koch snowflake is an $(\varepsilon, \infty)$-domain (see \cite[Prop.~6.30]{Uniform_domains_book}), one can verify that $O$ is an $(\varepsilon, \infty)$-domain as well.
	\end{enumerate}

	\begin{figure}
		\centering
		\begin{tikzpicture}[decoration=Koch snowflake]
			\draw[black, line width = 2.5] (0,0) -- (1.5,0);
			\draw[black, line width = 2.5] (4,0) -- (5,0);
			\draw[fill=blue!30] decorate{ decorate{ decorate{ decorate{ (1.5,0) -- (4,0) }}}};
			\draw[black, line width = 2.5] [smooth, samples=100,domain=0:1.7] plot ({\x}, {\x*\x});
			\draw[fill=blue!30] (0,0) -- (5,0) -- (5,-2) --(0,-2);
			\draw[fill=blue!30] (0,-2) -- (-2,-2) -- (-2,2.89) --(0,2.89);
			\draw[->] (-2,0) -- (5.3,0) node[right] {$x$};
			\draw[->] (0,-2) -- (0,3.3) node[above] {$y$};
			\fill[blue!30] (0,0) -- plot[domain=0:1.8,smooth] (\x,{\x^2}) -- (1.7,2.89) -- (0,2.89);
			\draw (0,2.89) -- (1.7,2.89);
			\node[above] at (3,1.5) {$O^c$};
			\node[below] at (1,-1) {$O$};
			\node[right] at (1,1) {$D$};
			\node[right] at (3,0.8) {$N$};
			\node[inner sep=1.5pt, fill=black, circle] at (1.5,0) (x) {};
			\node[inner sep=1.5pt, fill=black, circle] at (4,0) (y) {};
			\node[below] at (x.south) {$(1,0)$};
			\node[below] at (y.south) {$(2,0)$};
		\end{tikzpicture}
		\caption{A geometric constellation in $\R^2$ that satisfies~\hyperref[The geometric setup: Definition (Fat)]{$\Fat$} and~\hyperref[The geometric setup: Definition (LU)]{$\LU$}, but not the geometric setup introduced in this subsection. Here, $N$ can be constructed by the following algorithm: divide the line segment between $(1,0)$ and $(2,0)$ into three parts of equal length, remove the middle one, and build an equilateral triangle over this segment. Then apply this procedure to each of the four remaining segments and iterate.}
		\label{fig. The geometric setup: (Fat) and (LU), but not (I) to (III)}
	\end{figure}
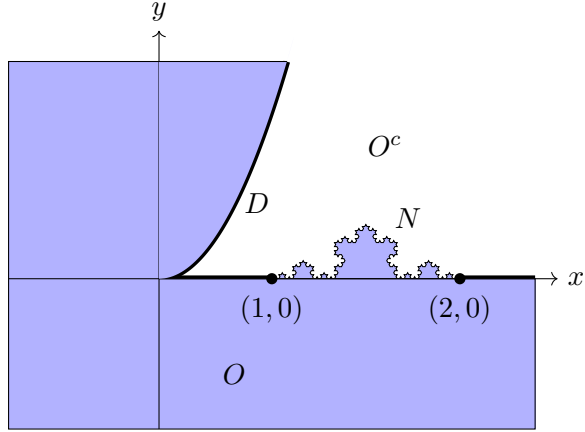

	We close this section by showing that (under the background assumption~\hyperref[The geometric setup: Definition (LU)]{$\LU$}) having~\hyperref[The geometric setup: Definition (Fat)]{$\Fat$} is just as good as having the abstract assumption~\hyperref[The geometric setup: Proposition: (P)]{$\AssP$}. Recall that~\hyperref[The geometric setup: Definition (LU)]{$\LU$} is void in the case of pure Dirichlet boundary conditions.

	\begin{proposition} \label{Optimality of the geometric setup: Proposition: 2-fatness optimal}
		Let $d \geq 2$ and assume~\hyperref[The geometric setup: Definition (LU)]{$\LU$}. Then~\hyperref[The geometric setup: Definition (Fat)]{$\Fat$} is equivalent to~\hyperref[The geometric setup: Proposition: (P)]{$\AssP$}.
	\end{proposition}

	\begin{proof}
		That~\hyperref[The geometric setup: Definition (Fat)]{$\Fat$} and~\hyperref[The geometric setup: Definition (LU)]{$\LU$} imply~\hyperref[The geometric setup: Proposition: (P)]{$\AssP$} has been shown in Proposition~\ref{The geometric setup: Proposition: (P)}. For the converse statement, we borrow ideas from \cite[Thm.~3.3]{P_D_implies_2-fat}. To show~\hyperref[The geometric setup: Definition (Fat)]{$\Fat$} we fix $x \in D$ and $r \leq r_0$. We consider two cases.

		\textbf{(1) $\boldsymbol{x \in D \setminus N_{\delta}}$.} Let $u \in \smooth[B(x, 2r)]$ with $u =1$ on $\overline{B(x,r)} \cap O^c$. If we assume that
		\begin{equation*}
			\frac{1}{4} |B(x, \nicefrac{r}{2 c_1})| \leq \| u \|_{\L^2(B(x, \nicefrac{r}{2 c_1}))}^2,
		\end{equation*}
		then Poincar\'{e}'s inequality applied on $B(x, 2r)$ implies
		\begin{equation*}
			r^{d-2} \les \| \nabla u \|_{\L^2(B(x, 2r))}^2.
		\end{equation*}
		Now, we assume the converse estimate. Then
		\begin{align*}
			\frac{1}{2} |B(x, \nicefrac{r}{2 c_1})| &\leq \| 1-u \|_{\L^2(B(x, \nicefrac{r}{2 c_1}))}^2 + \| u \|_{\L^2(B(x, \nicefrac{r}{2 c_1}))}^2
			\\&\leq \| 1-u \|_{\L^2(B(x, \nicefrac{r}{2 c_1}))}^2 + \frac{1}{4} |B(x, \nicefrac{r}{2 c_1})|
		\end{align*}
		and hence
		\begin{equation*}
			r^d \les \| 1- u \|_{\L^2(B(x, \nicefrac{r}{2 c_1}))}^2.
		\end{equation*}
		Let $\varphi \in \smooth[B(x,r)]$ with $\varphi =1$ on $B(x, \nicefrac{r}{2})$ and put $v \coloneqq \varphi (1-u)$. Note that $v \in \smooth[\R^d]$ with $v = 0$ on $D$. Hence, $v \in \H_D^1(O)$ by \cite[Thm.~9.1.3]{Hedberg} and~\hyperref[The geometric setup: Proposition: (P)]{$\AssP$} yields
		\begin{align*}
			r^d \les \int_{B(x, \nicefrac{r}{2 c_1})} |1-u|^2 = \int_{O(x, \nicefrac{r}{2 c_1})} |v|^2 \les r^2  \int_{O(x, \nicefrac{r}{2})} |\nabla v|^2 = r^2 \int_{B(x, \nicefrac{r}{2})} |\nabla u|^2.
		\end{align*}
		This shows that $O^c$ is locally $2$-fat in $D \setminus N_{\delta}$.

		\textbf{(2) $\boldsymbol{x \in D \cap N_{\delta}}$.} To prove that $D$ is locally $2$-fat in $D \cap N_{\delta}$, we systematically replace $O^c$ by $D$ and $B(x, \nicefrac{r}{2 c_1})$ by $O(x, \nicefrac{r}{2 c_1})$ in (1) and apply the same argument. The key points are that we now have $v \in \H_0^1(\R^d \setminus D)$ and hence $v|_O \in \H_D^1(O)$, and $|O(x, \nicefrac{r}{2 c_1})| \simeq r^d$ due to~\hyperref[The geometric setup: Proposition: (ICC)]{$\ICC$}.
	\end{proof}


	\section{Properties $\rD(\mu)$, $\rG(\mu)$ and $\H(\mu)$}  \label{Section: D(mu), G(mu), H(mu) for MBC}

	The property that we are mostly interested in is the Gaussian estimate for the kernel of the semigroup $(\e^{- t L})_{t \geq 0}$. Let us introduce this property in detail:

	\begin{definition}
		Let $\mu \in (0,1]$. We say that $L$ has \textbf{property} $\boldsymbol{\rG(\mu)}$ if the following holds:
		\begin{enumerate}
			\item[(G1)] For any $t > 0$ there is a measurable function $K_t \colon O \times O \to \C$ such that
			\begin{equation*}
				(\e^{-t L} f)(x) = \int_O K_t(x,y) f(y) \, \d y \qquad (f \in \L^2(O), \text{a.e.} \; x \in O).
			\end{equation*}
			\item[(G2)] There are $b,c, \omega > 0$ such that we have for each $t > 0$ that
			\begin{equation*}
				|K_t(x,y)| \leq c t^{- \frac{d}{2}} \e^{- b \frac{|x-y|^2}{t}} \e^{\omega t}.
			\end{equation*}
			\item[(G3)] For every $x,x',y,y' \in O$ and $t> 0$ we have
			\begin{equation*}
				|K_t(x,y) - K_t(x',y')| \leq c t^{- \frac{d}{2} - \frac{\mu}{2}} ( |x-x'| + |y-y'| )^{\mu} \e^{\omega t}.
			\end{equation*}
		\end{enumerate}
	\end{definition}

	\begin{remark} \label{D(mu), G(mu) and H(mu) for MBC: Remark: G(mu) stability}
		The following facts will be useful in this paper:
		\begin{enumerate}
			\item Property $\rG(\mu)$ is stable under taking adjoints since the kernel of the adjoint semigroup is given by $K_t^*(x,y) = \overline{K_t(y,x)}$.

			\item Logarithmic convex combinations of (G2) and (G3) yield for all $\nu \in (0, \mu)$, each $x,y \in O$, $h \in \R^d$ with $y+h \in O$ and $t > 0$ a bound
			\begin{equation*}
				|K_t(x,y+h) - K_t(x,y)| \leq c t^{- \frac{d}{2}} \left( \frac{|h|}{\sqrt{t}} \right)^{\nu} \e^{- b \frac{|x-y|^2}{t}} \e^{\omega t},
			\end{equation*}
			with different constants $b, c, \omega$ provided that $|h| \leq \nicefrac{|x-y|}{2}$. A similar estimate holds true in the $x$-variable.
		\end{enumerate}


	\end{remark}

	The (eventually equivalent) properties $\rD(\mu)$ and $\H(\mu)$ talk about the regularity of weak solutions in subsets of $O$. For pure Dirichlet boundary conditions, these properties have appeared in the introduction, but their adaptation to general boundary conditions requires some care. Following \cite{ter-Elst_Rehberg}, we do that by looking at solutions to $- \Div A \nabla u = 0$ in $O(x,r) = O \cap B(x,r)$ that are compatible with the ``global" boundary conditions (Dirichlet on $D$, Neumann on $N$), that is, we use test functions with pure Dirichlet boundary conditions only on $\partial O(x,r) \setminus N(x,r)$. In this case we write $L_D u = 0$ in $O(x,r)$ and the precise variational formulation is as follows:

	\begin{definition}
		Let $x \in \overline{O}$, $r > 0$, $u \in \H^1(O(x,r))$ and $f \in \L^2(O(x,r)), F \in \L^2(O(x,r))^d$. We write $L_D u = f - \Div F $ in $O(x,r)$ if
		\begin{equation*}
			\int_{O(x,r)} A \nabla u \cdot \overline{\nabla \varphi} = \int_{O(x,r)} f \overline{\varphi} +F \cdot \overline{\nabla \varphi} \qquad  (\varphi \in \H_{\partial O(x,r) \setminus N(x,r)}^1(O(x,r))).
		\end{equation*}
		In addition, given $u \in \H_D^1(O)$ and $f \in \L^2(O), F \in \L^2(O)^d$, we write $L_D u = f - \Div F$ in $O$ if
		\begin{equation*}
			\int_{O} A \nabla u \cdot \overline{\nabla \varphi} = \int_{O} f \overline{\varphi} + F \cdot \overline{\nabla \varphi}  \qquad  (\varphi \in \H_D^1(O)).
		\end{equation*}
	\end{definition}

	One reason why this definition of $L_D u = f - \Div F$ is natural for our purpose is because the class of test functions in the previous definition is canonically embedded into $\H_D^1(O)$: If $u \in \H_D^1(O)$ satisfies a global equation $L_D u = f - \Div F$, then also $L_D u = f - \Div F$ in all local sets $O(x,r)$ due to part (i) of the following lemma applied with $U=B(x,r)$.

	\begin{lemma} \label{D(mu), G(mu) and H(mu) for MBC: Lemma: Sobolev extension by 0}
		Let $U \sub \R^d$ be open and $p \in [1, \infty)$.
		\begin{enumerate}
			\item $N \cap U$ is open in $\partial (O \cap U)$ and the $0$-extension $\cE_0 \colon \W_{\partial (O \cap U) \setminus (N \cap U)}^{1,p}(O \cap U) \to \W^{1,p}_D(O)$ is isometric.

			\item If $\psi \in \smooth[U]$, then multiplication by $\psi$ maps the space $\W^{1,p}_D(O)$ boundedly into $\W^{1,p}_{\partial(O \cap U) \setminus (N \cap U)}(O \cap U)$.
		\end{enumerate}
	\end{lemma}

	\begin{proof}
		For (i) see \cite[Lem.~6.3]{ter-Elst_Rehberg}. For (ii) we pick $u \in \W^{1,p}_D(O)$. By definition this means that there is a sequence $(\varphi_n)_n \sub \smooth[\R^d \setminus D]$ with $\varphi_n|_O \to u$ in $\W^{1,p}(O)$. Then $\psi \varphi_n \in \smooth[\R^d]$ and to see that $\supp(\psi \varphi_n) \cap [\partial (O \cap U) \setminus (N \cap U)] = \varnothing$, we notice that $\supp(\psi \varphi_n) \sub U \setminus D$ and
		\begin{equation*}
			(U \setminus D) \cap (\partial (O \cap U) \setminus (N \cap U)) \sub (\partial O \cap U) \setminus ( D \cup (N \cap U)) = \varnothing.
		\end{equation*}
		We have shown that $\psi \varphi_n \in \rC^\infty_{\partial(O \cap U) \setminus (N \cap U)}(\R^d)$ and the claim follows by passing to the limit in $n$.
	\end{proof}

	Now, we introduce $\rD(\mu)$ and $\H(\mu)$.
	\begin{definition}
		Let $\mu \in (0,1]$. We say that $L$ has \textbf{property $\boldsymbol{\rD(\mu)}$} if there is some $c_{\rD(\mu)} > 0$ such that for all $0<r \leq R\leq 1$, every $x \in \overline{O}$ and all $u \in \H^1_D(O)$ with $L_D u =0$ in $O(x,R)$ we have
		\begin{equation}
			\int_{O(x,r)} |\nabla u|^2 \leq c_{\rD(\mu)} \left( \frac{r}{R} \right)^{d-2 + 2 \mu} \int_{O(x, R)} |\nabla u|^2.  \label{eq: Definition: Property D(mu)}
		\end{equation}
	\end{definition}

	\begin{remark} \label{(D(mu), G(mu) and H(mu) for MBC: Remark: Scaling D(mu)}
		The interest in property $\rD(\mu)$ lies in radii that are not comparable by absolute constants. Otherwise the estimate holds by monotonicity of the integral with any choice of $\mu$. In particular, we can replace the condition $0 < r \leq R \leq 1$ by the more flexible condition $0 < c r \leq R \leq R_0$ for any fixed $R_0 > 0$ and $c > 1$.
	\end{remark}

	\begin{definition} \label{D(mu), G(mu) and H(mu) for MBC: Definition: H(mu)}
		Let $\mu \in (0,1]$. We say that $L$ has \textbf{property $\boldsymbol{\H(\mu)}$} if there is some $c_{\H(\mu)} > 0$ such that for all $r \in (0,1]$, every $x \in \overline{O}$ and all $u \in \H^1_D(O)$ with $L_D u =0$ in $O(x,r)$ we have that $u$ has a continuous representative in $O(x, \nicefrac{r}{2})$ that satisfies
		\begin{equation}
			r^{\mu} [u]^{(\mu)}_{O(x, \frac{r}{2})} \leq c_{\H(\mu)} r^{-\frac{d}{2}} \| u \|_{\L^2(O(x,r))}.  \label{eq: Definition: Property H(mu)}
		\end{equation}
	\end{definition}

	Note that this deﬁnition is different from the one given in the introduction, but with some work it turns out to be equivalent as we will see in Lemma~\ref{(D(mu), G(mu) and H(mu) for MBC: Lemma: Add L-infinity part in H(mu)} below. In the setting of $\H(\mu)$ the function $u$ extends continuously to $\overline{O(x,\nicefrac{r}{2})}$ and the Dirichlet condition gets a pointwise meaning:

	\begin{lemma} \label{(D(mu), G(mu) and H(mu) for MBC: Lemma: Pointwise real. of bdr. conditions}
		Assume~\hyperref[The geometric setup: Definition (Fat)]{$\Fat$} and~\hyperref[The geometric setup: Definition (LU)]{$\LU$}. Let $x \in \overline{O}$, $r > 0$ and $u \in \H_D^1(O)$. If $u$ has a continuous representative in $\overline{O(x, r)}$, then $u(y) = 0$ for all $y \in \overline{D(x, r)}$.
	\end{lemma}

	\begin{proof}
		By continuity, and since $y$ and $r$ are arbitrary, it suffices to prove $u(y)=0$ when $y \in D$. We can assume that $u$ is real-valued, since otherwise we consider real- and imaginary parts separately. For the sake of a contradiction, suppose that $|u(y)|  \neq 0$. By continuity, pick $0<\rho \leq r_0 \wedge r$ and $c>0$ such that $|u| \geq c$ on $O(y,\rho)$. Repeated application of the truncation property in Lemma~\ref{Real valued A: Lemma: Auxiliary results} gives $|u| \wedge c \in \H^1_D(O)$. But on $O(y,\rho)$ this function is the constant $c$ and the Poincar\'e inequality in Proposition~\ref{The geometric setup: Proposition: (P)} yields the contradiction $c=0$.
	\end{proof}

	A further consequence of our geometric setup is that property $\H(\mu)$ yields a posteriori local boundedness of $L$-harmonic functions.

	\begin{lemma}  \label{(D(mu), G(mu) and H(mu) for MBC: Lemma: Add L-infinity part in H(mu)}
		Assume~\hyperref[The geometric setup: Definition (Fat)]{$\Fat$} and~\hyperref[The geometric setup: Definition (LU)]{$\LU$} and let $L$ have property $\H(\mu)$. Then there is some $c_{\H(\mu)} > 0$ such that for all $r \in (0,1]$, every $x \in \overline{O}$ and all $u \in \H^1_D(O)$ with $L_D u =0$ in $O(x,r)$ we have that $u$ has a continuous representative in $O(x, \nicefrac{r}{2})$ that satisfies \begin{equation}
			\| u \|_{\L^{\infty}(O(x, \frac{r}{2}))} + r^{\mu} [u]^{(\mu)}_{O(x, \frac{r}{2})} \leq c_{\H(\mu)} r^{-\frac{d}{2}} \| u \|_{\L^2(O(x,r))}.  \label{eq: Definition: Property H(mu) + L-infty-bound}
		\end{equation}
	\end{lemma}

	\begin{proof}
		We only need to bound the $\L^{\infty}$-norm. We distinguish the following cases:

		\textbf{(1) $\boldmath x \in N_{\nicefrac{\delta}{2}}$ or $\boldmath B(x, \nicefrac{r}{4}) \sub O$.} Using $\H(\mu)$ we have for all $y, z \in O(x, \nicefrac{r}{2})$ that
		\begin{equation} \label{eq: Appendix: Remark: H(mu), L infinity removed}
			|u(y)| \leq r^{\mu}[u]^{(\mu)}_{O(x, \frac{r}{2})} + |u(z)| \les r^{-\frac{d}{2}} \| u \|_{\L^2(O(x, r))} + |u(z)|.
		\end{equation}
		Now, we average with respect to $z$ on $O(x, \nicefrac{r}{2})$ to get
		\begin{equation}
			|u(y)| \les ( r^{-\frac{d}{2}} + |O(x,\nicefrac{r}{2})|^{- \frac{1}{2}} ) \| u \|_{\L^2(O(x, r))}.
		\end{equation}
		By either~\hyperref[The geometric setup: Proposition: (ICC)]{$\ICC$} or $B(x, \nicefrac{r}{4}) \sub O$ we get $|O(x, \nicefrac{r}{2})| \simeq r^d$, which proves the claim.

		\textbf{(2) $\boldsymbol{x \in (N_{\nicefrac{\delta}{2}})^c}$ and $\boldmath (\partial O)(x, \nicefrac{r}{4}) \neq \boldsymbol{\varnothing}$.} We consider two subcases.

		\textbf{(2.1) $\boldmath D(x, \nicefrac{r}{4}) \neq \boldsymbol{\varnothing}$.} Pick $y \in D(x, \nicefrac{r}{4})$. Lemma~\ref{(D(mu), G(mu) and H(mu) for MBC: Lemma: Pointwise real. of bdr. conditions} implies $u(y) =0$ and we get for all $z \in O(x, \nicefrac{r}{2})$ that
		\begin{equation*}
			|u(z)| = |u(z) - u(y)| \leq r^{\mu}[u]^{(\mu)}_{O(x, \frac{r}{2})} \les r^{-\frac{d}{2}} \| u \|_{\L^2(O(x, r))}.
		\end{equation*}
		\textbf{(2.2) $\boldmath D(x, \nicefrac{r}{4}) = \boldsymbol{\varnothing}$.} Pick $w \in N(x, \nicefrac{r}{4})$. Then $O(w, \nicefrac{r}{4}) \sub O(x, \nicefrac{r}{2})$ and we have for all $y \in O(x, \nicefrac{r}{2})$ and $z \in O(w, \nicefrac{r}{4})$ that
		\begin{equation*}
			|u(y)| \leq r^{\mu}[u]^{(\mu)}_{O(x, \frac{r}{2})} + |u(z)|.
		\end{equation*}
		Now, we average with respect to $z$ on $O(w, \nicefrac{r}{4})$ and conclude as in the first case. Note that $w \in N$, so we have~\hyperref[The geometric setup: Proposition: (ICC)]{$\ICC$} at our disposal.
	\end{proof}

	\begin{remark} \label{(D(mu), G(mu) and H(mu) for MBC: Remark: H(mu) Scaling}
		Using Lemma~\ref{(D(mu), G(mu) and H(mu) for MBC: Lemma: Add L-infinity part in H(mu)}, the following modifications can be made in Definition~\ref{D(mu), G(mu) and H(mu) for MBC: Definition: H(mu)} and Lemma~\ref{(D(mu), G(mu) and H(mu) for MBC: Lemma: Add L-infinity part in H(mu)}.

		\begin{enumerate}
			\item It is possible to replace the condition $r \in (0,1]$ by $r \in (0,R]$ for any $R > 0$.

			Indeed, it suffices to consider $R > 1$ and $r \in (1, R)$. The $\L^{\infty}$-part is clear, as it is a pointwise estimate for all $y \in O(x, \nicefrac{r}{2})$. To bound the H\"older seminorm, we pick $y, z \in O(x, \nicefrac{r}{2})$. If $|y-z| \geq \nicefrac{1}{8}$, then we can use the $\L^{\infty}$-bound. If $|y-z| < \nicefrac{1}{8}$, then we can apply the estimate in $O(x, \nicefrac{1}{4})$.

			\item By the same type of argument, the radius $\nicefrac{r}{2}$ on the left-hand side of \eqref{eq: Definition: Property H(mu)} can be replaced by $\gamma r$ for any $\gamma \in (0,1)$.
		\end{enumerate}
	\end{remark}

	Next, we discuss the solvability of the local problem $L_D u = f - \Div F$ in $O(x,R)$ with an a priori bound that has the correct scaling in $R$.
	To see that this does not come for free, we consider a simple counterexample.

	Fix $r \in (0,1]$ and put $O_{r} \coloneqq B(0,r) \cup B(4 e_1, 1)$ as the union of two disjoint balls. We impose Neumann boundary conditions on $\partial B(0, r)$ and Dirichlet boundary conditions on $\partial B (4 e_1, 1)$. Take some $f \in \L^2(O_{r})$ that is not average free over $B(0,r)$. Choosing $R = 2 $, the local problem $L_D u = f$ in $O_{r}(0, R) = B(0, r)$ cannot have a solution as we can take the constant $1$-function as a test function.

	However, changing the radius from $R = 2$ to $\rho = \nicefrac{r}{2}$ yields a pure Dirichlet problem $L_D u = f$ in $O_r(0, \rho) = B(0, \rho)$, which admits a unique solution by the Lax-Milgram lemma. The correct scaling in $\rho$ in the a priori estimate comes from the classical Poincar\'{e} inequality on balls. Now, the key observation is that our geometric setup does not allow that $r$ shrinks to $0$ (see Definition~\ref{The geometric setup: Definition (LU)}). This ensures that the ratio $\nicefrac{\rho}{R}$ is bounded from below. We will need this fact in Section~\ref{Section: (1) to (2)}.

	Getting the correct scaling in $R$ can be more difficult. Here, the geometry has to ensure that the local Dirichlet part $\partial O(x, R) \setminus N(x,R)$ is large enough in a suitable sense.

	The key point in the next lemma is the following: even when we cannot solve every local problem with an a priori bound that has the correct scaling in $R$, we can use our geometric setup to do it for some smaller radius $\rho$ still comparable to $R$.

	\begin{lemma} \label{D(mu), G(mu) and H(mu) for MBC: Lemma: Existence of weak solutions + A prior bound}
		Assume~\hyperref[The geometric setup: Definition (Fat)]{$\Fat$} and~\hyperref[The geometric setup: Definition (LU)]{$\LU$}. Let $x \in \overline{O}$, $r \leq r_0$, $f \in \L^2(O(x,r))$ and $F\in \L^2(O(x,r))^d$. There is some $\rho \in [\nicefrac{r}{4}, r]$ such that the problem $L_D v = f - \Div F$ in $O(x, \rho)$ has a unique weak solution $v \in \H^1_{\partial O(x, \rho) \setminus N(x,\rho)}(O(x, \rho))$ that satisfies
		\begin{equation}  \label{eq: D(mu), G(mu), H(mu): a priori bound}
			\| \nabla v \|_{\L^2(O(x, \rho))} \les \rho \| f \|_{\L^2(O(x,\rho))} + \| F \|_{\L^2(O(x,\rho))}
		\end{equation}
		with an implicit constant depending on $\lambda$ and geometry.
	\end{lemma}

	\begin{proof}
		The proof is divided into three main cases. We abbreviate
		\begin{equation*}
			V_{\rho} \coloneqq \H^1_{\partial O(x, \rho) \setminus N(x, \rho)}(O(x, \rho)).
		\end{equation*}
		The main issue, as seen from the example above, is the coercivity of the form on $V_{\rho}$.

		\textbf{(1) $\boldsymbol{r < \d_{\partial O}(x)}$.} Then $O(x, r) = B(x, r)$ and we have Poincar\'{e}'s inequality on $V_r = \H^1_0(B(x, r))$ at our disposal. Hence, the result for $\rho = r$ follows from the Lax--Milgram lemma.

		\textbf{(2) $\boldsymbol{\d_D(x) \leq r}$.} Lemma~\ref{D(mu), G(mu) and H(mu) for MBC: Lemma: Sobolev extension by 0} (i) joint with~\hyperref[The geometric setup: Proposition: (P)]{$\AssP$} implies the Poincar\'{e} inequality
		\begin{equation}
			\| \varphi \|_{\L^2(O(x, r))} \les r \| \nabla \varphi_0 \|_{\L^2(O(x, c_1 r))} = r \| \nabla \varphi \|_{\L^2(O(x, r))} \qquad (\varphi \in V_r) \label{eq: Geometric Assumption: Existence of weak solutions, Poincare}
		\end{equation}
		and we conclude as in the first case.

		\textbf{(3) $\boldsymbol{r < \d_D(x)}$ and $\boldsymbol{\d_N(x) \leq r}$.} We need to consider two subcases:

		\textbf{(3.1) $\boldsymbol{\partial B(x, r) \cap O = \varnothing}$.} Since $x \in \overline{O}$, this means that $O$ splits into two components $O = O_{\loc} \cup (O \setminus O_{\loc})$, where $O_{\loc}$ is open and contained in $B(x, r)$. As $r < \d_D(x)$, the boundary of all connected components of $O_{\loc}$ intersects $\partial O$ in $N$. Hence, $O$ has a connected component with diameter less than $2 r \leq C \delta$ that intersects $N$ in contradiction with~\hyperref[The geometric setup: Definition (LU)]{$\LU$}. Thus, this case can never occur.

		\textbf{(3.2) $\boldsymbol{\partial B(x, r) \cap O \neq \varnothing}$.} As in the second case it suffices to prove
		\begin{equation}
			\| \varphi \|_{\L^2(O(x, \rho))} \les \rho \| \nabla \varphi \|_{\L^2(O(x, \rho))} \qquad (\varphi \in V_\rho) \label{eq: Geometric Assumption: Existence of weak solutions, Poincare with rho}
		\end{equation}
		for some $\rho \in [\nicefrac{r}{4}, r]$. By the argument in (3.1) we can assume that there is some $y \in \partial B(x, \nicefrac{r}{2}) \cap O \sub N_{\nicefrac{\delta}{2}} \cap O$.

		Now, our goal is to show that there is a radius $\rho$ (comparable to $r$) such that $\partial O(x,\rho)$ carries a large portion of Dirichlet boundary conditions (not necessarily coming from $D$). For this we will find a ball $B(z, \nicefrac{\alpha \rho}{4} )$ that lies inside $O$ with center $z \in \partial B(x, \rho)$: By~\hyperref[The geometric setup: Proposition: (ICC)]{$\ICC$} there is some $z$ with $B(z, \nicefrac{\alpha r}{4}) \sub O(y, \nicefrac{r}{4})$. We set $\rho \coloneqq |z-x|$ so that $\rho \in [\nicefrac{r}{4}, \nicefrac{3 r}{4}]$. Notice that $z \in \partial B(x, \rho)$ and $B(z, \nicefrac{\alpha \rho}{4}) \sub O$, which is exactly what we need (see Figure~\ref{fig. D(mu), G(mu) and H(mu): (ICC) ensures large Dirichlet part}).

		\begin{figure}
			\centering
			\begin{tikzpicture}
				\draw (0,0) circle (4cm);
				\draw (0,0) circle (2cm);
				\draw (1.42,-1.42) circle (1cm);
				\draw (1.22, -1.02) circle (0.42 cm);
				\draw (0,0) circle (1.59 cm);
				\draw[black, domain=-4:4, samples=100] plot (\x, {\x*\x*\x/15 - \x*\x/15 + 0.5});
				\node[inner sep=1.5pt, fill=black, circle] at (0,0) (x) {};
				\node[inner sep=1.5pt, fill=black, circle] at (1.42,-1.42) (y) {};
				\node[inner sep=1.5pt, fill=black, circle] at (1.22,-1.02) (z) {};
				\node[below] at (0,-2.5) {$O$};
				\node[above] at (0,2.5) {$O^c$};
				\node[right] at (x.south) {$x$};
				\node[right] at (y.south) {$y$};
				\node[left] at (z.south) {$z$};
				\draw (0,0) -- (0,1.59) node[midway, right] {$\rho$};
				\draw (0,0) -- (0,-2) node[midway, left] {$\frac{r}{2}$};
				\draw (0,0) -- (-4,0) node[pos=0.7, above] {$r$};
			\end{tikzpicture}
			\caption{Geometric configuration, where $\partial O(x, \rho) \setminus N(x,\rho)$ is large enough.}
			\label{fig. D(mu), G(mu) and H(mu): (ICC) ensures large Dirichlet part}
		\end{figure}
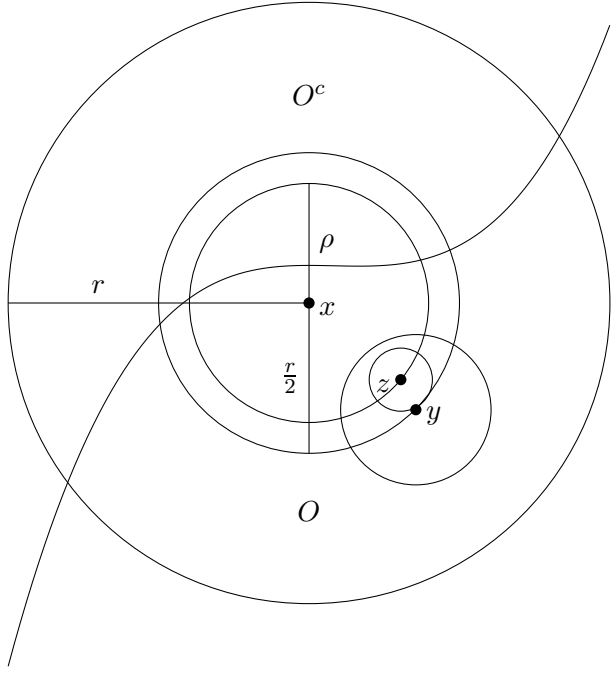

		Let now $\varphi \in \rC^{\infty}_{\partial O(x, \rho) \setminus N(x, \rho)}(O(x, \rho))$ and extend $\varphi$ by $0$ to $O$ (see Lemma~\ref{D(mu), G(mu) and H(mu) for MBC: Lemma: Sobolev extension by 0} (i)). Then
		\begin{equation*}
			\| \varphi \|_{\L^2(O(x, \rho))} \leq \| \cE \varphi_0 \|_{\L^2(B(x, \rho))}.
		\end{equation*}
		Note that $\cE \varphi_0 \in \H^1(B(x, \rho))$ vanishes on $\partial B(x, \rho) \cap B(z, \nicefrac{\alpha \rho}{4})$ and thus we have
		\begin{equation*}
			\| \varphi \|_{\L^2(O(x, \rho))} \les \rho \| \nabla \cE \varphi_0 \|_{\L^2(B(x, \rho))}.
		\end{equation*}
		Indeed, when $x = 0$, $\rho =1$ and $z$ is the north pole of $B(x,\rho)$, then this Poincar\'{e} inequality follows by compactness and then we can use scaling and a rigid motion. Finally, we use that $\cE$ is local and homogeneous in order to derive
		\begin{equation*}
			\| \nabla \cE \varphi_0 \|_{\L^2(B(x, \rho))} \les \| \nabla \varphi_0 \|_{\L^2(O(x, c_1 \rho))} = \| \nabla \varphi \|_{\L^2(O(x, \rho))} .
		\end{equation*}
		The combination of these three estimates proves \eqref{eq: Geometric Assumption: Existence of weak solutions, Poincare with rho} and completes the proof.
	\end{proof}

	We will show Theorem~\ref{Theorem: Main result of the paper} by proving the implications (i) $\Longrightarrow$ (ii), (ii) $\Longrightarrow$ (iii), and (iii) $\Longrightarrow$ (i) in this order as in \cite{Auscher_Tchamitchian_Kato}. This is the content of the following three sections.

	\section{From $\rD(\mu)$ to $\rG(\mu)$} \label{Section: (1) to (2)}

	In this section we prove the implication $\mathrm{(i)} \Rightarrow \mathrm{(ii)}$ of Theorem~\ref{Theorem: Main result of the paper}. Throughout the entire section we make the geometric assumptions~\hyperref[The geometric setup: Definition (Fat)]{$\Fat$} and~\hyperref[The geometric setup: Definition (LU)]{$\LU$}, see Definitions~\ref{The geometric setup: Definition (Fat)} and \ref{The geometric setup: Definition (LU)}, and our goal is thus to show:

	\begin{theorem} \label{(1) to (2): Main result}
		Let $L$ and $L^*$ have property $\rD(\mu_0)$ and fix $\mu \in (0, \mu_0)$. Then $L$ has property $\rG(\mu)$ with implicit constants depending on geometry, ellipticity and $[c_{\rD(\mu)}, \mu, \mu_0]$.
	\end{theorem}

	To prove this result, we use $\rD(\mu_0)$ to obtain semigroup bounds in $\L^{\infty}$ and $\Cdot^{\mu}$ with correct scaling. Once in place, existence of the kernel with estimates will follow from the Dunford--Pettis theorem. This will be explained at the end of this section. Compared to \cite{ter-Elst_Rehberg, Auscher_Tchamitchian_Kato}, we interpolate the $\L^{\infty}$ and $\Cdot^{\mu}$-bounds for the semigroup with the $\L^2$ off-diagonal estimates from Proposition~\ref{Operator theoretic setting and relevant function spaces: Proposition: L2 ODE} instead of directly applying Davies' perturbation method. This provides a much shorter and streamlined argument, since it does not produce lower order perturbations for the divergence form operator.

	To bound $\e^{-t L}$ in $\L^{\infty}$ and $\Cdot^{\mu}$, we use Morrey and Campanato spaces and bootstrap regularity. Let us introduce them properly and refer to \cite[Chap.~3]{Giaquinta} or \cite{Campa} for more information.

	\begin{definition}
		Let $\kappa \in [0,d)$, $r > 0$ and $x \in \overline{O}$. We define the \textbf{local Morrey space}
		\begin{equation*}
			\L^{\kappa}_{x, r} (O) \coloneqq \left\{ u \in \L^2(O) \colon \| u \|_{\L^{\kappa}_{x,r}(O)} \coloneqq \sup_{\rho \in (0, r]} \rho^{- \frac{\kappa}{2}} \| u \|_{\L^2(O(x,\rho))} < \infty \right\}.
		\end{equation*}
	\end{definition}

	\begin{definition}
		Let $\kappa \in [0,d+2)$, $r > 0$ and $x \in \overline{O}$. We define the \textbf{local Campanato space}
		\begin{equation*}
			\cL^{\kappa}_{x, r} (O) \coloneqq \left\{ u \in \L^2(O) \colon [ u ]_{\cL^{\kappa}_{x,r}(O)} \coloneqq \sup_{\rho \in (0, r]} \rho^{- \frac{\kappa}{2}} \| u - (u)_{O(x,\rho)} \|_{\L^2(O(x,\rho))} < \infty \right\},
		\end{equation*}
		and set
		\begin{equation*}
			\| u \|_{\cL^{\kappa}_{x, r}(O)} \coloneqq \| u \|_{\L^2(O(x,r))} + [ u ]_{\cL^{\kappa}_{x,r}(O)}.
		\end{equation*}
	\end{definition}

	\textbf{Convention.} We abbreviate $\| \cdot \|_{\L^{\kappa}_{x,r}(O)} \eqqcolon \| \cdot \|_{\kappa, x,r}$ and drop the dependence on $x$ whenever the context is clear. We also write
	\begin{equation*}
		\| \cdot \|_{\cL^{\kappa}_{x,r}(O)} \eqqcolon \| \cdot \|_{\cL^{\kappa}_{x,r}} \eqqcolon \| \cdot \|_{\cL^{\kappa}_r},
	\end{equation*}
	and we use the same abbreviations for $[ \, \cdot \, ]_{\cL^{\kappa}_{x,r}(O)}$.

	The following results from \cite[Lem.~3.1]{ter-Elst_Rehberg} are central.

	\begin{lemma} \label{(1) to (2): Lemma: Classical Morrey and Campanato Lemma}
		Let $\gamma \in (0,1)$, $r \in (0,1]$, $\kappa \in [0, d +2)$ and $x,y \in \overline{O}$ with $|O(x,\rho)| \land |O(y,\rho)| \geq \gamma |B(x,\rho)|$ for all $\rho \in (0,1]$. The following properties hold true with implicit constants depending only on $[d, \gamma, \kappa]$:

		\begin{enumerate}
			\item[$\mathrm{(i)}$] If $\kappa < d$, then $\L^{\kappa}_{x,r}(O) \cong \cL^{\kappa}_{x,r}(O)$ with estimate
			\begin{equation*}
				[u]_{\cL_{x, r}^{\kappa}} \leq \| u \|_{\kappa, x, r} \les r^{- \frac{d}{2}} \| u \|_{\L^2(O(x,r))} + [u ]_{\cL_{x,r}^{\kappa}} \qquad (u \in \L_{x,r}^{\kappa}(O)).
			\end{equation*}
	
			\item[$\mathrm{(ii)}$] If $\kappa > d$ and $u \in \cL^{\kappa}_{x,r}(O)$, then $u(x) \coloneqq \lim_{\rho \searrow 0} (u)_{O(x, \rho)}$ exists and
			\begin{equation*}
				|u(x) - (u)_{O(x,\rho)}| \les \rho^{\frac{\kappa -d}{2}} [ u ]_{\cL^{\kappa}_{x,\rho}} \qquad (\rho \in (0,r]).
			\end{equation*}

			\item[$\mathrm{(iii)}$] If $\kappa > d$, $|x-y| \leq \nicefrac{r}{2}$ and $u \in \cL^{\kappa}_{x,r}(O) \cap \cL^{\kappa}_{y,r}(O)$, then
			\begin{equation*}
				|u(x) - u(y)| \les ( [ u ]_{\cL^{\kappa}_{x,r}} + [u]_{\cL^{\kappa}_{y,r}} ) |x-y|^{\frac{\kappa - d}{2}}.
			\end{equation*}
		\end{enumerate}
	\end{lemma}

	We begin with gradient estimates for global weak solutions in Morrey spaces. Later, we will apply these estimates iteratively to $u_t = \e^{- t L} u$ with $u \in \L^2(O)$.

	\begin{proposition}  \label{(1) to (2): Proposition: A priori bound for weak solution in Morrey spaces}
		Assume that $L$ has property $\rD(\mu)$. Let $\kappa \in [0,d)$, $\sigma \in (0,2]$ with $\kappa + \sigma < d-2 +2 \mu$, $x \in \overline{O}$, $R_0 \in (0, r_0]$ and $f \in \L^{\kappa}_{x, R_0}(O)$. Then we have for all $u \in \H_D^1(O)$ with $L_Du = f$ in $O(x,R_0)$ and $\varepsilon \in (0,1]$ the estimate
		\begin{equation*}
			\| \nabla u \|_{\kappa + \sigma, x, R_0} \les \varepsilon^{2- \sigma} \| f \|_{\kappa, x, R_0} + \varepsilon^{-(\kappa + \sigma)} \| \nabla u \|_{\L^2(O(x,R_0))},
		\end{equation*}
		where the implicit constant depends only on $[\lambda, c_{\rD(\mu)}, \mu, \kappa, \sigma, R_0]$ and geometry.
	\end{proposition}

	Before we start the proof, we recall the classical Campanato lemma.

	\begin{lemma}[{\cite[Chap.~3, Lem.~2.1]{Giaquinta}}]\label{Appendix: Lemma: Campanato}
		Let $\phi \colon \R \rightarrow [0, \infty)$ be non-decreasing, and let $C_1, C_2, R_0, \varepsilon \geq 0$ as well as $0 \leq \beta < \gamma$. If we have
		\begin{equation}\label{CampaEstimate}
			\phi(r) \leq C_1 \bigg[ \bigg( \frac{r}{R} \bigg)^{\gamma}+\varepsilon \bigg] \phi(R)+ C_2 R^{\beta} \qquad (0< r \leq R \leq R_0),
		\end{equation}
		then there exists $\varepsilon_0,c >0$ depending only on $[C_1, \gamma, \beta]$ such that if $\varepsilon<\varepsilon_0$, then
		\begin{equation}\label{CampaConseguence}
			\phi(r) \leq c \bigg[ \bigg( \frac{r}{R} \bigg)^{\beta} \phi(R) + C_2 r^{\beta }\bigg] \qquad (0 < r \leq R \leq R_0).
		\end{equation}
	\end{lemma}

	\begin{proof}[\rm\bf{Proof of Proposition~\ref{(1) to (2): Proposition: A priori bound for weak solution in Morrey spaces}.}]
		Let $0 < r \leq R \leq R_0 \varepsilon^2$ and define the function
		\begin{equation*}
			\phi(r) \coloneqq \int_{O(x,r)} |\nabla u|^2.
		\end{equation*}
		We pick $\rho \in [\nicefrac{R}{4}, R]$ as in Lemma~\ref{D(mu), G(mu) and H(mu) for MBC: Lemma: Existence of weak solutions + A prior bound}, so that there is some $v \in \H_{\partial O(x,\rho) \setminus N(x, \rho)}^1(O(x,\rho))$ such that $L_D v = f$ in $O(x, \rho)$. It also satisfies the a priori bound \eqref{eq: D(mu), G(mu), H(mu): a priori bound}, which implies
		\begin{equation}  \label{eq: D(mu) to G(mu): A priori estimate in Morrey space}
			\int_{O(x, \rho)} |\nabla v|^2 \les R^{\kappa +2} \| f \|_{ \kappa, R}^2 \leq (R_0 \varepsilon^2)^{2 - \sigma} \| f \|_{\kappa, R_0}^2 R^{\kappa + \sigma}.
		\end{equation}
		By Lemma~\ref{D(mu), G(mu) and H(mu) for MBC: Lemma: Sobolev extension by 0} (i) we can extend $v$ by $0$ and view it as an element of $\H_D^1(O)$. Then $w \coloneqq u -v \in \H_D^1(O)$ satisfies $L_D w =0$ in $O(x, \rho)$. Provided that $r \leq \rho$, we can use property $\rD(\mu)$ to get
		\begin{align}
			\phi(r) &\les \int_{O(x,r)} |\nabla v|^2 + \int_{O(x,r)} |\nabla w|^2  \notag
			\\&\les \int_{O(x,r)} |\nabla v|^2 + \left( \frac{r}{\rho} \right)^{d-2+2\mu}  \int_{O(x,\rho)} |\nabla w|^2  \notag
			\\&\les \int_{O(x,\rho)} |\nabla v|^2 + \left( \frac{r}{\rho} \right)^{d -2 +2 \mu} \phi(\rho). \label{eq: (1) to (2): Proposition: A priori bound for weak solution in Morrey spaces}
		\end{align}
		Inserting \eqref{eq: D(mu) to G(mu): A priori estimate in Morrey space} into \eqref{eq: (1) to (2): Proposition: A priori bound for weak solution in Morrey spaces} and using $\nicefrac{R}{4} \leq \rho \leq R$ delivers
		\begin{equation*}
			\phi(r) \les \left( \frac{r}{R} \right)^{d-2 +2 \mu} \phi(R) + (R_0 \varepsilon^2)^{2 - \sigma} \| f \|_{\kappa, R_0}^2 R^{\kappa + \sigma}.
		\end{equation*}
		If $r \geq \rho$, then $\nicefrac{r}{R} \geq \nicefrac{1}{4}$ and the same estimate follows by monotonicity of $\phi$, even without the semi-norm of $f$. Lemma~\ref{Appendix: Lemma: Campanato} improves this bound to
		\begin{equation*}
			\phi(r) \les \left( \frac{r}{R} \right)^{\kappa + \sigma} \phi(R) + (R_0 \varepsilon^2)^{2 - \sigma} \| f \|_{\kappa, R_0}^2  r^{\kappa + \sigma},
		\end{equation*}
		with an implicit constant depending only on $[c_{\rD(\mu)}, \mu, \kappa, \sigma]$ and geometry. Hence, we get for all $r, R$ as above that 
		\begin{equation*}
			r^{- \frac{1}{2} (\kappa + \sigma)} \phi(r)^{\frac{1}{2}} \les R^{- \frac{1}{2}(\kappa + \sigma)} \| \nabla u \|_{\L^2(O(x,R_0))} + (R_0 \varepsilon^2)^{1 - \frac{\sigma}{2}} \| f \|_{\kappa, R_0}.
		\end{equation*}
		In particular, if we pick $R \coloneqq R_0 \varepsilon^2$, then we get for $0 < r \leq R_0 \varepsilon^2$ the estimate
		\begin{equation*}
			r^{- \frac{1}{2} (\kappa + \sigma)} \phi(r)^{\frac{1}{2}} \les R_0^{- \frac{1}{2}(\kappa + \sigma)} \varepsilon^{-(\kappa + \sigma)} \| \nabla u \|_{\L^2(O(x,R_0))} + R_0^{1 - \frac{\sigma}{2}} \varepsilon^{2 - \sigma} \| f \|_{\kappa, R_0}.
		\end{equation*}
		As before, this estimate remains valid for $R_0 \varepsilon^2 < r \leq R_0$ by monotonicity of $\phi$. Taking the supremum in $r \leq R_0$ yields the claim.
	\end{proof}

	\begin{lemma} \label{(1) to (2): Lemma: Campanato norm controlled}
		Let $\kappa \in [0,d)$ and $\sigma \in (0,2]$. There is some $c > 0$ depending only on geometry, $\kappa$ and $\sigma$ such that we have for all $\varepsilon \in (0,1]$, $u \in \H_D^1(O)$ and $x \in \overline{O}$ that
		\begin{equation*}
			[ u ]_{\cL^{\kappa + \sigma}_{x,r_0}} \leq c ( \varepsilon^{2- \sigma} \| \nabla u \|_{ \kappa, x, c_1 r_0} + \varepsilon^{- (\kappa + \sigma)} \| u \|_{\L^2(O(x,r_0))}).
		\end{equation*}
		Moreover, if $x \notin N_{\nicefrac{\delta}{2}}$, then the same estimate holds for $u_0$ on $B(x,r_0)$.
	\end{lemma}

	\begin{proof}
		If $r \in (0, \varepsilon^2 r_0]$, then we have by Remark~\ref{The geometric setup: Remark: (P) follows from boundary inequality} that
		\begin{align*}
			r^{- \frac{1}{2} (\kappa + \sigma)} \| u - (u)_{O(x,r)} \|_{\L^2(O(x,r))} &\les_{c_0,d} r^{1- \frac{\sigma}{2}} r^{- \frac{\kappa}{2}} \| \nabla u \|_{\L^2(O(x, c_1 r))}
			\\&\les_{c_1, \kappa, \sigma} r_0^{1- \frac{\sigma}{2}} \varepsilon^{2- \sigma} \| \nabla u \|_{\kappa, x, c_1 r_0}.
		\end{align*}
		In the other case, we get
		\begin{align*}
			r^{- \frac{1}{2}(\kappa + \sigma)} \| u - (u)_{O(x,r)} \|_{\L^2(O(x,r))} &\leq 2r^{- \frac{1}{2}(\kappa + \sigma)} \| u \|_{\L^2(O(x,r))}
			\\&\leq r_0^{- \frac{1}{2}(\kappa + \sigma)} \varepsilon^{- (\kappa + \sigma)} \| u \|_{\L^2(O(x,r_0))}
		\end{align*}
		as required.

		Finally, if $x \notin N_{\nicefrac{\delta}{2}}$, then $u_0 \in \H^1(B(x,r))$ since we have $r_0 < \nicefrac{\delta}{2}$ by definition and hence we can use the standard Poincar\'e inequality on balls in the first case.
	\end{proof}

	Next, we proceed as follows with $u_t = \e^{- t L} u$, where $u \in \L^2(O)$:

	\begin{itemize}
		\item We increase the regularity of $\nabla u_t$ in Morrey spaces up to the critical exponent $d-2+2 \mu$.

		\item We pass to an estimate for the Campanato seminorm of $u_t$ with exponent $d+2\mu$.
	\end{itemize}

	\begin{lemma} \label{(1) to (2): Lemma: Induction lemma for Gaussian estimates} Let $L$ have property $\rD(\mu_0)$ and let $\mu \in (0, \mu_0)$. There are $c, \omega > 0$ and $\gamma \in (0,1]$ depending only on geometry, ellipticity and $[c_{\rD(\mu)}, \mu, \mu_0]$ such that
		\begin{equation}
			[u_t]_{\cL_{x,\gamma r_0}^{d+ 2 \mu}} \leq c t^{- \frac{d}{4} - \frac{\mu}{2}} \e^{\omega t} \| u \|_2 \qquad (t > 0, u \in \L^2(O), x \in \overline{O}). \label{eq: (1) to (2): Induction lemma for Gaussian estimates}
		\end{equation}
		Moreover, if $x \notin N_{\nicefrac{\delta}{2}}$, then we can replace $u$ by $u_0$ on the left-hand side.
	\end{lemma}

	\begin{proof}
		Let $\kappa \in [0, d-2 +2 \mu_0)$ and $r \in (0, r_0]$. Consider the following statement:

		\textbf{$\boldsymbol{P(\kappa, r)}$.} There are $c, \omega > 0$ depending on geometry and $[\lambda, \Lambda, c_{\rD(\mu)}, \mu_0, \kappa]$ such that
		\begin{equation*}
			\| u_t \|_{\kappa, x, r} + \| \sqrt{t} \nabla u_t \|_{\kappa, x, r} \leq c t^{- \frac{\kappa}{4}} \e^{\omega t} \| u \|_2 \qquad (t > 0, u \in \L^2(O), x \in \overline{O}).
		\end{equation*}
		$P(0, r_0)$ holds true by the $\L^2$-theory in Section~\ref{Section: Operator theoretic setting and relevant function spaces}. Let $\sigma \in (0,2]$ with $\kappa + \sigma < d-2 +2 \mu_0$. We claim that
		\begin{equation*}
			P(\kappa, r) \Longrightarrow P(\kappa + \sigma, \nicefrac{r}{c_1}).
		\end{equation*}
		This will yield \eqref{eq: (1) to (2): Induction lemma for Gaussian estimates} with $\gamma \coloneqq \nicefrac{1}{c_1^{m_0+2}}$, where $m_0$ is the largest integer with $2 m_0 < d -2 + 2 \mu$, by iterating $(m_0+1)$-times and a final application of the Poincar\'e inequality in Remark~\ref{The geometric setup: Remark: (P) follows from boundary inequality}. For the additional claim when $x \notin N_{\nicefrac{\delta}{2}}$, we simply use the standard Poincar\'{e} inequality on balls in the final step as in the previous proof.

		Assume that $P(\kappa, r)$ is valid and define $\varepsilon \coloneqq t^{\nicefrac{1}{4}} \e^{-t} \in (0,1]$. We prove $P(\kappa + \sigma, \nicefrac{r}{c_1})$ in two steps.

		\textbf{Non-gradient bound.} If $x \in N_{\nicefrac{\delta}{2}}$, then we have~\hyperref[The geometric setup: Proposition: (ICC)]{$\ICC$} at hand and we can apply Lemma~\ref{(1) to (2): Lemma: Classical Morrey and Campanato Lemma} (i) and then Lemma~\ref{(1) to (2): Lemma: Campanato norm controlled} to get
		\begin{align*}
			\| u_t \|_{\kappa+ \sigma, \nicefrac{r}{c_1}} &\les \| u_t \|_{\cL^{\kappa + \sigma}_{\nicefrac{r}{c_1}}} \les (t^{\frac{1}{4}} \e^{-t} )^{2- \sigma} \| \nabla u_t \|_{ \kappa, r} + (t^{\frac{1}{4}} \e^{-t} )^{- (\kappa + \sigma)} \| u_t \|_{2}.
		\end{align*}
		Using $P(\kappa, r)$ and the $\L^2$-contractivity of the semigroup gives us
		\begin{equation*}
			\| u_t \|_{\kappa+ \sigma, \nicefrac{r}{c_1}} \les (t^{- \frac{\kappa + \sigma}{4}} \e^{\omega  t} + t^{- \frac{\kappa + \sigma}{4}} \e^{(d+2) t}) \| u \|_2 \les t^{- \frac{\kappa + \sigma}{4}}  \e^{(\omega \lor (d+2)) t} \| u \|_2.
		\end{equation*}
		If $x \notin N_{\nicefrac{\delta}{2}}$, then we can do the first step for the $0$-extension on $B(x,r)$, to which Lemma~\ref{(1) to (2): Lemma: Campanato norm controlled} applies as well.

		\textbf{Gradient bound.} Proposition~\ref{(1) to (2): Proposition: A priori bound for weak solution in Morrey spaces} with $f = L u_t = \e^{\nicefrac{-t}{2} L} L u_{\nicefrac{t}{2}}$ reveals the bound
		\begin{align*}
			\| \sqrt{t} \nabla u_t \|_{\kappa + \sigma, \nicefrac{r}{c_1}}
			&\les \sqrt{t} (t^{\frac{1}{4}} \e^{-t})^{2 - \sigma} \| \e^{- \frac{t}{2} L} L u_{\nicefrac{t}{2}} \|_{ \kappa, \nicefrac{r}{c_1}} + (t^{\frac{1}{4}} \e^{-t})^{- (\kappa + \sigma)} \| \sqrt{t} \nabla u_t \|_{2}
			\\& \les t^{- \frac{\sigma}{4}} \| \e^{- \frac{t}{2} L} ( t L u_{\nicefrac{t}{2}} ) \|_{\kappa, \nicefrac{r}{c_1}} + t^{- \frac{\kappa + \sigma}{4}} \e^{(d + 2) t} \| \sqrt{t} \nabla u_t \|_{2}.
		\end{align*}
		Next, we use $P(\kappa, r)$ joint with the bound $\| t L u_{\nicefrac{t}{2}} \|_2 \les \| u \|_2$ for analytic semigroups for the first summand and the estimate $\| \sqrt{t} \nabla u_t \|_2 \les \| u \|_2$ from ellipticity (or Proposition~\ref{Operator theoretic setting and relevant function spaces: Proposition: L2 ODE}) for the second one to deduce
		\begin{equation*}
			\| \sqrt{t} \nabla u_t \|_{\kappa + \sigma, \nicefrac{r}{c_1}} \les t^{-\frac{\kappa + \sigma}{4}} \e^{(\frac{\omega}{2} \lor (d+2)) t} \| u \|_2. \qedhere
		\end{equation*}
	\end{proof}

	Finally, we can prove:

	\begin{proposition} \label{(1) to (2): Proposition: Almost Gaussian estimates}
		Assume that $L$ has property $\rD(\mu_0)$ and fix $\mu \in (0, \mu_0)$. Then there are $c, \omega > 0$ depending on geometry, ellipticity and $[c_{\rD(\mu)}, \mu, \mu_0]$ such that
		\begin{equation} \label{eq: L2 - Linfty + Holder estimate}
			\| u_t \|_{\infty} + t^{\frac{\mu}{2}}  [u_t]_O^{(\mu)} \leq c \e^{\omega t} t^{- \frac{d}{4}} \| u \|_2 \qquad (t > 0, u \in \L^2(O)).
		\end{equation}
	\end{proposition}

	\begin{proof}
		We prove the two bounds separately.

		\textbf{$\boldsymbol{\L^{\infty}}$-bound.} Let $r \coloneqq \gamma r_0 \sqrt{t} \e^{-t} \leq \gamma r_0$ and fix $x \in O$.

		\textbf{(1) $\boldsymbol{x \in O \cap N_{\nicefrac{\delta}{2}}}$.} Due to~\hyperref[The geometric setup: Proposition: (ICC)]{$\ICC$}, we can apply Lemma~\ref{(1) to (2): Lemma: Classical Morrey and Campanato Lemma} (ii) to get
		\begin{align*}
			|u_t(x)| &\les r^{\mu} [ u_t ]_{\cL^{d+2 \mu}_{x,r}} + |(u_t)_{O(x,r)}| \leq
			r^{\mu} [ u_t ]_{\cL^{d+2 \mu}_{x,r}} + |O(x,r)|^{-\frac{1}{2}} \| u_t \|_2.
		\end{align*}
		Lemma~\ref{(1) to (2): Lemma: Induction lemma for Gaussian estimates} controls the first summand and~\hyperref[The geometric setup: Proposition: (ICC)]{$\ICC$} together with the contractivity of the semigroup the second one:
		\begin{equation*}
			|u_t(x)| \les t^{- \frac{d}{4}} \e^{- \mu t} \e^{\omega t} \| u \|_2 + r^{- \frac{d}{2}} \| u \|_{2} \les t^{- \frac{d}{4}} \e^{(\omega \lor \frac{d}{2}) t} \| u \|_2.
		\end{equation*}
		\textbf{(2) $\boldsymbol{x \in O \setminus N_{\nicefrac{\delta}{2}}}$.} The argument is the same upon working with the $0$-extension of $u_t$ on $B(x, r)$ instead of $u_t$ on $O(x,r)$.

		\textbf{$\boldsymbol{\Cdot^{\mu}}$-bound.} Fix $x,y \in O$. For $|x-y| > \nicefrac{\gamma r_0}{2}$ we use the $\L^{\infty}$-bound and that $t^{\nicefrac{\mu}{2}} \leq \e^{\nicefrac{\mu t}{2}}$:
		\begin{equation*}
			t^{\frac{\mu}{2}} \frac{|u_t(x)-u_t(y)|}{|x-y|^\mu} \lesssim \e^{\frac{\mu}{2} t} \| u_t\|_\infty \leq c \e^{(\omega +\frac{\mu}{2}) t} t^{- \frac{d}{4}} \| u \|_2.
		\end{equation*}
		Now, let $|x-y| \leq \nicefrac{\gamma r_0}{2}$. We distinguish three cases.

		\textbf{(1) $\boldsymbol{x,y \in O \cap N_{\nicefrac{\delta}{2}}}$.} Lemma~\ref{(1) to (2): Lemma: Classical Morrey and Campanato Lemma} (iii) joint with Lemma~\ref{(1) to (2): Lemma: Induction lemma for Gaussian estimates} gives
		\begin{equation*}
			|u_t(x) - u_t(y) | \les |x-y|^{\mu} \big( [u_t]_{\cL^{d+2 \mu}_{x,\gamma r_0}} + [u_t]_{\cL^{d+2 \mu}_{y,\gamma r_0}} \big) \les |x-y|^{\mu} t^{- \frac{d}{4} - \frac{\mu}{2}} \e^{\omega t} \| u \|_2.
		\end{equation*}
		\textbf{(2) $\boldsymbol{x,y \in O \setminus N_{\nicefrac{\delta}{2}}}$.} This case is again identical to the first one upon working with the $0$-extension $(u_t)_0$.

		\textbf{(3) $\boldsymbol{x \in O \cap N_{\nicefrac{\delta}{2}}}$ and $\boldsymbol{y \in O \setminus N_{\nicefrac{\delta}{2}}}$.} The proof of \cite[Lem.~3.1]{ter-Elst_Rehberg} easily reveals the more precise estimate
		\begin{equation*}
			|u_t(x) - u_t(y)| = |u_t(x) - (u_t)_0(y)| \les |x-y|^{\mu} \big( [u_t]_{\cL^{d+2 \mu}_{x,\gamma r_0}} + [(u_t)_0]_{\cL^{d+ 2 \mu}_{y,\gamma r_0}(B(y,\gamma r_0))} \big),
		\end{equation*}
		which is enough to conclude once more by Lemma~\ref{(1) to (2): Lemma: Induction lemma for Gaussian estimates}.
	\end{proof}

	We come to the main result, Theorem~\ref{(1) to (2): Main result}. For this we need a criterion to decide when a linear operator is given by a measurable kernel, known as the Dunford--Pettis theorem.

	\begin{theorem}[Dunford--Pettis, {\cite[Thm.~1.3]{Dunford-Pettis}}] \label{D(mu) to G(mu): Theorem: Dunford-Pettis}
		Let $p \in [1, \infty)$. The map
		\begin{equation*}
			\L^{\infty}(O; \L^{p'}(O)) \to \cL(\L^p(O), \L^{\infty}(O)), \quad K \mapsto \left( f \mapsto \int_O K(\cdot, y) f(y) \, \d y \right)
		\end{equation*}
		is an isometric isomorphism.
	\end{theorem}

	Armed with this result, we are going to use off-diagonal estimates as follows:

	\begin{corollary} \label{D(mu) to G(mu): Corollary: G(mu) equivalent to L2-Linfty-ODE + L2-Holder}
		Consider the following two statements:
		\begin{enumerate}
			\item There are $p \in [1,2]$, $\mu \in (0,1]$ and $c, \omega > 0$ such that
			\begin{equation} \label{eq: D(mu) to G(mu): Reduction L2-Linfty_ODE}
				\| \1_F \e^{-t L} \1_E u \|_{\infty}  + \| \1_F \e^{-t L^*} \1_E u \|_{\infty} \les \e^{\omega t} t^{- \frac{d}{2p}} \e^{-c \frac{\d(E,F)^2}{t}} \| \1_E u \|_p,
			\end{equation}
			for all measurable sets $E, F \sub O$, $t > 0$ and $u \in \L^2(O)$ as well as
			\begin{equation} \label{eq: D(mu) to G(mu): Reduction L2-Cmu}
				[ \e^{-t L} u]_O^{(\mu)} +  [ \e^{-t L^*} u]_O^{(\mu)} \les \e^{\omega t} t^{- \frac{\mu}{2} - \frac{d}{2p}} \| u \|_p \qquad (t > 0, u \in \L^2(O)).
			\end{equation}

			\item $L$ has property $\rG(\mu)$.
		\end{enumerate}
		Then $\mathrm{(i)}$ implies $\mathrm{(ii)}$ and, conversely, $\mathrm{(ii)}$ implies $\mathrm{(i)}$ for every $p \in [1,2]$ and any $\nu \in (0, \mu)$ in place of $\mu$.
	\end{corollary}

	The result is not particularly deep, but we believe that the precise formulation and the flexibility coming from the exponent $p$ will also be useful for other applications.

	\begin{proof}
		\textbf{(i) $\boldsymbol{\Longrightarrow}$ (ii):} The assumption with $E=F=O$ and Theorem~\ref{D(mu) to G(mu): Theorem: Dunford-Pettis} imply that $\e^{-tL}$ is given by a measurable kernel, $(\e^{-tL}f)(x) = \int_O K_t(x,y)f(y) \, \d y$ say, which is (G1). But this does not yet give the desired pointwise estimates.

		To this end, we use the change-of-exponents formulas from \cite[Chap.~4]{Auscher-Egert} for this type of off-diagonal estimates and the semigroup law to obtain from \eqref{eq: D(mu) to G(mu): Reduction L2-Linfty_ODE} and \eqref{eq: D(mu) to G(mu): Reduction L2-Cmu} the following two estimates:\footnote{In the language of \cite{Auscher-Egert} we start with $\L^p - \L^{\infty}$, go to $\L^2 - \L^{\infty}$ (\cite[Rem.~4.8 \& Lem.~ 4.14]{Auscher-Egert}) and $\L^1 - \L^2$ (duality), and finally to $\L^1 - \L^{\infty}$ (\cite[Lem.~4.6]{Auscher-Egert}).} There are $c, \omega > 0$ such that for all measurable sets $E, F \sub O$, $t > 0$ and $u \in \L^1(O) \cap \L^2(O)$ we have
		\begin{align}   \label{eq: D(mu) to G(mu): Reduction L1-Linfty-ODE}
			\| \1_F \e^{-t L} \1_E u \|_{\infty} + \| \1_F \e^{-t L^*} \1_E u \|_{\infty}
			&\les \e^{\omega t} t^{- \frac{d}{2}} \e^{-c \frac{\d(E,F)^2}{t}} \| \1_E u \|_1, \\
			\label{eq: D(mu) to G(mu): Reduction L1-Cmu}
			[ \e^{-t L} u]_O^{(\mu)} + [ \e^{-t L^*} u]_O^{(\mu)}  &\les \e^{\omega t} t^{- \frac{\mu}{2} - \frac{d}{2}} \| u \|_1.
		\end{align}
		Using the kernel and the $\L^1-\L^\infty$-duality on $E$, the first bound means that
		\begin{equation}
			\underset{x \in F,\, y\in E}{\esssup} | K_t(x,y)|\lesssim \e^{\omega t} t^{- \frac{d}{2}} \e^{-c \frac{\d(E,F)^2}{t}}.
		\end{equation}
		Taking $E$ and $F$ as balls in $O$ with small radius around $x$ and $y$ yields the pointwise bound (G2). In the same manner, the H\"older bound for the semigroup means that
		\begin{equation}
			\underset{y \in O}{\esssup} \, | K_t(x,y)-K_t(x',y)| \les \e^{\omega t} t^{- \frac{\mu}{2}-\frac{d}{2}} |x - x'|^{\mu},
		\end{equation}
		for all $x,x' \in O$, which is one half of (G3). The other half follows from the bounds for $\e^{-tL^*}$ with kernel $K^*_t(x,y)= \overline{K_t(y,x)}$.

		\textbf{(ii) $\boldsymbol{\Longrightarrow}$ (i):} Since $\rG(\mu)$ is stable under taking adjoints and $L^*$ is of the same type as $L$, it suffices to prove \eqref{eq: D(mu) to G(mu): Reduction L2-Linfty_ODE} and \eqref{eq: D(mu) to G(mu): Reduction L2-Cmu} for $L$. Given $x \in O$, we use (G2) and polar coordinates to get
		\begin{align}
			\label{eq: D(mu) to G(mu): Corollary: p'kernel bound}
			\begin{split}
				\left( \int_E |K_t(x,y)|^{p'} \, \d y \right)^{\frac{1}{p'}} &\les t^{- \frac{d}{2}} \e^{\omega t} \left( \int_{\d_E(x)}^{\infty} \e^{- p' b \frac{r^2}{t}} r^d \, \frac{\d r}{r} \right)^{\frac{1}{p'}}
				\\&= t^{- \frac{d}{2p}} \e^{\omega t} \left( \int_{\nicefrac{\d_E(x)}{\sqrt{t}}}^{\infty} \e^{- p' b r^2} r^d \, \frac{\d r}{r} \right)^{\frac{1}{p'}} \les t^{- \frac{d}{2p}} \e^{\omega t} \e^{- b \frac{\d_E(x)^2}{2 t} },
			\end{split}
		\end{align}
		with the obvious modifications when $p' = \infty$. Hence, \eqref{eq: D(mu) to G(mu): Reduction L2-Linfty_ODE} follows from Hölder's inequality. To prove \eqref{eq: D(mu) to G(mu): Reduction L2-Cmu}, we let $x \in O$ and $h \in \R^d \setminus \{ 0 \}$ with $x + h \in O$. First, we consider the case $|h| \leq \sqrt{t}$. We split the domain of integration of
		\begin{equation*}
			\left( \int_O |K_t(x+ h, y) - K_t(x,y)|^{p'} \, \d y \right)^{\frac{1}{p'}}
		\end{equation*}
		into the two regions $O(x, 2 |h|)$ and $O \setminus B(x, 2 |h|)$. On $O(x, 2 |h|)$ we invoke (G3) and obtain due to $|h| \leq \sqrt{t}$ that
		\begin{equation*}
			\left( \int_{O(x, 2 |h|)} |K_t(x+ h, y) - K_t(x,y)|^{p'} \, \d y \right)^{\frac{1}{p'}} \les \e^{\omega t} t^{- \frac{\mu}{2} - \frac{d}{2}} |h|^{\mu + \frac{d}{p'}} \leq \e^{\omega t} t^{- \frac{\mu}{2} - \frac{d}{2p}} |h|^{\mu}.
		\end{equation*}
		On $O \setminus B(x, 2 |h|)$ we use Remark~\ref{D(mu), G(mu) and H(mu) for MBC: Remark: G(mu) stability} (ii), which delivers the estimate
		\begin{align*}
			\left( \int_{O \setminus B(x, 2|h|)} |K_t(x+ h, y) - K_t(x,y)|^{p'} \, \d y \right)^{\frac{1}{p'}} &\les \e^{\omega t} t^{- \frac{\nu}{2} - \frac{d}{2}} |h|^{\nu} \left( \int_{2 |h|}^{\infty} \e^{- p' b \frac{r^2}{t}} r^d \, \frac{\d r}{r} \right)^{\frac{1}{p'}}
			\\&= \e^{\omega t} t^{- \frac{\nu}{2} - \frac{d}{2p}} |h|^{\nu} \left( \int_{\nicefrac{2 |h|}{\sqrt{t}}}^{\infty} \e^{- p' b r^2} r^d \, \frac{\d r}{r} \right)^{\frac{1}{p'}}
			\\&\les \e^{\omega t} t^{- \frac{\nu}{2} - \frac{d}{2p}} |h|^{\nu}.
		\end{align*}
		For $|h| > \sqrt{t}$ these estimates also hold, simply by \eqref{eq: D(mu) to G(mu): Corollary: p'kernel bound} and the triangle inequality. Combining the last two estimates, \eqref{eq: D(mu) to G(mu): Reduction L2-Cmu} follows again from Hölder's inequality.
	\end{proof}

	\begin{proof}[\rm\bf{Proof of Theorem~\ref{(1) to (2): Main result}.}] Of course, we base the argument on Corollary~\ref{D(mu) to G(mu): Corollary: G(mu) equivalent to L2-Linfty-ODE + L2-Holder} with $p=2$.  Since \eqref{eq: D(mu) to G(mu): Reduction L2-Cmu} has been shown in Proposition~\ref{(1) to (2): Proposition: Almost Gaussian estimates}, it remains to show \eqref{eq: D(mu) to G(mu): Reduction L2-Linfty_ODE}. Since $L$ and $L^*$ are of the same type, we only need to argue for $L$. The missing control of the $\L^\infty$-norm from the H\"older bound \eqref{eq: D(mu) to G(mu): Reduction L2-Cmu} and the $\L^2$-theory is a refined version of the proof of Lemma~\ref{(D(mu), G(mu) and H(mu) for MBC: Lemma: Add L-infinity part in H(mu)}.

		Let $r > 0$, $x \in F$ and normalize $\| \1_E u \|_2 = 1$. We have for all $y \in O(x,r)$ that
		\begin{equation*}
			|(\1_E u)_{t}(x)| \leq [(\1_E u)_{t}]_O^{(\mu)} r^{\mu} + |(\1_E u )_{t}(y)|.
		\end{equation*}
		By averaging over $O(x,r)$ with respect to $y$, using \eqref{eq: D(mu) to G(mu): Reduction L2-Cmu} for the first summand and H\"older's inequality for the second one, we deduce
		\begin{equation*}
			|(\1_E u)_{t}(x)| \les \e^{\omega t} t^{- \frac{d}{4}} \left( \frac{r}{\sqrt{t}} \right)^{\mu} + |O(x,r)|^{-\frac{1}{2}} \| \1_{O(x,r)} (\1_E u)_{t} \|_2.
		\end{equation*}
		We pick $r \coloneqq \e^{\nicefrac{- c' \d(E,F)^2}{t}} \sqrt{t}$, with $c' > 0$ to be chosen, to get
		\begin{equation}
			|(\1_E u)_{t}(x)| \les \e^{\omega t} t^{-\frac{d}{4}} \e^{- c' \mu \frac{\d(E,F)^2}{t}} +  |O(x,r)|^{-\frac{1}{2}} \| \1_{O(x,r)} (\1_E u)_t \|_2. \label{eq: D(mu) to G(mu): Final proof, 1}
		\end{equation}
		It remains to bound the second summand. Independently of our choice of $c'$ we have $\d(O(x,r), E) \geq \d(E,F) - \sqrt{t}$. Now, we consider three cases.

		\textbf{(1) $\boldmath \sqrt{t} \leq \nicefrac{\delta}{4}$ and $\boldsymbol{x \in N_{\nicefrac{\delta}{2}}}$.} Then $|O(x,r)| \simeq r^{d}$ by~\hyperref[The geometric setup: Proposition: (ICC)]{$\ICC$}. If $\sqrt{t} \leq \nicefrac{\d(E,F)}{2}$, then $\d(O(x,r),E) \geq \nicefrac{\d(E,F)}{2}$ and Proposition~\ref{Operator theoretic setting and relevant function spaces: Proposition: L2 ODE} yields
		\begin{equation*}
			|O(x,r)|^{-\frac{1}{2}} \| \1_{O(x,r)} (\1_E u)_{t} \|_2 \les r^{-\frac{d}{2}} \e^{-c \frac{\d(E,F)^2}{4t}} = t^{- \frac{d}{4}} \e^{- (\frac{c}{4} - \frac{d}{2} c') \frac{\d(E,F)^2}{t}}.
		\end{equation*}
		Therefore, we pick $c' < \nicefrac{c}{4d}$ to conclude. If $\sqrt{t} > \nicefrac{\d(E,F)}{2}$, then the same estimate holds as it is not saying anything more than $\L^2$-boundedness of the semigroup.

		\textbf{(2) $\boldmath \sqrt{t} \leq \nicefrac{\delta}{4}$ and $\boldsymbol{x \in F \setminus N_{\nicefrac{\delta}{2}}}$.} Since $r \leq \nicefrac{\delta}{4}$, we have either $B(x,r) \sub O$, in which case we can proceed as before, or $D(x,r) \neq \varnothing$. In the latter case, we pick $z \in D(x,r)$ and start over new. Namely, since we have $(\1_E u)_{t}(z) = 0$ by Lemma~\ref{(D(mu), G(mu) and H(mu) for MBC: Lemma: Pointwise real. of bdr. conditions}, we get directly that
		\begin{equation*}
			|(\1_E u)_{t}(x)| \leq [(\1_E u)_{t}]_O^{(\mu)} r^{\mu} \les  t^{-\frac{d}{4}} \e^{- c' \mu \frac{\d(E,F)^2}{t}}.
		\end{equation*}
		\textbf{(3) $\boldmath \sqrt{t} > \nicefrac{\delta}{4}$.} Let $G$ be the set of all $x \in O$ with $\d_F(x) < \nicefrac{\d(E,F)}{2}$. We split
		\begin{equation*}
			\| \1_F (\1_E u)_{t} \|_{\infty} \leq \| \e^{- \frac{\delta^2}{16} L} \1_G \e^{- (t- \frac{\delta^2}{16}) L} \1_E u \|_{\infty} + \| \1_F \e^{- \frac{\delta^2}{16} L} \1_{G^c} \e^{- (t- \frac{\delta^2}{16}) L} \1_E u \|_{\infty}.
		\end{equation*}
		Using \eqref{eq: L2 - Linfty + Holder estimate} and Proposition~\ref{Operator theoretic setting and relevant function spaces: Proposition: L2 ODE} to bound the first summand and \eqref{eq: D(mu) to G(mu): Reduction L2-Linfty_ODE} with $t = \nicefrac{\delta^2}{16}$ from the previous cases (1) and (2) for the second one, we infer
		\begin{equation*}
			\| \1_F (\1_E u)_{t} \|_{\infty} \les \e^{-c \frac{\d(E,F)^2}{4 t- \nicefrac{\delta^2}{4}}} + \e^{- \frac{4 c}{\delta^2} \, \d (E,F)^2} \leq 2\e^{- \frac{c}{4} \frac{\d(E,F)^2}{t}} \lesssim \e^{\omega t} t^{-\frac{d}{4}} \e^{- \frac{c}{4} \frac{\d(E,F)^2}{t}} .
		\end{equation*}
		This completes the proof of \eqref{eq: D(mu) to G(mu): Reduction L2-Linfty_ODE}, hence of the theorem.
	\end{proof}

	\begin{remark} \label{D(mu) to G(mu): Remark: Linfty bound only for fixed t}
		In the above proof we have used the $\L^{\infty}$-bound in \eqref{eq: L2 - Linfty + Holder estimate} only for one fixed value of $t$, namely for $t = \nicefrac{\delta^2}{16}$ in the third step of the argument. This observation will be useful in Section~\ref{Section: Property G(mu) for d=2}.
	\end{remark}

	\section{From $\rG(\mu)$ to $\H(\mu)$} \label{Section: (2) to (3)}

	We highlight that this implication does not depend on the geometry at all, which is a fundamental difference compared to the other parts of the equivalence. Heuristically, this is due to the fact that $\rG(\mu)$ is a global property and $\rD(\mu)$ and $\H(\mu)$ are local ones.

	\begin{theorem} \label{(2) to (3): Theorem: G(mu) implies H(mu)}
		Assume that $L$ has property $\rG(\mu_0)$ and let $\mu \in (0, \mu_0)$. Then $L$ and $L^*$ have property $\mathrm{H}(\mu)$.
	\end{theorem}

	The idea of the proof dates back to \cite{Auscher_Tchamitchian_Kato}. However, using the equivalent formulation of $\rG(\mu)$ from Corollary~\ref{D(mu) to G(mu): Corollary: G(mu) equivalent to L2-Linfty-ODE + L2-Holder}, we provide a shorter argument even when $O=\R^d$. Before we start with the proof, let us recall Caccioppoli's inequality for mixed boundary conditions. The proof is identical to the standard argument in e.g.\@ \cite[Lem.~16.6]{Auscher-Egert} since the test function class for $L_D u =0$ is invariant under multiplication with functions in $\smooth[\R^d]$.

	\begin{lemma}[Caccioppoli inequality]
		\label{(2) to (3): Lemma: Caccioppoli}
		Let $x \in \overline{O}$, $r > 0$, $c \in (0,1)$. If $u \in \H^1_D(O)$ solves $L_Du =0$ in $O(x,r)$, then
		\begin{equation*}
			r \| \nabla u \|_{\L^2(O(x, c r))} \les_{\lambda, \Lambda} \frac{1}{1 - c} \| u - \1_{[\d_D(x) > r]} \cdot (u)_{O(x,r)} \|_{\L^2(O(x,r))}.
		\end{equation*}
	\end{lemma}

	\begin{proof}[\rm\bf{Proof of Theorem~\ref{(2) to (3): Theorem: G(mu) implies H(mu)}.}]
		We are going to use $\rG(\mu_0)$ through the equivalent estimates \eqref{eq: D(mu) to G(mu): Reduction L2-Linfty_ODE} and \eqref{eq: D(mu) to G(mu): Reduction L2-Cmu}, see Corollary~\ref{D(mu) to G(mu): Corollary: G(mu) equivalent to L2-Linfty-ODE + L2-Holder}. Since $\rG(\mu_0)$ is stable under taking adjoints, it suffices to show $\mathrm{H}(\mu)$ for $L$.

		Let $x \in \overline{O}$, $r \leq 1$ and $u \in \H_D^1(O)$ with $L_D u = 0$ in $O(x,r)$. Let $y \in O(x, \nicefrac{r}{2})$ and $h \in \R^d$ such that $y + h \in O(x, \nicefrac{r}{2})$. Pick $\varepsilon > 0$ such that $B(y, \varepsilon) \cup B(y + h, \varepsilon) \sub O(x, \nicefrac{r}{2})$ and define $\tau_h f \coloneqq f(\cdot + h)$.

		First, we claim that there is some $c > 0$ depending on the constants in $\rG(\mu_0)$, ellipticity, geometry, $\mu$ and $\mu_0$ such that
		\begin{equation} \label{eq: G(mu) to H(mu): Interpolating L1-L2 estimate}
			\| \e^{- t L^*} (\tau_{-h} -1) f \|_{\L^2(O \setminus B(x, \frac{3}{4} r))} \leq c |h|^{\mu} t^{- \frac{d}{4} - \frac{\mu}{2}} \e^{- \frac{r^2}{c t}} \| f \|_{\L^1(B(y, \varepsilon))}
		\end{equation}
		for all $t \leq 1$ and $f \in \L^1(B(y, \varepsilon))$.

		To see the claim, fix $\nu \in (\mu, \mu_0)$. Property $\rG(\mu_0)$ implies the bound \eqref{eq: D(mu) to G(mu): Reduction L2-Cmu}, so that
		\begin{equation*}
			\| (\tau_h -1)\e^{- t L} f \|_{\L^{\infty}(B(y, \varepsilon))} \les |h|^{\nu} t^{- \frac{d}{4} - \frac{\nu}{2}} \| f \|_2 \qquad (t \leq 1, f\in \L^2(O)).
		\end{equation*}
		By duality, we get
		\begin{equation} \label{eq: G(mu) to H(mu): Dual estimate, L2-C(mu)}
			\| \e^{- t L^*}(\tau_{-h} -1) f \|_2 \les |h|^{\nu} t^{- \frac{d}{4} - \frac{\nu}{2}} \| f \|_{\L^1(B(y, \varepsilon))} \qquad (t \leq 1, f\in \L^1(B(y, \varepsilon))).
		\end{equation}
		Let $\theta \coloneqq \nicefrac{\mu}{\nu}$. We use \eqref{eq: G(mu) to H(mu): Dual estimate, L2-C(mu)} to estimate
		\begin{align*}
			&\| \e^{-t L^*} (\tau_{-h} -1) f \|_{\L^2(O \setminus B(x, \frac{3}{4} r))}
			\\ \leq& \; \| \e^{-t L^*} (\tau_{-h} -1) f \|_2^{\theta} \cdot \| \e^{-t L^*} (\tau_{-h} -1) f \|_{\L^2(O \setminus B(x, \frac{3}{4} r))}^{1 - \theta}
			\\ \les& \; |h|^{\mu} t^{- \frac{\theta d}{4} - \frac{\mu}{2}} \| f \|_{\L^1(B(y, \varepsilon))}^{\theta} \cdot \| \e^{-t L^*} (\tau_{-h} -1) f \|_{\L^2(O \setminus B(x, \frac{3}{4} r))}^{1 - \theta}.
		\end{align*}
		Eventually, we apply the dual estimate of \eqref{eq: D(mu) to G(mu): Reduction L2-Linfty_ODE} with $p=2$, $E = O(x, \nicefrac{r}{2})$ and $F = O \setminus B(x, \nicefrac{3r}{4})$, that is
		\begin{equation*}
			\| \e^{-t L^*} (\tau_{-h} -1) f \|_{\L^2(O \setminus B(x, \frac{3}{4} r))} \les t^{-\frac{d}{4}} \e^{- c \frac{r^2}{t}} \| f \|_{\L^1(B(y, \varepsilon))}.
		\end{equation*}
		The previous two estimates together yield our claim \eqref{eq: G(mu) to H(mu): Interpolating L1-L2 estimate}.

		With the claim at hand, we prove $\H(\mu)$. By duality, it suffices to show for all $\varphi \in \smooth[B(y,\varepsilon)]$ normalized to $\| \varphi\|_1 = 1$ that
		\begin{equation*}
			|( (\tau_h - 1) u \, | \, \varphi)_2| = |(u \, | \, (\tau_{-h} -1) \varphi )_2| \les r^{- \frac{d}{2} - \mu} \| u \|_{\L^2(O(x, r))} |h|^{\mu},
		\end{equation*}
		with an implicit constant not depending on $\varepsilon$.

		We normalize $\| u \|_{\L^2(O(x,r))} =1$, abbreviate $\varphi_h \coloneqq (\tau_{-h} -1) \varphi$ and pick $\chi \in \smooth[\R^d]$ with $\1_{B(x, \nicefrac{7r}{8})} \leq \chi \leq \1_{B(x, \nicefrac{8 r}{9})}$ and $\| \nabla \chi \|_{\infty} \les r^{-1}$. The fundamental theorem of calculus and \eqref{eq: G(mu) to H(mu): Dual estimate, L2-C(mu)} deliver
		\begin{align*}
			|(\tau_h -1) u \, | \, \varphi )_2| &= \bigg|(u \chi \, | \, \e^{-r^2 L^*} \varphi_h)_2 +  \int_0^{r^2} ( u \chi \, | \, L^* \e^{-t L^*} \varphi_h)_2 \, \d t \bigg|
			\\&\leq \| \e^{- r^2 L^*} \varphi_h \|_2 +  \int_0^{r^2} \left| ( u \chi \, | \, L^* \e^{-t L^*} \varphi_h)_2 \right| \, \d t.
			\\&\les r^{- \frac{d}{2} - \mu} |h|^{\mu} +  \int_0^{r^2} \left|( u \chi \, | \, L^* \e^{-t L^*} \varphi_h)_2 \right| \, \d t
			\\&= r^{- \frac{d}{2} - \mu} |h|^{\mu} +  \int_0^{r^2} \left|( A \nabla (u \chi ) \, | \, \nabla \e^{-t L^*} \varphi_h)_2 \right| \, \d t
			\\&\eqqcolon r^{- \frac{d}{2} - \mu} |h|^{\mu} + \mathrm{(I)}.
		\end{align*}
		\textbf{Estimate for (I).} Thanks to Lemma~\ref{D(mu), G(mu) and H(mu) for MBC: Lemma: Sobolev extension by 0} (ii), we know that $\chi \e^{-t L^*} \varphi_h$ serves as a test function for the equation $L_D u = 0$ in $O(x,r)$. Hence, we get
		\begin{align*}
			(A \nabla (u \chi) \, | \, \nabla \e^{-t L^*} \varphi_{h})_2 &= ( u A \nabla \chi \, | \, \nabla \e^{- t L^*} \varphi_{h})_2 + (A \nabla u \, | \, \chi \nabla \e^{- t L^*} \varphi_{h})_2
			\\&= (u A \nabla \chi \, | \, \nabla \e^{- t L^*} \varphi_{h})_2 - ( A \nabla u \, | \, (\e^{-t L^*} \varphi_{h}) \nabla \chi)_2.
		\end{align*}
		Thus, using the properties of $\chi$ and H\"older's inequality, we obtain
		\begin{align*}
			\mathrm{(I)} &\leq \int_0^{r^2} \left| ( u A \nabla \chi \, | \, \nabla \e^{- t L^*} \varphi_{h})_2 \right| \, \d t + \int_0^{r^2} \left| ( A \nabla u \, | \, (\e^{-t L^*} \varphi_{h}) \nabla \chi)_2 \right|  \, \d t
			\\&\les r^{-1} \int_0^{r^2} \| \1_{B(x, \frac{7}{8} r)^c} \sqrt{t} \nabla \e^{-t L^*} \varphi_h \|_2 \, \frac{\d t}{\sqrt{t}}
			\\& \quad + r^{-1} \| \nabla u \|_{\L^2(O(x, \frac{8}{9} r))}  \int_0^{r^2} \| \1_{B(x, \frac{7}{8} r)^c} \e^{-t L^*} \varphi_h \|_2 \, \d t
			\\&\eqqcolon \mathrm{(II)} + \mathrm{(III)}.
		\end{align*}
		\textbf{Estimate for (II).} We split
		\begin{align*}
			\| \1_{B(x, \frac{7}{8} r)^c} \sqrt{t} \nabla \e^{-t L^*} \varphi_h \|_2 &\leq \| \1_{B(x, \frac{7}{8} r)^c} \sqrt{t} \nabla \e^{-\frac{t}{2} L^*} \1_{B(x, \frac{3}{4} r)} \e^{- \frac{t}{2} L^*} \varphi_h \|_2
			\\& \quad + \| \sqrt{t} \nabla \e^{- \frac{t}{2} L^*} \1_{B(x, \frac{3}{4} r)^c} \e^{- \frac{t}{2} L^*} \varphi_h \|_2.
		\end{align*}
		To bound the first summand, we use Proposition~\ref{Operator theoretic setting and relevant function spaces: Proposition: L2 ODE} joint with \eqref{eq: G(mu) to H(mu): Dual estimate, L2-C(mu)} to get
		\begin{equation*}
			\| \1_{B(x, \frac{7}{8} r)^c} \sqrt{t} \nabla \e^{-\frac{t}{2} L^*} \1_{B(x, \frac{3}{4} r)} \e^{- \frac{t}{2} L^*} \varphi_h \|_2 \les \e^{-c \frac{r^2}{64 t}} |h|^{\mu} t^{- \frac{d}{4} - \frac{\mu}{2}}.
		\end{equation*}
		As for the second summand, we use \eqref{eq: G(mu) to H(mu): Interpolating L1-L2 estimate} instead of \eqref{eq: G(mu) to H(mu): Dual estimate, L2-C(mu)} to obtain the same upper bound. Hence, the substitution $s = \nicefrac{r^2}{t}$ eventually reveals
		\begin{align*}
			\mathrm{(II)} \les r^{-1} |h|^{\mu} \int_0^{r^2} \e^{-c \frac{r^2}{64 t}} t^{- \frac{d}{4} - \frac{\mu}{2} - \frac{1}{2}} \, \d t \simeq r^{- \frac{d}{2} - \mu} |h|^{\mu}.
		\end{align*}
		\textbf{Estimate for (III).} Using Caccioppoli's inequality, \eqref{eq: G(mu) to H(mu): Interpolating L1-L2 estimate} and again the substitution $s = \nicefrac{r^2}{t}$, we have
		\begin{equation*}
			\mathrm{(III)} \les r^{-2} |h|^{\mu} \int_0^{r^2} \e^{-c \frac{r^2}{t}} t^{- \frac{d}{4} - \frac{\mu}{2}} \, \d t \simeq r^{- \frac{d}{2} - \mu} |h|^{\mu},
		\end{equation*}
		which completes the proof.
	\end{proof}


	\section{From $\H(\mu)$ to $\rD(\mu)$} \label{Section: (3) to (1)}

	In this section we close the circle of implications by proving:

	\begin{theorem} \label{(3) to (1): Theorem: Main result}
		Let~\hyperref[The geometric setup: Definition (Fat)]{$\Fat$} and~\hyperref[The geometric setup: Definition (LU)]{$\LU$} be satisfied. If $L$ has property $\H(\mu)$, then $L$ has property $\rD(\mu)$.
	\end{theorem}

	\begin{proof}
		It suffices to control the growth of the Dirichlet integral when $0 < 4 c_1 r \leq R \leq r_0$, $x \in \overline{O}$ and $u \in \H_D^1(O)$ with $L_Du =0$ in $O(x, R)$, see Remark~\ref{(D(mu), G(mu) and H(mu) for MBC: Remark: Scaling D(mu)}. We distinguish between two cases.

		\textbf{(1) $\boldmath{\nicefrac{R}{c_1} < \d_D(x)}$.}  Let $\varphi \in \smoothD[\R^d]$ with $\varphi = 1$ on $\overline{B(x, \nicefrac{R}{c_1})}$. Then$$v \coloneqq \varphi (u - (u)_{O(x, \nicefrac{R}{c_1})}) \in \H_D^1(O) \quad \& \quad L_D v =0 \; \text{in} \; O(x, \nicefrac{R}{c_1}).$$Thanks to property $\H(\mu)$, $v$ has a continuous representative in $O(x, \nicefrac{R}{2 c_1}) \supseteq O(x, 2r)$. We start with the Caccioppoli inequality
		\begin{align*}
			\int_{O(x,r)} |\nabla u|^2
			&= \int_{O(x,r)} |\nabla v|^2
			\\&\les r^{-2} \int_{O(x, 2r)} |v - v(x)|^2,
			\intertext{where now we can bring $\H(\mu)$ into play and then apply~\hyperref[The geometric setup: Proposition: (P)]{$\AssP$} in order to get}
			&\les r^{-2} R^{-d - 2 \mu} \left( \int_{O(x, 2r)} |x-y|^{2 \mu} \, \d y \right) \int_{O(x, \nicefrac{R}{c_1})} |v|^2
			\\&= r^{-2} R^{-d - 2 \mu} \left( \int_{O(x, 2r)} |x-y|^{2 \mu} \, \d y \right) \int_{O(x, \nicefrac{R}{c_1})} |u - (u)_{O(x, \nicefrac{R}{c_1})}|^2
			\\&\les  \left( \frac{r}{R} \right)^{d- 2 + 2 \mu} \int_{O(x, R)} |\nabla u|^2.
		\end{align*}
		\textbf{(2) $\boldmath{\d_D(x) \leq \nicefrac{R}{c_1}}$.} We consider two subcases.

		\textbf{(2.1) $\boldsymbol{\d_D(x) \leq 2 r}$.} We pick $x_D \in D$ with $\d_D(x) = |x-x_D|$. The most important observation is that (the continuous representative of) $u$ vanishes in $x_D$ by Lemma~\ref{(D(mu), G(mu) and H(mu) for MBC: Lemma: Pointwise real. of bdr. conditions}. The Caccioppoli inequality yields
		\begin{align*}
			\int_{O(x,r)} |\nabla u|^2
			&\les r^{-2} \int_{O(x, 2r)} |u|^2.
			\intertext{By property $\H(\mu)$ on $O(x, \nicefrac{R}{2 c_1}) \supseteq O(x,2r)$ and~\hyperref[The geometric setup: Proposition: (P)]{$\AssP$}, we obtain in the usual manner that}
			&\les r^{-2} \int_{O(x, 2r)} |u(y) - u(x)|^2 \, \d y  + r^{d-2} |u(x) - u(x_D)|^2
			\\&\les R^{-d - 2\mu} (r^{d-2 + 2 \mu} + r^{d-2} \d_D(x)^{2 \mu}) \int_{O(x, \nicefrac{R}{c_1})} |u|^2
			\\&\les \left( \frac{r}{R} \right)^{d-2 + 2 \mu} \int_{O(x, R)} |\nabla u|^2.
		\end{align*}
		\textbf{(2.2) $\boldsymbol{2 r < \d_D(x)}$.} In this case we can replace $u$ by $u - u(x)$ when we apply the Caccioppoli inequality in case (2.1). Then $x_D$ is not needed and we conclude by the same chain of estimates.
	\end{proof}

	\section{Property $\H(\mu)$ for divergence form operators with real coefficients}  \label{Section: Real-valued A}

	Our goal in this section is to show Theorem~\ref{Theorem: Main result 2}. As mentioned in the introduction, we can go one step further and relax~\hyperref[The geometric setup: Definition (Fat)]{$\Fat$} and~\hyperref[The geometric setup: Definition (LU)]{$\LU$} to the following axiomatic framework, where $p \in (1,2)$.

	\begin{enumerate}
		\item[$\AssE$] \textbf{Embedding property for $\boldsymbol{\H_D^1(O)}$.} If $d\geq3$, assume there is $c_E > 0$ such that
		\begin{equation}
			\| u \|_{2^*} \leq c_E \| u \|_{1,2} \qquad (u \in \H_D^1(O)).
		\end{equation}
		If $d=2$, assume there are $q \in (2, \infty)$ and $c_E >0$ such that
		\begin{equation}
			\| u \|_q \leq c_E \| u \|_{1,2}^{1 - \frac{2}{q}} \| u \|_2^{\frac{2}{q}} \qquad (u \in \H_D^1(O)).
		\end{equation}
		\item[$\AssPp$] \textbf{Weak $\boldsymbol{p}$-Poincar\'{e} inequality.} There are $c_0, r_0 > 0, c_1 \geq 1$ with
		\begin{equation}
			\| u - \1_{[\d_D(x) > r ]} \cdot (u)_{O(x,r)} \|_{\L^p(O(x,r))} \leq c_0 r \| \nabla u \|_{\L^p(O(x, c_1 r))}
		\end{equation}
		for all $u \in \W_D^{1,p}(O)$, each $x \in \overline{O}$ and all $r \in (0, r_0]$.
	\end{enumerate}

	The implicit constants in~\hyperref[Section: Real-valued A]{$\AssPp$} might be different from the ones chosen in and after Proposition~\ref{The geometric setup: Proposition: (P)} but they serve the exact same purpose. It is only that in this section we postulate~\hyperref[Section: Real-valued A]{$\AssPp$} instead of deriving it from geometric assumptions.

	Let us explain why these properties follow from our concrete geometric assumptions in the previous sections.

	\begin{lemma} \label{Real valued A: Lemma: Fat und LU imply E and Pp}
		Assumptions~\hyperref[The geometric setup: Definition (Fat)]{$\Fat$} and~\hyperref[The geometric setup: Definition (LU)]{$\LU$} imply~\hyperref[Section: Real-valued A]{$\AssE$} and~\hyperref[Section: Real-valued A]{$\AssPp$} for some $p \in (1, 2)$. Moreover,~\hyperref[Section: Real-valued A]{$\AssE$} holds true in the pure Dirichlet case on any open set $O$.
	\end{lemma}

	\begin{proof}
		For $d \geq 3$,~\hyperref[Section: Real-valued A]{$\AssE$} follows from~\hyperref[The geometric setup: Theorem: Extension operator]{$\AssExt$} and the embedding $\H^1(\R^d) \sub \L^{2^*}(\R^d)$. Similarly, for $d = 2$, we can use all $q \in (2, \infty)$ since we have the interpolation inequality
		\begin{equation*}
			\| u \|_{\L^q(\R^d)} \les \| \nabla u \|_{\L^2(\R^d)}^{1 - \frac{2}{q}} \| u \|_{\L^2(\R^d)}^{\frac{2}{q}} \qquad (u \in \H^1(\R^d)),
		\end{equation*}
		see \cite[Lec.~2, Thm.]{Nirenberg_Sobolev_Interpolation}.

		To show~\hyperref[Section: Real-valued A]{$\AssPp$}, we borrow a deep result from capacity theory:
         \begin{quote}
			Let $d \geq 2$ and $C$ be closed. If $C$ is locally $2$-fat, then $C$ is locally $p$-fat for some $p \in (1,2)$.
		\end{quote}
        This is called `self-improvement of $p$-fatness', a phenomenon that is attributed to Lewis~\cite[Thm.~1]{Lewis_P-fat_Open_ended} but for a slightly different version of capacities. With our definition, the proof can be found in \cite[Thm.~8.2]{Mikkonen}. At least when $d \geq 3$, one can rely directly on Lewis' result because the two versions of local $2$-fatness coincide. For convenience, we include this argument in Appendix~\ref{Section: Appendix}.

		We apply this result to the auxiliary set $D \cup (O^c \setminus N_{\delta}^{\Sigma})$ in Lemma~\ref{The geometric setup: Lemma: (Fat) rephrased} to see that~\hyperref[The geometric setup: Definition (Fat)]{$\Fat$} self-improves to~\hyperref[The geometric setup: Definition (Fat)]{$\Fatp$} for some $p \in (1,2)$. Hence,~\hyperref[Section: Real-valued A]{$\AssPp$} follows from Proposition~\ref{The geometric setup: Proposition: (P)}.
	\end{proof}

	To prove Theorem~\ref{Theorem: Main result 2}, we show in a first step that any $L$-harmonic function $u$ is locally bounded. For this we adapt the proof of \cite[Prop.~8.15]{Giaquinta_Martinazzi} to derive a decay condition on the super-level sets of $u$. Here, we only need the embedding~\hyperref[Section: Real-valued A]{$\AssE$}.
	\newline
	In a second step, we follow a classical approach \cite{DeGio57, Giaquinta_Martinazzi} and use the local boundedness of $u$ to obtain estimates for its oscillation, which eventually results in local H\"older continuity. At this point,~\hyperref[Section: Real-valued A]{$\AssPp$} is crucial, as it allows us to get a \textit{quantitative} decay in the super-level sets of $u$. This uses ideas from non-linear methods \cite{DiBe89, DiBenedetto}.

	\subsection{Local boundedness of functions in the De Giorgi class.}  \label{Subsection: Local boundedness, real A}

	The central point of this subsection is that the local boundedness of a function is a consequence of lying in a function class rather than solving an equation. However, the latter is needed to obtain uniform control on the implicit constants, which is essential for Theorem~\ref{Theorem: Main result 2}.

	\begin{definition} \label{Real valued A: Definition: De Giorgi class}
		Let $x \in \overline{O}$ and $R > 0$. We define $\mathrm{DG}_{D, x, R}(O)$ as the set of all $u \in \H_D^1(O; \R)$ for which there exists a constant $c >0$ such that for all $r \in (0,R)$ and $k \in [0, \infty)$ it holds
		\begin{equation} \label{eq: De Giorgi class estimate}
			\int_{O(x,r)} |\nabla (u \mp k)^{\pm}|^2 \leq \frac{c}{(R-r)^2} \int_{O(x,R)} |(u \mp k)^{\pm}|^2.
		\end{equation}
		In addition, if $R < \d_D(x)$, then we require both estimates for all $k \in \R$.
	\end{definition}

	\begin{remark}  \label{Real valued A: Remark: De Giorgi class}
		Notice that $u \in \mathrm{DG}_{D, x, R}(O)$ if and only if $-u \in \mathrm{DG}_{D, x, R}(O)$ due to the identities $(-u-k)^+ = (u + k)^-$ and $(-u + k)^- = (u - k)^+$.
	\end{remark}

	Next, we state the main result of this section. We define
	\begin{align}  \label{eq: Real-valued A: Definition of theta}
		\theta \coloneqq \frac{1}{2} + \sqrt{\frac{1}{4} + \frac{2}{\delta}} > 1, \quad \text{where} \quad \delta \coloneqq \begin{cases}
			\frac{2 q}{q -2} \quad &(d =2), \\ d  &(d \geq 3).
		\end{cases}
	\end{align}

	\begin{theorem} \label{Real valued A: Theorem: Local boundedness}
		Assume~\hyperref[Section: Real-valued A]{$\AssE$}. Let $x \in \overline{O}$, $R_0 > 0$ and $u \in \mathrm{DG}_{D, x, R_0}(O)$. There is some $c > 0$ such that we have for all $k \geq 0$ and $R \in (0, R_0]$ the estimate
		\begin{equation*}
			\underset{O(x, \frac{R}{2})}{\esssup} \; u^{+} \leq k + c \left( R^{-d} \int_{O(x,R)} |(u-k)^{+}|^2 \right)^{\frac{1}{2}} (R^{-d} |\{ u > k \}(x,R)|)^{\frac{\theta -1}{2}}.
		\end{equation*}
		In addition, if $R < \d_D(x)$, then we can allow for all $k \in \R$ in the estimate. Furthermore, if $A$ is real-valued and $u \in \H_D^1(O; \R)$ with $L_Du =0$ in $O(x,R)$, then $c > 0$ depends only on $d$, ellipticity, $R_0$ and~\hyperref[Section: Real-valued A]{$\AssE$}.
	\end{theorem}

	Before we come to the proof of Theorem~\ref{Real valued A: Theorem: Local boundedness}, let us show that $\mathrm{DG}_{D, x, R}(O)$ is the natural energy class associated to the equation $L_D u = 0$ in $O(x, R)$.

	\begin{lemma}  \label{Real valued A: Lemma: Caccioppoli}
		Let $A$ be real-valued, $x \in \overline{O}$, $R > 0$ and $u \in \H_D^1(O; \R)$ such that $L_Du =0$ in $O(x,R)$. Then $ u \in \mathrm{DG}_{D, x, R}(O)$ with an implicit constant depending only on ellipticity.
	\end{lemma}

	\begin{proof}
		Owing to Remark~\ref{Real valued A: Remark: De Giorgi class} we only have to show the estimate for $(u-k)^+$. First, assume that $\d_D(x) \leq R$ and $k \geq 0$.

		Let $r \in (0, R)$ and $\varphi \in \smooth[B(x, R)]$ be $[0,1]$-valued with $\varphi =1$ on $B(x,r)$ and $\| \nabla \varphi \|_{\L^{\infty}(\R^d)} \leq \nicefrac{2}{(R -r)}$.

		Put $w\coloneqq (u-k)^+ \varphi$ and $v\coloneqq (u-k)^+ \varphi^2$. They are contained in $\H^1_D(O)$ and test functions for the equation for $u$, see Lemma~\ref{Real valued A: Lemma: Auxiliary results} and Lemma~\ref{D(mu), G(mu) and H(mu) for MBC: Lemma: Sobolev extension by 0} (ii). By the product rule and ellipticity, we have
		\begin{align}
			\frac{\lambda}{2} \int_O |\nabla (u-k)^+|^2 \varphi^2
			&\leq \lambda \int_O |\nabla w|^2 + |(u-k)^+ \nabla \varphi|^2 \\
			&\leq \int_O A \nabla w \cdot \nabla w + \lambda |(u-k)^+ \nabla \varphi|^2 \\
			&= \int_O A \nabla u \cdot \varphi \nabla w + A (u-k)^+ \nabla \varphi \cdot \nabla w + \lambda |(u-k)^+ \nabla \varphi|^2 \\
			&= \int_O - A \nabla u \cdot  w \nabla \varphi  + A (u-k)^+ \nabla \varphi \cdot \nabla w + \lambda |(u-k)^+ \nabla \varphi|^2,
			\intertext{where we have used the equation for $u$ with test function $v$. By definition of $w$ we can continue by}
			&\leq \int_O - A \varphi \nabla (u-k)^+ \cdot  (u-k)^+  \nabla \varphi  + A (u-k)^+ \nabla \varphi \cdot \varphi \nabla (u-k)^+ \\
			&\quad + A (u-k)^+ \nabla \varphi \cdot (u-k)^+ \nabla \varphi + \lambda |(u-k)^+ \nabla \varphi|^2.
		\end{align}
		At this point, we can use the boundedness of $A$ and Young's inequality to absorb all terms with $\varphi \nabla (u-k)^+$ on the right. We are left with
		\begin{equation}
			\int_O |\nabla (u-k)^+|^2 \varphi^2  \leq c \int_O |(u-k)^+ \nabla \varphi|^2,
		\end{equation}
		where $c$ depends on ellipticity, and \eqref{eq: De Giorgi class estimate} follows by the choice of $\varphi$.

		Finally, if $R < \d_D(x)$, then $v,w \in \H_D^1(O)$ for all $k \in \R$ and the same argument applies.
	\end{proof}

	\begin{proof}[\rm\bf{Proof of Theorem~\ref{Real valued A: Theorem: Local boundedness}.}] We begin with the case $\d_D(x) \leq R$.

		Fix $r \in (0, R)$ and $k \geq 0$. Let $\eta \in \smooth[B(x, \nicefrac{(r+ R)}{2})]$ be $[0,1]$-valued with $\eta =1$ on $B(x,r)$ and $\| \nabla \eta \|_{\L^{\infty}(\R^d)} \leq \nicefrac{4}{(R -r)}$. Then $\eta (u-k)^+ \in \H_D^1(O)$ by Lemmas~\ref{Real valued A: Lemma: Auxiliary results} and~\ref{D(mu), G(mu) and H(mu) for MBC: Lemma: Sobolev extension by 0}~(ii). We deduce from \eqref{eq: De Giorgi class estimate} that
		\begin{align}
			\| \nabla (\eta (u-k)^+) \|_{\L^2(O(x, \frac{r+R}{2}))} &\leq \| (u-k)^+ \nabla \eta \|_{\L^2(O(x,\frac{r+R}{2}))} + \| \eta \nabla (u-k)^+ \|_{\L^2(O(x,\frac{r+R}{2}))} \notag
			\\& \les \frac{1}{R -r} \| (u-k)^+ \|_{\L^2(O(x, R))}.  \label{eq: Local boundedness, 1}
		\end{align}
		By Lemma~\ref{Real valued A: Lemma: Caccioppoli} the implicit constant depends only on ellipticity in the case that $L_Du =0$ in $O(x,R)$. We abbreviate
		\begin{equation*}
			A_{k,R} \coloneqq \{ u > k \}(x,R).
		\end{equation*}
		First, let $d \geq 3$. Using~\hyperref[Section: Real-valued A]{$\AssE$} and \eqref{eq: Local boundedness, 1}, we get
		\begin{align}
			\label{eq: Real valued A: Local boundedness, embedding step}
			\begin{split}
				\| u-k \|_{\L^{2^*}(A_{k,r})} &\leq \| \eta (u-k)^+ \|_{2^*}
				\\&\les \| \nabla (\eta (u-k)^+) \|_{\L^2(O(x, \frac{r+R}{2}))} + \| \eta (u-k)^+ \|_{\L^2(O(x, \frac{r+R}{2}))}
				\\&\les \frac{1+R_0}{R-r} \| u-k \|_{\L^2(A_{k,R})}.
			\end{split}
		\end{align}
		H\"older's inequality yields
		\begin{equation}
			\| u - k \|_{\L^2(A_{k,r})} \leq  |A_{k,r}|^{\frac{1}{d}} \| u - k \|_{\L^{2^*}(A_{k,r})} \les \frac{|A_{k,R}|^{\frac{1}{d}}}{R -r } \| u-k \|_{\L^2(A_{k,R})}.  \label{eq: Local boundedness (1)}
		\end{equation}
		Now, let $d =2$. Then~\hyperref[Section: Real-valued A]{$\AssE$} postulates a Gagliardo--Nirenberg-type inequality.  When we apply $\AssE $ in \eqref{eq: Real valued A: Local boundedness, embedding step}, we obtain instead of \eqref{eq: Local boundedness (1)} the estimate
		\begin{equation}  \label{eq: Local boundedness (1), d=2}
			\| u - k \|_{\L^2(A_{k,r})} \leq |A_{k,r}|^{\frac{1}{\delta}} \| u - k \|_{\L^q(A_{k,r})} \les \frac{|A_{k,R}|^{\frac{1}{\delta}}}{(R -r)^{\frac{2}{\delta}}} \| u-k \|_{\L^2(A_{k,R})}.
		\end{equation}
		Recall the definition of $\theta$ from \eqref{eq: Real-valued A: Definition of theta} and define
		\begin{equation*}
			\Phi(k, r) \coloneqq \| u -k \|_{\L^2(A_{k,r})}^{\delta \theta} |A_{k,r}|.
		\end{equation*}
		We raise \eqref{eq: Local boundedness (1)} and \eqref{eq: Local boundedness (1), d=2} to the $\delta \theta$-th power and multiply by $|A_{k,r}|$ to get for all $h < k$ the estimate
		\begin{align*}
			\Phi(k,r) &\les \frac{1}{(R -r)^{d \theta}} |A_{k,R}|^{\theta} |A_{k, r}| \| u - k \|_{\L^2(A_{k,R})}^{\delta \theta}
			\\&\les \frac{1}{(R -r)^{d \theta}} |A_{h,R}|^{\theta} |A_{k, r}| \| u - h \|_{\L^2(A_{h,R})}^{\delta \theta}
			\\&\leq \frac{1}{(R -r)^{d \theta}} \frac{1}{(k-h)^2} |A_{h,R}|^{\theta} \| u - h \|_{\L^2(A_{h,R})}^{2 + \delta \theta}.
		\end{align*}
		Using that $\theta$ is the positive solution to $\theta^2 - \theta - \nicefrac{2}{\delta} =0$, we have proven so far that there is some $c > 0$ such that we have for all $0 < r < R$, $k \geq 0$ and $h < k$ the estimate
		\begin{equation}
			\Phi(k,r) \leq \frac{c}{(R-r)^{d \theta}} \frac{1}{(k-h)^2} \Phi(h,R)^{\theta}. \label{eq: Local boundedness (2)}
		\end{equation}
		At this point, we set up an iteration scheme to conclude. Let $\zeta > 0$, set $k_n \coloneqq k + \zeta - \nicefrac{\zeta}{2^n}$ and $r_n \coloneqq \nicefrac{R}{2} + \nicefrac{R}{2^{n+1}}$, so that
		\begin{equation*}
			\Phi(k_{n+1}, r_{n+1}) \leq c 2^{d \theta} \frac{2^{(n+1)(d \theta +2)}}{R^{d \theta} \zeta^2} \Phi(k_n, r_n)^{\theta}.
		\end{equation*}
		Let $\mu \coloneqq \nicefrac{(d \theta +2)}{(\theta -1)} > 0$ and put $\psi_n \coloneqq 2^{\mu n} \Phi(k_n, r_n)$. Then
		\begin{equation*}
			\psi_{n+1} \leq  c \frac{2^{\theta (\mu + d)}}{R^{d \theta} \zeta^2} \psi_n^{\theta}.
		\end{equation*}
		We choose $\zeta > 0$ such that
		\begin{equation*}
			\psi_0^{1- \theta} = c \frac{2^{\theta (\mu + d)}}{R^{d \theta} \zeta^2},
		\end{equation*}
		that is,
		\begin{equation*}
			\zeta = c^{\frac{1}{2}} 2^{\frac{\theta (\mu + d)}{2}} R^{- \frac{d \theta}{2}} \| u - k \|_{\L^2(A_{k, R})} |A_{k, R}|^{\frac{\theta -1}{2}} = c' R^{- \frac{d}{2}} \| u - k \|_{\L^2(A_{k, R})} ( R^{-d} |A_{k, R}|)^{\frac{\theta -1}{2}}.
		\end{equation*}
		Our choice of $\zeta$ and induction yields $\psi_n \leq \psi_0$ for all $n \in \N_0$. This eventually implies
		\begin{equation*}
			\Phi (k + \zeta, \nicefrac{R}{2}) \leq \limsup_{n \to \infty} \Phi(k_n, r_n) = \limsup_{n \to \infty} 2^{- \mu n} \psi_n = 0.
		\end{equation*}
		Thus, $\Phi(k+ \zeta, \nicefrac{R}{2}) =0$ and hence $|A_{k + \zeta,\nicefrac{R}{2}}| =0$ or $u= k + \zeta$ on $A_{k + \zeta,\nicefrac{R}{2}}$. In both cases we conclude that
		\begin{equation*}
			\underset{O(x, \frac{R}{2})}{\esssup} \; u^+ \leq k + \zeta,
		\end{equation*}
		as claimed.

		The restriction $k \geq 0$ was only used in order to apply \eqref{eq: De Giorgi class estimate} and to guarantee that $\eta (u - k)^+ \in \H_D^1(O)$. But when $R < \d_D(x)$, this is true for all $k \in \R$ simply by the support of $\eta$ and the same argument applies.
	\end{proof}


	\subsection{Property \boldmath$\H(\mu)$ for $L$}  \label{Subsection: Property H(mu), real A}

	So far, we have worked under assumption (E) alone. Now, we will add~\hyperref[Section: Real-valued A]{$\AssPp$} in order to upgrade local boundedness to local H\"older continuity.

	\begin{theorem}   \label{Real valued A: Theorem: Property H(mu)}
		Assume~\hyperref[Section: Real-valued A]{$\AssE$} and~\hyperref[Section: Real-valued A]{$\AssPp$}. Let $x \in \overline{O}$, $R \leq r_0$ and $u \in \mathrm{DG}_{D, x, R}(O)$. Then $u$ is locally H\"older continuous in $O(x, \nicefrac{R}{4})$ with
		\begin{equation*}
			R^{\mu} [u]^{(\mu)}_{O(x, \frac{R}{4})} \les R^{-\frac{d}{2}} \| u \|_{\L^2(O(x,R))}.
		\end{equation*}
		The implicit constant and $\mu$ depend only on $d$,~\hyperref[Section: Real-valued A]{$\AssPp$},~\hyperref[Section: Real-valued A]{$\AssE$} and the constant in \eqref{eq: De Giorgi class estimate}.
	\end{theorem}

	This result implies Theorem~\ref{Theorem: Main result 2}.

	\begin{proof}[\rm\bf{Proof of Theorem~\ref{Theorem: Main result 2} from Theorem~\ref{Real valued A: Theorem: Property H(mu)}}] Let $u \in \H_D^1(O)$ with $L_D u = 0$ in $O(x, r)$ for some $x \in \overline{O}$. By Remark~\ref{(D(mu), G(mu) and H(mu) for MBC: Remark: H(mu) Scaling} we can assume $r \leq r_0$ and it suffices to prove \eqref{eq: Definition: Property H(mu)} with $O(x, \nicefrac{r}{4})$ on the left-hand side.

		Since $A$ is real-valued, $\re(u), \im(u) \in \H_D^1(O; \R)$ solve the same equation as $u$. According to Lemma~\ref{Real valued A: Lemma: Caccioppoli} they belong to $\mathrm{DG}_{D, x, r}(O)$. As the geometric assumptions~\hyperref[The geometric setup: Definition (Fat)]{$\Fat$} and~\hyperref[The geometric setup: Definition (LU)]{$\LU$} imply~\hyperref[Section: Real-valued A]{$\AssE$} and~\hyperref[Section: Real-valued A]{$\AssPp$} by Lemma~\ref{Real valued A: Lemma: Fat und LU imply E and Pp}, we can apply Theorem~\ref{Real valued A: Theorem: Property H(mu)} to $\re(u)$ and $\im(u)$ to complete the proof.
	\end{proof}

	In order to prove Theorem~\ref{Real valued A: Theorem: Property H(mu)}, we still need two short lemmas. The slight asymmetry in the scaling of the radius between (i) and (ii) is unavoidable and the reader might want to think of $c_1=1$ on a first reading.

	\begin{lemma}  \label{Real valued A: Lemma: 2-isoperimetric inequality}
		Assume~\hyperref[Section: Real-valued A]{$\AssPp$}, let $u \in \H^1_D(O; \R)$, $r \in (0, r_0]$ and $x \in \overline{O}$.
		\begin{enumerate}
			\item If $r < \d_D(x)$, then it holds for all $h < k$ that
			\begin{equation} \label{eq: 2-isoperimetric near Neumann}
				(k-h)^p |\{ u \geq k \} (x,\nicefrac{r}{c_1})| \leq \frac{c_0^p  r^p|O(x,\nicefrac{r}{c_1})|^p}{|\{ u \leq h \}(x,\nicefrac{r}{c_1})|^p} \int_{\{ h \leq u \leq k\}(x,r)} |\nabla u|^p.
			\end{equation}

			\item If $\d_D(x) \leq r$, then it holds for all $0 \leq h < k$ that
			\begin{equation} \label{eq: 2-isoperimatric near Dirichlet}
				(k-h)^p |\{ u \geq k \} (x,r)| \leq c_0^p r^p \int_{\{ h \leq u \leq k\}(x,c_1 r)} |\nabla u|^p.
			\end{equation}
		\end{enumerate}
	\end{lemma}

	\begin{proof}[\rm\bf{Proof of (i).}] Let $\varphi \in \smoothD[\R^d]$ with $\varphi =1$ on $B(x, r)$. Set $w \coloneqq \varphi ((u-h)^+-(u-k)^+)$ and estimate
		\begin{equation*}
			(w)_{O(x,\nicefrac{r}{c_1})} \leq \frac{(k-h) |\{ u > h \}(x,\nicefrac{r}{c_1})|}{|O(x,\nicefrac{r}{c_1})|} \leq k -h.
		\end{equation*}
		As $w = k - h$ on $\{ u \geq k \}(x,\nicefrac{r}{c_1})$, we can use this bound and~\hyperref[Section: Real-valued A]{$\AssPp$} to get
		\begin{align*}
			&(k-h)^p \left( 1-\frac{|\{ u > h \}(x,\nicefrac{r}{c_1})|}{|O(x,\nicefrac{r}{c_1})|} \right)^p |\{ u \geq k \}(x,\nicefrac{r}{c_1})|
			\\=& \int_{\{ u \geq k \}(x,\nicefrac{r}{c_1})} \bigg| k-h-\frac{(k-h)|\{ u > h \}(x,\nicefrac{r}{c_1})|}{|O(x,\nicefrac{r}{c_1})|} \bigg|^p
			\\\leq& \int_{O(x,\nicefrac{r}{c_1})} |w-(w)_{O(x,\nicefrac{r}{c_1})}|^p
			\\\leq& \, c_0^p r^p \int_{O(x, r)} |\nabla w|^p.
		\end{align*}
		Since $\nabla w = \1_{\{ h \leq  u \leq k \}} \nabla u$ in $O(x, r)$, we are done.

		\textbf{Proof of (ii).} We first use Lemma~\ref{Real valued A: Lemma: Auxiliary results} to conclude that $w \coloneqq (u-h)^+ - (u-k)^+ \in \H_D^1(O)$.~\hyperref[Section: Real-valued A]{$\AssPp$} implies
		\begin{align*}
			(k-h)^p |\{ u \geq k \}(x,r)| &= \int_{\{ u \geq k \}(x,r)} |w|^p \leq c_0^p r^p \int_{O(x,c_1 r)} |\nabla w|^p.
		\end{align*}
		This completes the proof.
	\end{proof}

	\begin{definition}
		Let $u \colon O \to \R$ be a measurable function, $x \in \overline{O}$ and $r > 0$. We define \begin{enumerate}
			\item[(i)] $M_{x,u}(r) \coloneqq \underset{O(x,r)}{\esssup} \; u$.

			\item[(ii)] $ m_{x,u}(r) \coloneqq \underset{O(x,r)}{\essinf} \; u$.

			\item[(iii)] $\osc_{x,u}(r) \coloneqq M_{x,u}(r) - m_{x, u}(r)$.
		\end{enumerate}

		Here the third expression is only defined, when at least one of the summands is finite or both are infinite with a different sign.
	\end{definition}

	\begin{lemma}[Shrinking Lemma] \label{Real valued A: Lemma: ShrinkingLemma}
		Assume~\hyperref[Section: Real-valued A]{$\AssE$} and~\hyperref[Section: Real-valued A]{$\AssPp$}. Let $r \in (0, 4 c_1 r_0]$, $x \in \overline{O}$, $u \in \mathrm{DG}_{D, x, r}(O)$ with $\osc_{x, u}(\nicefrac{r}{2}) > 0$ and define $\gamma(p) \coloneqq 1 -\nicefrac{p}{2} \in (0,1)$.
		\begin{enumerate}
			\item If $\nicefrac{r}{4 c_1} < \d_D(x)$ and
			\begin{equation} \label{eq: measure}
				\left|\left\{ u > M_{x,u}(\nicefrac{r}{2}) - 2^{-1} \osc_{x,u}(\nicefrac{r}{2}) \right\} (x,\nicefrac{r}{4 c_1^2}) \right| \leq \frac{1}{2} |O(x,\nicefrac{r}{4 c_1^2})|,
			\end{equation}
			then there is some $c>0$ depending only on~\hyperref[Section: Real-valued A]{$\AssPp$} and the implicit constant in \eqref{eq: De Giorgi class estimate} for $u$ such that for each $n \in \N$ the super-level sets of $u$ shrink by the law
			\begin{equation}\label{eq: shrinkyouhellofalevelset}
				\left| \left\{ u \geq M_{x,u}(\nicefrac{r}{2}) - 2^{-(n+1)} \osc_{x,u}(\nicefrac{r}{2}) \right\} (x,\nicefrac{r}{4 c_1^2}) \right| \leq c |O(x,\nicefrac{r}{4 c_1})| \cdot n^{- \gamma(p)}.
			\end{equation}
			\item If $\d_D(x) \leq \nicefrac{r}{4 c_1}$ and
			\begin{equation*}
				M_{x,u}(\nicefrac{r}{2}) - 2^{-1} \osc_{x,u}(\nicefrac{r}{2}) \geq 0,
			\end{equation*}
			then there is some $c>0$ depending only on~\hyperref[Section: Real-valued A]{$\AssPp$} and the implicit constant in \eqref{eq: De Giorgi class estimate} for $u$ such that for each $n \in \N$ the super-level sets of $u$ shrink by the law
			\begin{equation}\label{eq: shrinkyouhellofalevelset}
				\left| \left\{ u \geq M_{x,u}(\nicefrac{r}{2}) - 2^{-(n+1)} \osc_{x,u}(\nicefrac{r}{2}) \right\} (x,\nicefrac{r}{4 c_1}) \right| \leq c |O(x,\nicefrac{r}{4})| \cdot n^{- \gamma(p)}.
			\end{equation}
		\end{enumerate}
	\end{lemma}

	\begin{proof}
		Theorem~\ref{Real valued A: Theorem: Local boundedness} implies $\osc_{x,u}(\nicefrac{r}{2}) < \infty$. Define for $i = 0, \up, n$ the numbers
		\begin{equation} \label{eq: Real-valued A: Shrinking lemma: Def of k_i}
			k_i \coloneqq k_i(u) \coloneqq M_{x,u}(\nicefrac{r}{2}) - 2^{-(i+1)} \osc_{x,u}(\nicefrac{r}{2}) \quad \& \quad A_{i,r} \coloneqq \{ u \geq k_i \} (x,r).
		\end{equation}
		We begin with case (i). Using that $A_{i,r}$ is decreasing in $i$ joint with the assumption
		\begin{equation*}
			|\{ u \leq k_0 \} (x,\nicefrac{r}{4 c_1^2})| \geq \frac{1}{2} |O(x,\nicefrac{r}{4 c_1^2})|,
		\end{equation*}
		we derive from \eqref{eq: 2-isoperimetric near Neumann} that
		\begin{equation} \label{eq: Shrinking Lemma: Use of 2-iso}
			|k_{i+1}-k_i|^p |A_{i+1,\nicefrac{r}{4 c_1^2}}| \les r^p \int_{\{ k_i \leq u \leq k_{i+1} \}(x,\nicefrac{r}{4 c_1})} |\nabla u|^p.
		\end{equation}
		Now, we use H\"older's inequality and $u \in \mathrm{DG}_{D, x, r}(O)$ to infer
		\begin{align*}
			|k_{i+1}-k_i|^p |A_{i+1,\nicefrac{r}{4 c_1^2}}| &\les r^p \left( \int_{O(x,\nicefrac{r}{4 c_1})} |\nabla (u-k_i)^+|^2 \right)^{\frac{p}{2}} |A_{i,\nicefrac{r}{4 c_1 }} \setminus A_{i+1,\nicefrac{r}{4 c_1}}|^{1 - \frac{p}{2}}
			\\&\les \left( \int_{O(x,\nicefrac{r}{2})} |(u-k_i)^+|^2 \right)^{\frac{p}{2}}  |A_{i,\nicefrac{r}{4 c_1}} \setminus A_{i+1,\nicefrac{r}{4 c_1}}|^{1 - \frac{p}{2}}
		\end{align*}
		Now, we use that $(u-k_{i})^+ \leq 2^{-(i+1)} \osc_{x, u}(\nicefrac{r}{2}) =2( k_{i+1} -k_{i})$ on $O(x, \nicefrac{r}{2})$ and that $|O(x,r)| \lesssim r_0^d$ to conclude the bound
		\begin{align*}
			|k_{i+1}-k_i|^p |A_{i+1,\nicefrac{r}{4 c_1^2}}|
			&\les |k_{i+1}-k_i|^p  |A_{i,\nicefrac{r}{4 c_1}} \setminus A_{i+1,\nicefrac{r}{4 c_1}}|^{\gamma(p)}.
		\end{align*}
		Finally, we cancel the term $|k_{i+1} -k_i|^p$, raise both sides to the $\nicefrac{1}{\gamma(p)}$-th power, sum from $i=1$ to $n$ and bound the left from below by $n |A_{n+1,\nicefrac{r}{4 c_1^2}}|^{\nicefrac{1}{\gamma(p)}}$ in order to get
		\begin{align*}
			n |A_{n+1,\nicefrac{r}{4 c_1^2}}|^{\gamma(p)^{-1}} \lesssim |O(x, \nicefrac{r}{4c_1})|.
		\end{align*}

		\textbf{Proof of (ii).} We perform the same proof as in (i), using \eqref{eq: 2-isoperimatric near Dirichlet} instead of \eqref{eq: 2-isoperimetric near Neumann} in \eqref{eq: Shrinking Lemma: Use of 2-iso}. Here, we also use the assumption $k_0(u) \geq 0$.
	\end{proof}

	Finally, we come to the

	\begin{proof}[\rm\bf{Proof of Theorem~\ref{Real valued A: Theorem: Property H(mu)}.}]

		We can assume that $\osc_{x, u}(\nicefrac{R}{2}) > 0$, since otherwise $u$ is equal to a constant almost everywhere and there is nothing to prove. For $\osc_{x, u}(\nicefrac{R}{2})$ we divide the proof into two parts.

		\textbf{(1) $\boldmath \nicefrac{R}{4 c_1} < \d_D(x)$.} We define $k_n = k_n(u)$ as in \eqref{eq: Real-valued A: Shrinking lemma: Def of k_i} with $r = R$ and write
		\begin{equation*}
			O(x,\nicefrac{R}{4 c_1^2}) \supseteq \{ u > k_0(u) \} (x,\nicefrac{R}{4 c_1^2}) \cup \{ -u > k_0(-u) \}(x,\nicefrac{R}{4 c_1^2}).
		\end{equation*}
		At least one of the disjoint sets on the right has measure at most $\frac{1}{2}|O(x,\nicefrac{R}{4 c_1^2})|$, say the first one, because otherwise we work with $- u$. Theorem~\ref{Real valued A: Theorem: Local boundedness} implies that
		\begin{equation*}
			M_{x,u}(\nicefrac{R}{8 c_1^2}) \leq k_n + c(M_{x,u}(\nicefrac{R}{2}) - k_n) \left( \frac{|\{ u \geq k_n \} (x, \nicefrac{R}{4 c_1^2})|}{R^d} \right)^{\frac{\theta-1}{2}},
		\end{equation*}
		with $\theta$ as in \eqref{eq: Real-valued A: Definition of theta}. Since $|\{ u > k_0(u) \}(x, \nicefrac{R}{4 c_1^2})| \leq \frac{1}{2} |O(x, \nicefrac{R}{4 c_1^2})|$, Lemma~\ref{Real valued A: Lemma: ShrinkingLemma} (i) yields that there is some $n \in \N$ depending only on~\hyperref[Section: Real-valued A]{$\AssE$},~\hyperref[Section: Real-valued A]{$\AssPp$} and the implicit constant in \eqref{eq: De Giorgi class estimate} for $u$ such that
		\begin{equation}
			c \left( \frac{|\{ u \geq k_n \} (x, \nicefrac{R}{4 c_1^2})|}{R^d} \right)^{\frac{\theta -1}{2}} < \frac{1}{2}. \label{eq: Real valued A: Final proof: Make term small}
		\end{equation}
		This implies that
		\begin{align*}
			M_{x,u}(\nicefrac{R}{8 c_1^2}) &\leq M_{x,u}(\nicefrac{R}{2}) - 2^{-(n+2)} \osc_{x,u}(\nicefrac{R}{2}).
		\end{align*}
		Let $\sigma \coloneqq  1 -2^{-(n+2)}$. We subtract $m_{x,u}(\nicefrac{R}{8 c_1^2})$ and use $m_{x,u}(\nicefrac{R}{2}) \leq m_{x,u}(\nicefrac{R}{8 c_1^2})$ to obtain
		\begin{equation*}
			\osc_{x,u}(\nicefrac{R}{8 c_1^2}) \leq \sigma \osc_{x,u}(\nicefrac{R}{2}).
		\end{equation*}
		Now, let $0 < r \leq \rho \leq \nicefrac{R}{2}$ and fix $k \in \N$ with $r \in ((4 c_1^2)^{-k} \rho, (4 c_1^2)^{-k+1} \rho]$. The latter estimate delivers with $\mu \coloneqq \nicefrac{-\ln(\sigma)}{\ln(4 c_1^2)} \in (0,1)$ that
		\begin{equation}
			\label{eq: Real valued A: Final proof: osc-bound}
			\osc_{x,u}(r) \leq \osc_{x,u} ((4 c_1^2)^{-k+1} \rho) \leq \sigma^{k-1} \osc_{x,u} (\rho) \leq \sigma^{-1} \left( \frac{r}{\rho} \right)^{\mu} \osc_{x,u}(\rho),
		\end{equation}
		where we have used that $\sigma^{k} = ((4 c_1^2 )^{-k})^{\mu} \leq (\nicefrac{r}{\rho})^{\mu}$.

		Next, let $y, z \in O(x, \nicefrac{R}{4})$. If $|y-z| > \nicefrac{R}{8}$, then
		\begin{align*}
			\frac{|u(y)-u(z)|}{|y-z|^\mu}
			\leq \frac{2 M_{x,|u|}(\nicefrac{R}{4})}{(\nicefrac{R}{8})^\mu}
			\les R^{-\mu - \frac{d}{2}} \|u\|_{\L^2(O(x, \nicefrac{R}{4}))},
		\end{align*}
		where the final step is due to Theorem~\ref{Real valued A: Theorem: Local boundedness} with $k=0$. Now, let $|y-z| \leq \nicefrac{R}{8}$. Then $O(y, \nicefrac{R}{8}) \sub O(x, \nicefrac{R}{2})$. Hence, \eqref{eq: Real valued A: Final proof: osc-bound} gives
		\begin{align*}
			|u(y) - u(z)| &\leq \osc_{y, u} (|y-z|) \les \left( \frac{|y-z|}{\nicefrac{R}{8}} \right)^{\mu} \osc_{y, u} \left( \nicefrac{R}{8} \right) \les \left( \frac{|y-z|}{R} \right)^{\mu} 2 M_{x, |u|}(\nicefrac{R}{2})
		\end{align*}
		and we conclude by Theorem~\ref{Real valued A: Theorem: Local boundedness} as before.

		\textbf{(2) $\boldmath \d_D(x) \leq \nicefrac{R}{4 c_1}$.} Since $k_0(-u) = - k_0(u)$ we can assume that $k_0(u) \geq 0$. Now, we can argue as before, using Lemma~\ref{Real valued A: Lemma: ShrinkingLemma} (ii) instead of (i) in \eqref{eq: Real valued A: Final proof: Make term small} and consequently replacing $c_1^2$ by $c_1$ everywhere.
	\end{proof}


	\section{Property $\rD(\mu)$ for complex perturbations} \label{Section: Property D(mu)}

	Property $\rD(\mu)$ is stable under small complex perturbations. This observation is due to Auscher when $O = \R^d$ \cite[Thm.~4.4]{Auscher_Heat-Kernel} and one of the main reasons to study Gaussian bounds through property $\rD(\mu)$. Indeed, the former is much harder to perturb.

	\begin{theorem} \label{Perturbation of property D(mu): Theorem}
		Assume~\hyperref[The geometric setup: Definition (Fat)]{$\Fat$} and~\hyperref[The geometric setup: Definition (LU)]{$\LU$}. Let $A, A_0 \in \L^{\infty}(O; \C^{d \times d})$ be uniformly strongly elliptic such that $A_0$ has ellipticity constant $\lambda > 0$ and $L_0 \coloneqq - \Div (A_0 \nabla \cdot)$ has property $\rD(\mu)$. Then for all $\nu \in (0, \mu)$ there is some $\varepsilon = \varepsilon(c_{\rD(\mu)}, \lambda, d, \mu, \nu) > 0$ such that if $\| A - A_0 \|_{\infty} < \varepsilon$, then $L = - \Div (A \nabla \cdot)$ has property $\rD(\nu)$.
	\end{theorem}

	In view of Theorems~\ref{Theorem: Main result of the paper} and \ref{Theorem: Main result 2} we record:

	\begin{corollary}
		Under the geometric assumptions~\hyperref[The geometric setup: Definition (Fat)]{$\Fat$} and~\hyperref[The geometric setup: Definition (LU)]{$\LU$} there is some $\varepsilon > 0$ depending on geometry and ellipticity such that if $\| \im(A) \|_{\infty} < \varepsilon$, then $L$ has property $\rD(\mu)$, $\rG(\mu)$ and $\H(\mu)$ for some $\mu \in (0,1]$.
	\end{corollary}

	\begin{proof}
		We adapt the argument in \cite[Thm.~4.4]{Auscher_Heat-Kernel}. Let $x \in \overline{O}$, $0 < r \leq \nicefrac{R}{4} \leq \nicefrac{r_0}{4}$, $u \in \H_D^1(O)$ with $L_Du = 0$ in $O(x,R)$ and define the function
		\begin{equation*}
			\phi(r) \coloneqq \| \nabla u \|_{\L^2(O(x,r))}.
		\end{equation*}
		By Lemma~\ref{D(mu), G(mu) and H(mu) for MBC: Lemma: Existence of weak solutions + A prior bound} we find $\rho \in [\nicefrac{R}{4}, R]$ and $v \in \H_{\partial O(x,\rho) \setminus N(x, \rho)}^1(O(x,\rho))$ such that
		\begin{equation*}
			L_{0, D} v = - \Div (A_0 \nabla u) \quad \text{in} \; O(x, \rho).
		\end{equation*}
		Lemma~\ref{D(mu), G(mu) and H(mu) for MBC: Lemma: Sobolev extension by 0} (i) allows us to extend $v$ by $0$ to an element of $\H^1_D(O)$. Hence, $w \coloneqq u -v \in \H_D^1(O)$ and $L_{0, D} w =0$ in $O(x, \rho)$. Since $L_0$ has property $\rD(\mu)$, we get
		\begin{align}
			\phi(r) &\leq \| \nabla v \|_{\L^2(O(x,r))} + \| \nabla w \|_{\L^2(O(x,r))}  \notag
			\\&\les \| \nabla v \|_{\L^2(O(x,r))} + \left( \frac{r}{\rho} \right)^{\frac{d}{2}-1+\mu}  \| \nabla w \|_{\L^2(O(x,\rho))}  \notag
			\\&\les \| \nabla v \|_{\L^2(O(x,\rho))} + \left( \frac{r}{R} \right)^{\frac{d}{2} -1 + \mu} \phi(R). \label{eq: Perturbarion of D(mu)}
		\end{align}
		Next, we use $v$ as a test function in $L_{0, D} v = - \Div (A_0 \nabla u)$ and $L_D u =0$ in $O(x,\rho)$ to obtain
		\begin{align*}
			\int_{O(x,\rho)} A_0 \nabla v \cdot \overline{\nabla v} = \int_{O(x,\rho)} A_0 \nabla u \cdot \overline{\nabla v} =  \int_{O(x,\rho)} (A_0-A) \nabla u \cdot \overline{\nabla v}.
		\end{align*}
		Thus, ellipticity and Cauchy--Schwarz yield
		\begin{align*}
			\int_{O(x, \rho)} |\nabla v|^2 \les \| A - A_0 \|_{\infty} \phi(R) \| \nabla v \|_{\L^2(O(x,\rho))}.
		\end{align*}
		We divide by $\| \nabla v \|_{\L^2(O(x,\rho))}$ and insert the resulting estimate back into \eqref{eq: Perturbarion of D(mu)} to get
		\begin{equation*}
			\phi(r) \les \left( \| A -A_0 \|_{\infty}+ \left( \frac{r}{R} \right)^{\frac{d}{2} -1 + \mu} \right) \phi(R).
		\end{equation*}
		Lemma~\ref{Appendix: Lemma: Campanato} yields for each $\nu \in (0, \mu)$ the claim
		\begin{equation*}
			\phi(r) \les \left( \frac{r}{R} \right)^{\frac{d}{2}-1+ \nu} \phi(R),
		\end{equation*}
		provided that $\| A- A_0 \|_{\infty}$ is small enough (depending only on $[c_{\rD(\mu)}, \lambda, d, \mu, \nu]$).
	\end{proof}


	\section{Property $\rG(\mu)$ in dimension $d =2$}\label{Section: Property G(mu) for d=2}

	In this section we prove that in dimension $d =2$ \textit{every} elliptic operator $L$ has property $\rG(\mu)$ for some $\mu \in (0,1]$. On $O = \R^d$, this is due to \cite{AMcIT-Heat-kernel-d=2}. In doing so, we need to assume that $D$ is a \textbf{$\boldsymbol{(d-1)}$-set}:
	\begin{equation*}
		\exists \, c > 0 \; \forall \, x \in D, r \leq 1 \colon \quad c r^{d-1} \leq \cH^{d-1}(D(x,r)) \leq c^{-1} r^{d-1}.  \tag*{\text{$(\rD)$}}
	\end{equation*}
	This geometric requirement implies that $D$ is locally $2$-fat, see Lemma~\ref{App: Fat and Thick: Lem: Thickness and fatness} for an explicit proof. As $D \sub O^c$, also~\hyperref[The geometric setup: Definition (Fat)]{$\Fat$} is satisfied. We need (D) to apply an extrapolation result from \cite{Bechtel_Lp}.

	\begin{theorem} \label{Property G(mu) for d=2: Theorem}
		Let $d=2$ and assume~\hyperref[The geometric setup: Definition (LU)]{$\LU$} and $\mathrm{(D)}$. Then $L$ has property $\rG(\mu)$ for some $\mu \in (0,1)$ depending on geometry and ellipticity.
	\end{theorem}

	\begin{proof}
		We are going to use the following two properties of the operator $L$ from \cite{Bechtel_Lp} that follow from~\hyperref[The geometric setup: Definition (LU)]{$\LU$} and (D): we use the extrapolation result from \cite[Prop.~7.1]{Bechtel_Lp} with $p=2$ to find some $q \in (2, \infty)$ such that the restriction of the Lax--Milgram isomorphism
		\begin{equation*}
			1 + \cL \colon \W^{1,2}_D(O) \to \W^{1,2}_D(O)^*, \quad u \mapsto (u \, | \, \cdot )_2 + a(u,\cdot )
		\end{equation*}
		to $\W^{1,q}_D(O) \cap \W^{1,2}_D(O)$ extends to an isomorphism from $\W^{1,q}_D(O)$ to $\W^{1,q'}_D(O)^*$ with inverse that coincides with $(1+ \cL)^{-1}$ on $\W^{1,q'}_D(O)^* \cap \W^{1, 2}_D(O)^*$. By  \cite[Cor.~3.5 \& Prop.~3.6]{Bechtel_Lp} we have the estimate
		\begin{equation*}
			\| t (1 +L) \e^{-t (1 + L)} u \|_{q_*} \les \| u \|_{q_*} \qquad (t > 0, u \in \L^{q_*}(O) \cap \L^2(O)).
		\end{equation*}

		Now, we can start the actual proof. Set $\mu \coloneqq 1 - \nicefrac{d}{q} \in (0,1)$. Due to~\hyperref[The geometric setup: Theorem: Extension operator]{$\AssExt$} we have the Sobolev embeddings $\W_D^{1,q}(O) \sub \mathrm{C}^{\mu}(O)$ and $\W^{1, q'}_D(O) \sub \L^{(q_{*})'}(O)$. By duality, the second embedding entails that $\L^{q_*}(O) \sub \W^{1,q'}_D(O)^*$. Using these embeddings and the fact that $(1+ \cL)^{-1}$ maps $\W^{1,q'}_D(O)^* \cap \W^{1,2}(O)^*$ boundedly into $\W_D^{1,q}(O)$ for the $q$-norms, we conclude that
		\begin{equation*}
			[(1 + \cL)^{-1} u]^{(\mu)}_O + \| (1 + \cL )^{-1} u \|_{\infty} \les \| u \|_{q_*} \qquad (u \in \L^{q_*}(O) \cap \L^2(O)).
		\end{equation*}
		Let $t > 0$ and $u \in \L^{q_*}(O) \cap \L^2(O)$. Using $(1 + \cL)|_{\L^2(O)} = 1+ L$, we get
		\begin{align*}
			[\e^{- t L}u]^{(\mu)}_O
			=  [ (1 + \cL)^{-1} (1 + L) \e^t \e^{- t (1+ L)}u ]^{(\mu)}_O
			&\les \e^t \| (1 + L) \e^{-t (1 + L)}u \|_{q_*}  \\
			&\les \e^t t^{-\frac{\mu}{2} - \frac{d}{2 q_*}} \| u \| _{q_*}.
		\end{align*}
		Replacing $[\, \cdot \, ]_O^{(\mu)}$ by $\| \cdot \|_{\infty}$ and setting $t = \nicefrac{\delta^2}{16}$ in the latter estimates, we deduce in the same manner
		\begin{equation*}
			\| \e^{- \frac{\delta^2}{16} L} u \|_{\infty} \les \| u \|_{q_*}.
		\end{equation*}
		The last two estimates allow us to repeat the proof of Theorem~\ref{(1) to (2): Main result} with $\L^2$ systematically replaced by $\L^{q_*}$, see also Remark~\ref{D(mu) to G(mu): Remark: Linfty bound only for fixed t}. The outcome is property (i) of Corollary~\ref{D(mu) to G(mu): Corollary: G(mu) equivalent to L2-Linfty-ODE + L2-Holder} and hence, $L$ has property $\rG(\mu)$.
	\end{proof}


	\appendix

	\section{Remarks on capacities} \label{Section: Appendix}

	For the reader's convenience, we include some results related to capacities that we could not find in the literature. For brevity we write $\L^p = \L^p(\R^d)$ and so on. Capacities and $p$-fatness have been introduced in Section~\ref{Subsection: Assumptions Fat and LU}.

	\begin{remark} \label{Appendix: Remark: p-fatness}
		We can change the parameters in Definition~\ref{eq: Definition: p-fatness} of local $p$-fatness.
		\begin{enumerate}
			\item It is possible to replace $2 B \coloneqq B(x,2r)$ by $\kappa B $ for each $\kappa > 1$ in \eqref{eq: Definition: p-fatness}. The interesting direction is when fatness is formulated with a reference ball $\kappa B$ and we want to switch to a larger radius called $2 > \kappa$ for simplicity. We pick $u \in \smooth[2 B]$ such that $u \geq 1$ on $\overline{B} \cap C$. Let $\eta \in \rC_{\cc}^{\infty}(\kappa B)$ be $[0,1]$-valued with $\eta = 1$ on $\overline{B}$ and $\| \nabla \eta \|_{\L^{\infty}} \les_{\kappa} r^{-1}$. Poincar\'{e}'s inequality yields
			\begin{equation*}
				\Capp_p (\overline{B} \cap C; \kappa B) \leq \| \nabla (u \eta ) \|_{\L^p(\kappa B)}^p \les_{\kappa, d,p} \| \nabla u \|_{\L^p(2B)}^p,
			\end{equation*}
			and thus
			\begin{equation*}
				\Capp_p (\overline{B} \cap C; \kappa B) \les \Capp_p(\overline{B} \cap C; 2B).
			\end{equation*}
			\item We can replace the condition $r \leq 1$ by $r \leq r_0$ for any $r_0 > 0$ in \eqref{eq: Definition: p-fatness}. This follows from the first remark and the monotonicity of capacities in the first argument.

			\item We can replace the requirement $u \geq 1$ on $\overline{B} \cap C$ by $u =1$ in an open neighborhood of $\overline{B} \cap C$ and $0 \leq u \leq 1$ everywhere. Indeed, for the interesting direction we pick $u \in \smooth[2B; \R]$ with $u \geq 1$ on $\overline{B} \cap C$. Let $\varepsilon \in (0,1)$ and put $v_{\varepsilon} \coloneqq ( (1- \varepsilon)^{-1} u \land 1) \lor 0$. Note that $v_{\varepsilon}$ is continuous with
			\begin{align*}
				v_{\varepsilon} = \begin{cases} 1 \qquad \qquad \quad  &(u \geq 1 - \varepsilon), \\ (1- \varepsilon)^{-1} u &(0 \leq u \leq 1 - \varepsilon), \\ 0  &(u \leq 0). \end{cases}
			\end{align*}
			In particular, as $\overline{B} \cap C$ is compact and $u$ is continuous, there is some $\delta > 0$ such that $v_{\varepsilon} =1$ on $(\overline{B} \cap C)_{\delta}$.
			We set $v_n \coloneqq \eta_n * v_{\varepsilon}$ for all $n \in \N$, where $\eta_n(x) = n^d \eta(n x)$ is a standard mollifier. Note that $v_n$ is smooth, $[0,1]$-valued and there is some $N \in \N$ with $v_n =1$ on $(\overline{B} \cap C)_{\nicefrac{\delta}{2}}$ and $v_n$ has compact support in $2B$ for all $n \geq N$. Next, using Young's inequality for convolutions and $\nabla v_{\varepsilon} = (1- \varepsilon)^{-1} \1_{[0 \leq u \leq 1 - \varepsilon]} \nabla u$, we get
			\begin{equation*}
				\| \nabla v_n \|_{\L^p(2B)}^p \leq \| \nabla v_{\varepsilon} \|_{\L^p(2B)}^p \leq (1- \varepsilon)^{-p} \| \nabla u \|_{\L^p(2B)}^p.
			\end{equation*}
			Hence, we derive
			\begin{align*}
				\inf_w \| \nabla w \|_{\L^p(2B)}^p \leq (1- \varepsilon)^{-p} \| \nabla u \|_{\L^p(2 B)}^p,
			\end{align*}
			where the infimum is taken over all $[0,1]$-valued $w \in \smooth[2 B]$ that are $1$ in an open neighborhood of $\overline{B} \cap C$. We conclude by letting $\varepsilon \to 0$.
		\end{enumerate}
	\end{remark}

	In Lewis' self-improvement \cite{Lewis_P-fat_Open_ended} the notion of Riesz capacities is used in the main result. We include a proof that this gives rise to an equivalent notion of $p$-fatness when $p \in (1,d)$. Note that this does not cover the case $p=2=d$. This is why we preferred to rely on the self-improvement for our (variational) capacity from \cite{Mikkonen}.

	Let $I_1(x) \coloneqq |x|^{1-d}$ be the Riesz potential of order $1$. Given $f \in \L^p$ with $f \geq 0$ a.e.\@, the convolution
	\begin{equation*}
		(I_1 * f )(x) = \int_{\R^d} |x-y|^{1-d} f(y) \, \d y
	\end{equation*}
	is defined for all $x \in \R^d$ as a number in $[0,\infty]$. Using classical results for the Riesz transform \cite[Chap.~III.1]{Stein II}, we obtain for all non-negative $f \in \rC_{\cc}^{\infty}$ that
	\begin{equation} \label{eq: App: Mihlin-Lp-estimate}
		\| \nabla (I_1 * f) \|_{\L^p} \les_{d,p} \| f \|_{\L^p}.
	\end{equation}

	\begin{definition}
		For $p \in (1,d)$ and closed $C \sub \R^d$ the \textbf{Riesz capacity of $\boldsymbol{C}$} is defined by
		\begin{equation*}
			R_{1,p}(C) \coloneqq \inf \left\{ \| f \|_{\L^p}^p \colon f \in \L^p \; \text{with} \; f \geq 0 \; \text{a.e.\@ and} \; I_1 * f \geq 1 \; \text{on} \; C \right\}.
		\end{equation*}
	\end{definition}

	Let us explicitly prove that Riesz and variational capacity give rise to the same notion of $p$-fatness when $p \in (1,d)$.

	\begin{lemma} \label{Appendix: Lemma: Riesz capacity, relative capacity}
		Let $p \in (1, d)$ and $C \sub \R^d$ be closed. Then $C$ is locally $p$-fat if and only if
		\begin{equation}
			\exists \, c > 0 \; \forall \, r \leq 1, x \in C \colon \quad R_{1,p}(\overline{B(x,r)} \cap C) \geq c r^{d-p}.  \label{eq: Appendix: Riesz-capacity lower bound}
		\end{equation}
	\end{lemma}

	\begin{proof}
		First, assume that $C$ is locally $p$-fat with lower bound $c > 0$ and abbreviate $B \coloneqq B(x,r)$. Let $f \in \L^p$ with $f \geq 0$ a.e.~and $u \coloneqq I_1 * f \geq 1$ on $\overline{B} \cap C$. Let $\varepsilon \in (0,1)$ and set $f_{\varepsilon} \coloneqq (1-\varepsilon)^{-1} f$. We define for $n \in \N$ the non-negative functions $f_n \coloneqq (\1_{B(0,n)} f_{\varepsilon}) * \eta_n \in \rC_{\cc}^{\infty}$, where $\eta_n(x) = n^d \eta(n x)$ is again a standard mollifier. Put $u_n \coloneqq I_1 * f_n$. We claim that there is some $N \in \N$ such that $u_n \geq 1$ on $\overline{B} \cap C$ for all $n \geq N$.

		Indeed, assume that this is not the case. Then there is a strictly increasing sequence $(n_j)_j \sub \N$ and a sequence $(x_j)_j \sub \overline{B} \cap C$ such that $u_{n_j}(x_j) < 1$ for all $j \in \N$. Since $\overline{B} \cap C$ is compact, there is some $x \in \overline{B} \cap C$ such that $x_j \to x$ along a subsequence. Using Fatou's lemma and passing to a further subsequence to guarantee $f_{n_j} \to f_{\varepsilon}$ a.e., we get $(I_1 * f_{\varepsilon}) (x) \leq 1$. Thus, $u(x) \leq 1 - \varepsilon < 1$, which is a contradiction to our assumption.

		From now on we let $n \geq N$. If $\nicefrac{c r^{\nicefrac{d}{p}} }{4} \leq \| u_n \|_{\L^p(2B)}$, then we obtain by H\"older's inequality and a Sobolev embedding that
		\begin{equation*}
			r^{\frac{d}{p}-1} \les r^{-1} \| u_n \|_{\L^p(2B)} \les \| u_n \|_{\L^{p^*}} \les \| \nabla u_n \|_{\L^p}.
		\end{equation*}
		Now, consider the case $\| u_n \|_{\L^p(2B)} < \nicefrac{c r^{\nicefrac{d}{p}} }{4}$. Let $\varphi \in \smooth[2 B]$ be $[0,1]$-valued with $\varphi =1$ on $\overline{B}$ and $\| \nabla \varphi \|_{\L^{\infty}} \leq \nicefrac{2}{r}$. Using that $C$ is locally $p$-fat, we estimate
		\begin{equation*}
			c r^{\frac{d}{p}-1} \leq \| \nabla (u_n \varphi ) \|_{\L^{p}(2 B)} \leq \| \nabla u_n \|_{\L^p(2 B)} + 2 r^{-1} \| u_n \|_{\L^p(2 B)} \leq \| \nabla u_n \|_{\L^p} + \frac{c}{2} r^{\frac{d}{p}-1}
		\end{equation*}
		and subtract the second summand on the right to obtain the same lower bound for $\| \nabla u_n \|_{\L^p}$ as before. Using the very definition of $u_n$ and \eqref{eq: App: Mihlin-Lp-estimate}, the upshot is that in either case we have
		\begin{equation*}
			r^{d-p} \les \| \nabla u_n \|_{\L^p}^p \les \| f_n \|_{\L^p}^p \leq (1-\varepsilon)^{-p} \| f \|_{\L^p}^p.
		\end{equation*}
		Thus, \eqref{eq: Appendix: Riesz-capacity lower bound} follows as $\varepsilon \to 0$.

		Now, suppose that \eqref{eq: Appendix: Riesz-capacity lower bound} holds true. To prove that $C$ is locally $p$-fat, pick $u \in \smooth[2B; \R]$ with $u \geq 1$ on $\overline{B} \cap C$. By \cite[p.~125]{Stein II} we find $f \in \L^p$ such that $u = I_1 * f$ and $\| f \|_{\L^p} \les \| \nabla u \|_{\L^p}$. In particular, $I_1 * f^+ \geq 1$ on $\overline{B} \cap C$. Hence, \eqref{eq: Appendix: Riesz-capacity lower bound} yields
		\begin{equation*}
			r^{d-p} \les \| f^+ \|_{\L^p}^p \leq \| f \|_{\L^p}^p \les \| \nabla u \|_{\L^p}^p = \| \nabla u \|_{\L^p(2 B)}^p.
		\end{equation*}
		This completes the proof.
	\end{proof}

\section{Relation between fatness and thickness} \label{Section: Appendix, Fatness and thickness}

For the reader's convenience, we also include a proof of the following result that compares $p$-fatness and thickness relative to the Hausdorff measure.

\begin{lemma} \label{App: Fat and Thick: Lem: Thickness and fatness}
	Let $\delta >0$, $p \in (1, d]$ and $s \in (d-p,d]$. Let $C \sub \R^d$ be closed and $\widehat{C} \sub \R^d$. If $C$ satisfies
	\begin{equation*}
		\cH^s(C(x,r)) \simeq r^s \qquad (r \leq \delta, x \in C \cap \widehat{C}_{2 \delta} ),
	\end{equation*}
	then $C$ is locally $p$-fat in $C \cap \widehat{C}_{\delta}$.
\end{lemma}

We will use the notion of $s$-dimensional Hausdorff content, denoted by $\cH_{\infty}^s$, and refer to \cite[Chap.~7]{Yeh_Real-Analysis} for further background.

\begin{proof}
	In view of \cite[Thm.~3.1]{Martio-Thickness_implies_fatness} applied with $h(r) = r^s$, it suffices to prove the lower bound
	\begin{equation}  \label{eq: Comparison of the geometric setup: D is (d-1)-thick}
		\cH_{\infty}^s(C(x,r)) \gtrsim r^s \qquad (r \leq \delta, x \in C \cap \widehat{C}_{\delta}).
	\end{equation}
	We adapt an argument in \cite[Lem.~6.6]{Auscher_Badr_Haller_Rehberg} to our needs.

	Fix $x \in C \cap \widehat{C}_{\delta}$, $r \leq \delta$ and let $\{ B_n \}_n = \{ B(x_n, r_n) \}_n$ be a covering of $C(x,r)$ with open balls centered in $C(x,r)$. The assumption yields
	\begin{equation*}
		r^s \les \cH^s (C(x,r)) = \cH^s \Big( C(x,r) \cap \bigcup_n B_n \Big) \leq \sum_n \cH^s (C(x,r) \cap B_n).
	\end{equation*}
	If $r_n > \delta$, then we use the other part of the assumption in the form$$\cH^s(C(x,r) \cap B_n) \leq \cH^s(C(x,r)) \les r^s \leq r_n^s.$$If $r_n \leq \delta$, then we note that $C(x,r) \cap B_n \sub C(x_n, r_n)$ and $x_n \in C \cap \widehat{C}_{2 \delta}$ in order to deduce the same bound $\cH^s(C(x,r) \cap B_n) \les r_n^s$. In total, we obtain
	\begin{equation*}
		r^s \les \sum_n r_n^s.
	\end{equation*}
	By definition, $\cH_{\infty}^s(C(x,r))$ is the infimum over all expressions as on the right-hand side. Thus, \eqref{eq: Comparison of the geometric setup: D is (d-1)-thick} follows.
\end{proof}


\section*{Conflicts of interest statement and Data availability}

Data sharing not applicable to this article as no datasets were generated or analysed during the current study.

All authors declare that they have no conflict of interest.


\end{document}